\documentclass[11pt]{article}
\usepackage{amssymb, amsmath, amsthm, oldgerm,srcltx}
\usepackage{graphicx}
\usepackage{xcolor}
\newcommand{\tc}{\textcolor{red}}

\newtheorem{thm}{Theorem}[section]
\newtheorem{lem}{Lemma}[section]
\newtheorem{cor}[lem]{Corollary}
\newtheorem{prop}[thm]{Proposition}
\newtheorem{rem}[thm]{Remark}
\numberwithin{equation}{section}

\newcommand{\abs}[1]{\left\vert#1\right\vert}

\newcommand{\E}{\mathbf{E}\,}

\newcommand{\R}{\mathbf{R}}

\newcommand{\re}{\mathrm{Re}\;\!}
\newcommand{\im}{\mathrm{Im}\;\!}
\newcommand{\Tr}{\mathrm{Tr}\;\!}
\newcommand{\q}{\quad}
\newcommand{\lln}{\operatorname{llog}_n}
\newcommand{\lna}{l_{n,\alpha}}

\newcommand{\lnak}{\lna^{\frac1{\varkappa}}}

\newenvironment{Proof of}{\removelastskip\par\medskip
\noindent{\em Proof of} \rm}{\penalty-20\null\hfill$\square$\par\medbreak}

\def\be{\begin{equation}}
\def\en{\end{equation}}
\def\bee{\begin{eqnarray*}}
\def\ene{\end{eqnarray*}}

\def\E{{\bf E}}

\def\R{{\mathbb R}}

\def\Tr{{\rm Tr}\,}

\def\<{\left<}
\def\>{\right>}

\def\1{{\bf 1}}
\def\4{\kern1pt}

\begin{document}
\bibliographystyle{}

\vspace{1in}
 \date{}
\title{\bf On the Rate of Convergence to the Semi-Circular Law}

\author{{\bf F. G\"otze}\\{\small Faculty of Mathematics}
\\{\small University of Bielefeld}\\{\small Germany}
\and {\bf A. Tikhomirov}$^{1}$\\{\small Department of Mathematics
}\\{\small Komi Research Center of Ural Branch of RAS,}\\{\small Syktyvkar State University}
\\{\small  Russia}
}
\maketitle
 \footnote{$^1$Research supported   by SFB 701 ``Spectral Structures and Topological Methods in Mathematics'' University of Bielefeld.
Research supported   by grants  RFBR N~11-01-00310a and by Program of Fundamental Research Ural Division of RAS, Project ¹ 12-P-1-1013}

\maketitle

\date{}

\begin{abstract}
Let $\mathbf X=(X_{jk})_{j,k=1}^n$ denote a  Hermitian random matrix with
entries $X_{jk}$, which are independent for $1\le j\le  k\le n$. We consider
the rate of convergence of the empirical  spectral distribution function of
the matrix $\mathbf X$ to the semi-circular law assuming that $\E X_{jk}=0$,
$\E X_{jk}^2=1$ and that the distributions of the  matrix elements $X_{jk}$ have a uniform
sub exponential decay in the sense that there exists a constant $\varkappa>0$ such that
for any $1\le j\le k\le n$ and any $t\ge 1$ we have
$$
\Pr\{|X_{jk}|>t\}\le A\exp\{-t^{\varkappa}\}.
$$

By means of a recursion argument it is shown that the
 Kolmogorov distance between the empirical spectral
distribution of the Wigner matrix $\mathbf W=\frac1{\sqrt n}\mathbf X$
and the  semicircular law is of order \newline$O(n^{-1}\log^b n)$ with some
positive constant $b>0$.
\end{abstract}
\maketitle


\section{Introduction}
\setcounter{equation}{0}

Consider a family $\mathbf X = \{X_{jk}\}$, $1 \leq j \leq k \leq n$,
of independent real random variables defined on some probability space
$(\Omega,{\textfrak M},\Pr)$, for any $n\ge 1$. Assume that $X_{jk} = X_{kj}$, for
$1 \leq k < j \leq n$, and introduce the symmetric matrices
\begin{displaymath}
 \mathbf W = \ \frac{1}{\sqrt{n}} \left(\begin{array}{cccc}
 X_{11} &  X_{12} & \cdots &  X_{1n} \\
 X_{21} & X_{22} & \cdots &  X_{2n} \\
\vdots & \vdots & \ddots & \vdots \\
 X_{n1} &  X_{n2} & \cdots &  X_{nn} \\
\end{array}
\right).
\end{displaymath}

The matrix $\mathbf W$ has a random spectrum $\{\lambda_1,\dots,\lambda_n\}$ and an
associated spectral distribution function
$\mathcal F_{n}(x) = \frac{1}{n}\ {\rm card}\,\{j \leq n: \lambda_j \leq
x\}, \quad x \in \R$.
Averaging over the random values $X_{ij}(\omega)$, define the expected
(non-random) empirical distribution functions
$ F_{n}(x) = \E\,\mathcal F_{n}(x)$.
Let $G(x)$ denote the semi-circular distribution function with density
$g(x)=G'(x)=\frac1{2\pi}\sqrt{4-x^2}\mathbb I_{[-2,2]}(x)$, where $\mathbb I_{[a,b]}(x)$
denotes the indicator--function of the interval $[a,b]$. We shall study the
rate of convergence of
 $\mathcal F_{n}(x)$ to the semi-circular law under the condition
\begin{equation}\label{exptails}
 \Pr\{|X_{jk}|>t\}\le A\exp\{-t^{\varkappa}\},
\end{equation}
 for some $\varkappa>0$, $A>0$ and for any $t\ge1$.
 The rate of convergence to the semi-circular law has been studied by several authors.
 We proved  in \cite{GT:03}  that the  Kolmogorov distance between $\mathcal F_n(x)$ and
 the distribution function
 $G(x)$,
 $\Delta_n^*:=\sup_x|\mathcal F_n(x)-G(x)|$ is of order
 $O_P(n^{-\frac12})$ (i.e. $n^{\frac12}\Delta_n^*$ is bounded in probability). Bai \cite{Bai:02} and Girko  \cite{Girko:02}
 showed that  $\Delta_n:=\sup_x| F_n(x)-G(x)|=O(n^{-\frac12})$.
 Bobkov, G\"otze and Tikhomirov \cite{BGT:08}  proved that $\Delta_n$ and $\E\Delta_n^*$
 have order $O(n^{-\frac23})$
assuming a Poincar\'e inequality for the
distribution of the  matrix elements. For the Gaussian Unitary Ensemble respectively for the Gaussian Orthogonal Ensemble, see  \cite{GT:05}
respectively  \cite{Tim:07},
it has been shown that $\Delta_n=O(n^{-1})$.
Denote by $\gamma_{n1}\le\ldots\le \gamma_{nn}$, the quantiles of $G$,
i.e.   $G(\gamma_{nj})=\frac jn$.
We introduce the notation
 $\lln :=\log\log n$.
Erd\"os, Yau and Yin \cite{ErdosYauYin:2010a} showed,  for  matrices with
elements $X_{jk}$  which have a uniformly
sub exponential decay, i.e. condition \eqref{exptails} holds,
 the following result
\begin{align} \label{yau}
\Pr\Bigl\{\, \exists\,\, j:\,|\lambda_j-\gamma_{nj}|\ge (\log n)^{C \lln }
\Big[\min\{(j,N-j+1)\Big]^{-\frac13}n^{-\frac23} \Bigr\} \quad\quad\quad\quad\notag\\ \le C\exp\{-(\log n)^{c \4 \lln}\},
\end{align}
for $n$ large enough.
It is straightforward to check that this bound implies that
\begin{equation}\label{rate}
\Pr\Bigl\{\sup_x|\mathcal F_n(x)-G(x)|\le Cn^{-1}(\log n)^{C \lln}\Bigr\}\ge 1- C\exp\{-(\log n)^{c \lln}\}.
\end{equation}
From the last inequality it is follows that $\E \Delta_n^*\le C\4 n^{-1}(\log n)^{C \4 \lln}$.
In this paper we derive   some improvement of the result \eqref{rate}
(reducing  the power of logarithm) using arguments
 similar to those used in \cite{ErdosYauYin:2010a}  and provide  a self-contained proof
 based on recursion  methods developed
in the papers of G\"otze and Tikhomirov \cite{GT:03}, \cite{GT:09} and \cite{T:09}.
It follows from the results of Gustavsson \cite{Gustavsson:2005} the best possible bound in the Gaussian case for
the rate of convergence in probability is $O(n^{-1}\sqrt{\log n})$.
For any positive constants $\alpha>0$  and $\varkappa>0$, define the quantities
\begin{equation}\label{beta}
 \lna :=\log n\4(\log \log n)^{\alpha}
\quad \text{and} \quad
 \beta_n:=(\lna)^{\frac1{\varkappa}+1}.
\end{equation}
The main result of this paper is the following
\begin{thm}\label{main} Let $\E X_{jk}=0$, $\E X_{jk}^2=1$.  Assume that
there exists a constant $\varkappa>0$ such that  inequality \eqref{exptails} holds, for any $1\le j\le k\le n$ and any $t\ge 1$.
Then, for  any positive $\alpha>0$ there exist  positive  constants $C$ and $c$ depending on
$\varkappa $, $A$ and $\alpha$ only
such that
\begin{equation} \label{kolmog}
\Pr\Big\{\sup_x|\mathcal F_n(x)-G(x)|>n^{-1}\beta_n^4\ln n\Big\}\le C\4 \exp\{-c \4 \lna \}.
\end{equation}

\end{thm}
\begin{rem}\label{localization}
 Inequality \eqref{kolmog} implies that
 \begin{align}
\Pr\Bigl\{\, \exists\,\,j\in[c\beta_n^4,n-c\beta_n^4]:\,|\lambda_j-\gamma_{nj}|\ge C\beta_n^4
\Big[\min\{j,N-j+1\}\Big]^{-\frac13}n^{-\frac23} \Bigr\} \notag\\ \le C\exp\{-(\log n)^{c \4 \lln}\}.
\end{align}
\end{rem}

For a proof of this remark see the Appendix, Subsection \ref{remark}.

The  Remark shows the rigidity  of eigenvalues except for 
 the edges $k\le C\beta_n^4$ or $k\ge n-C\beta_n^4$. (Compare with \eqref{yau}).

We apply the result of Theorem \ref{main} to study
the  eigenvectors of the matrix $\mathbf W$.
 Let $\mathbf u_j=(u_{j1},\ldots,u_{jn})^T$  be eigenvectors of the
matrix  $\mathbf W$ corresponding to the  eigenvalues $\lambda_j$, $j=1,\ldots,n$.
We prove the following result.
\begin{thm}\label{eigenvector} Under the conditions  of Theorem \ref{main}, for any positive $\alpha>0$,
  there exist positive  constants $C$ and $c$, depending on $\varkappa $ and $\alpha$ only
such that
\begin{equation}\label{deloc}
 \Pr\Bigl\{\max_{1\le j,k\le n}|u_{jk}|^2>\frac{\beta_n^2}n\Bigr\}\le C\exp\{- c\4 \lna\},
\end{equation}
and
\begin{equation}\label{deloc1}
\Pr\Bigl\{\max_{1\le k\le n}\left|\sum_{\nu=1}^k|u_{j\nu}|^2-\frac kn\right|>\frac{\beta_n^2}{\sqrt n}\Bigr\}
\le C\exp\{-c\4 \lna\}.
\end{equation}

\end{thm}

\section{Bounds for the  Kolmogorov distance between distribution functions  via  Stieltjes transforms}
To bound the error $\Delta_n^*$ we shall use an approach developed in previous work of the authors, see \cite{GT:03}.\\
We modify the bound of the  Kolmogorov distance between an arbitrary distribution function and the semi-circular distribution function
 via their Stieltjes transforms obtained in \cite{GT:03} Lemma 2.1. For $x\in[-2,2]$ define $\gamma(x):=2-|x|$.
 Given $\frac12>\varepsilon>0$ introduce the interval $\mathbb J_{\varepsilon}=\{x\in[-2,2]:\, \gamma(x)\ge\varepsilon\}$ and
$\mathbb J'_{\varepsilon}=\mathbb J_{\varepsilon/2}$.
For a distribution function $F$ denote by $S_F(z)$ its Stieltjes transform,
$$
S_F(z)=\int_{-\infty}^{\infty}\frac1{x-z}dF(x).
$$
\begin{prop}\label{smoothing}Let $v>0$ and $a>0$ and $\frac12>\varepsilon>0$ be positive numbers such that
\begin{equation} \label{constant}
 \frac1{\pi}\int_{|u|\le a}\frac1{u^2+1}du=\frac34=:\beta,
\end{equation}
and
\begin{equation}\label{avcond}
 2va\le \varepsilon^{\frac32}.
\end{equation}
If $G$ denotes the  distribution function of the standard semi-circular law, and $F$ is any distribution function,
 there exist some absolute constants $C_1$ and $C_2$ such that
\begin{align}
\Delta(F,G)&:= \sup_{x}|F(x)-G(x)|\notag\\&\le 2
\sup_{x\in\mathbb J'_{\varepsilon}}\big|\im\int_{-\infty}^x(S_F(u+i\frac v{\sqrt{\gamma}})-S_G(u+i\frac v{\sqrt{\gamma}}))du\big|+C_1v
+C_2\varepsilon^{\frac32}.  \notag
\end{align}

\end{prop}
\begin{rem}\label{rem2.2}
 For any $x\in\mathbb J_{\varepsilon}$ we have
$\gamma=\gamma(x)\ge\varepsilon$
and according to condition \eqref{avcond},
$\frac{av}{\sqrt\gamma}\le \frac{\varepsilon}2$.
\end{rem}

For a  proof of this Proposition see Subsection \ref{kolmdistance}
of the Appendix.

\begin{lem}\label{Cauchy}
 Under the conditions of Proposition \ref{smoothing}, for any $V>v$ and   $0<v\le \frac{\varepsilon^{3/2}}{2a}$
and $v'=v/\sqrt{\gamma}, \gamma= 2-|x|$, $x\in\mathbb J'_{\varepsilon}$
 as above, the following inequality holds
\begin{align}
 \sup_{x\in\mathbb J'_{\varepsilon}}&\left|\int_{-\infty}^x(\im(S_F(u+iv')-S_G(u+iv'))du\right|\notag\\&\le
\int_{-\infty}^{\infty}|S_F(u+iV)-S_G(u+iV)|du\notag\\&+
\sup_{x\in\mathbb J'_{\varepsilon}}\left|\int_{v'}^V\left(S_F(x+iu)-S_G(x+iu)\right)du\right|. \notag
\end{align}

\end{lem}
\begin{proof}Let $x\in \mathbb J'_{\varepsilon}$ be fixed. Let $\gamma=\gamma(x)$.
 Put $z=u+iv'$.   Since $v'=\frac v{\sqrt{\gamma}}\le \frac{\varepsilon}{2a}$, see \eqref{avcond},  we may assume without loss of
generality that $v'\le 4$
for  $x\in\mathbb J'_{\varepsilon}$.  Since the functions of $S_F(z)$ and $S_G(z)$ are analytic in the upper half-plane, it is enough to use Cauchy's theorem. We can write
 for $x\in\mathbb J'_{\varepsilon}$
\begin{equation} \notag
\int_{-\infty}^{x}\im(S_F(z)-S_G(z))du=\im\{\lim_{L\to\infty}\int_{-L}^x(S_F(u+iv')-S_G(u+iv'))du\}.
\end{equation}
By Cauchy's integral formula, we have
\begin{align}
 \int_{-L}^x(S_F(z)-S_G(z))du&=\int_{-L}^x(S_F(u+iV)-S_G(u+iV))du\notag\\&
+\int_{v'}^V(S_F(-L+iu)-S_G(-L+iu))du\notag\\&-\int_{v'}^V(S_F(x+iu)-S_G(x+iu))du. \notag
\end{align}
Denote by $\xi\text{ (resp. }\eta)$ a random variable with distribution function $F(x)$ (resp. $G(x)$). Then we have
\begin{equation} \notag
 |S_F(-L+iu)|=\left|\E\frac1{\xi+L-iu}\right|\le {v'}^{-1}\Pr\{|\xi|>L/2\}+\frac2L,
\end{equation}
for any $v'\le u\le V$.
Similarly,
\begin{equation} \notag
 |S_G(-L+iu)|\le {v'}^{-1}\Pr\{|\eta|>L/2\}+\frac2L.
\end{equation}
These inequalities imply that
\begin{equation} \notag
\left|\int_{v'}^V(S_F(-L+iu)-S_G(-L+iu))du\right|\to 0\quad\text{as}\quad L\to\infty,
\end{equation}
which completes the proof. 
\end{proof}
Combining the results of Proposition \ref{smoothing} and Lemma \ref{Cauchy}, we get
\begin{cor}\label{smoothing1}
 Under the conditions of Proposition \ref{smoothing} the following inequality holds
\begin{align}
 \Delta(F,G)&\le 2\int_{-\infty}^{\infty}|S_F(u+iV)-S_G(u+iV)|du+C_1v+C_2\varepsilon^{\frac32}\notag\\&
  + 2 \sup_{x\in\mathbb J'_{\varepsilon}}\int_{v'}^V|S_F(x+iu)-S_G(x+iu)|du,\notag
\end{align}
where $v'=\frac v{\sqrt{\gamma}}$ with $\gamma=2-|x|$ and $C_1,C_2 >0$ denote absolute constants.

\end{cor}

 We shall apply  the last inequality. We denote the Stieltjes transform of $\mathcal F_n(x)$
by $m_n(z)$ and the Stieltjes transform of the semi-circular law by $s(z)$. Let
$\mathbf R=\mathbf R(z)$ be the resolvent matrix of $\mathbf W$ given by
$\mathbf R=(\mathbf W-z\mathbf I_n)^{-1}$,
for all $z=u+iv$ with $v\ne 0$. Here and in what follows $\mathbf I_n$ denotes the identity matrix of dimension $n$.
Sometimes we shall omit the sub index in the notation of an identity  matrix.
It is well-known that the Stieltjes transform of the semi-circular distribution
satisfies the equation
\begin{equation}\label{stsemi}
s^2(z)+zs(z)+1=0
\end{equation}
(see, for example, equality (4.20) in \cite{GT:03}).
Furthermore, the
Stieltjes transform of an empirical  spectral distribution function $\mathcal F_n(x)$, say $m_n(z)$, is given by
\begin{equation}\label{trace}
m_n(z)=\frac1n\sum_{j=1}^n
R_{jj}=\frac1{n}\Tr\mathbf  R.
\end{equation}
 (see, for instance, equality (4.3)
in \cite{GT:03}).
The following often used inequalities
for  $\mathbf R_{jk}(z), \, 1 \le j,k\le n$
and $v>0$ follow  via
spectral representations and  $\abs{\lambda_j-z}^{-1} \le v^{-1}$ as well as $\abs{\lambda_j-z}^{-2}= v^{-1} \im (\lambda_j-z)^{-1}$,
\begin{equation}\label{l2res}
\abs{\mathbf R_{j,k}} \le v^{-1}, \q n^{-1}\sum_{j,k}\abs{\mathbf R_{j,k}}^2 \le v^{-1} m_n(z).
\end{equation}
\section{Auxiliary Lemmas}In this Section we state several Lemmas needed for the
proof of the main results. The proof of these are given in the Appendix.
\subsection{Truncation}
We consider truncated random variables $\widehat X_{jl}$ defined by
\begin{equation}\label{trunc000}
 \widehat X_{jl}:=X_{jl}\mathbb I\{|X_{jl}|\le c\lna^{\frac1{\varkappa}} \}-\E X_{jl}\mathbb I\{|X_{jl}|\le c\lna^{\frac1{\varkappa}}\}.
\end{equation}
Let $\widehat {\mathcal F}_n(x)$ denote the empirical spectral distribution function of the matrix $\widehat{\mathbf W}=\frac1{\sqrt n}(\widehat X_{jl})$ and let $\widehat m_n(z)$
be the corresponding Stieltjes transform. Let $\sigma_{jk}^2=\E(\widehat X_{jk})^2$. Introduce the  r.v.'s $\widetilde X_{jk}=\sigma_{jk}^{-1}\widehat X_{jk}$.
Consider the matrix $\widetilde {\mathbf W}=\frac1{\sqrt n}(\widetilde X_{jk})$.
Let $\widetilde m_n(z)$ denote the corresponding Stieltjes transform.
\begin{lem}\label{trunc}
 Assuming the conditions of Theorem \ref{main} there exist constants $C, c>0$ such that
 \begin{equation}\notag
  \Pr\{|m_n(z)-\widetilde m_n(z)|\ge \frac{C}{n^2v^2}\}\le \exp\{-c\lna\}.
 \end{equation}

\end{lem}

\begin{rem}\label{trunc00}
 In what follows we shall assume that r.v.'s $X_{jl}$ satisfy the condition
 \begin{equation}\label{trunc2}
  |X_{jl}|\le A_0\lna^{\frac1{\varkappa}}, \quad\E X_{jl}=0 \text{ and }\E X_{jk}^2=1,
 \end{equation}
 for some absolute constant $A_0$
We shall omit the symbol $\,\,\widehat{\cdot}\,\,$ in the notation of the truncated r.v.'s and corresponding characteristics of truncated matrices.
\end{rem}
\subsection{The Key Lemma}
Denote by $\mathbb T=\{1,\ldots,n\}$. For any $\mathbb A\subset\mathbb T$,
 introduce the matrices $\mathbf W^{(\mathbb A)}$,
which are  obtained from $\mathbf W$ by deleting the $j$-th row and the $j$-th column for $j\in\mathbb A$, and
the corresponding resolvent matrix $\mathbf R^{(\mathbb A)}$ defined by
$\mathbf R^{(\mathbb A)}:=(\mathbf W^{(\mathbb A)}-z\mathbf I_{n-|\mathbb A|})^{-1}$ and let $m_n^{(\mathbb A)}(z):=\frac1{n}\Tr\mathbf R^{(\mathbb A)}.$
Consider the index set $\mathbb T_{\mathbb A}:=\{1,\ldots,n\}\setminus \{\mathbb A\}$. By $\frak M^{(\mathbb A)}$  we shall denote  $\sigma$-algebra generated by r. v.'s $X_{rq}$ with $r,q\in\mathbb T_{\mathbb A}$.

We shall use the representation
\be\notag
R_{jj}=\frac1{-z+\frac1{\sqrt
n}X_{jj}-\frac1n{{\sum_{k,l\in\mathbb T_j}}}X_{jk}X_{jl}R^{(j)}_{kl}},
\en
(see,
for example, equality (4.6) in \cite{GT:03}). We may rewrite it as
follows
\be\label{repr01}
 R_{jj}=-\frac1{z+m_n(z)}+
\frac1{z+m_n(z)}\varepsilon_jR_{jj}, \en
where
$\varepsilon_j:=\varepsilon_{j1}+\varepsilon_{j2}+\varepsilon_{j3}+\varepsilon_{j4}$ with
\begin{align}
&\varepsilon_{j1}:=\frac1{\sqrt n}X_{jj},\quad
\varepsilon_{j2}:=\frac1n{\sum_{k\in\mathbb T_j}}(X_{jk}^2-1)R^{(j)}_{kk},\notag\\
&\varepsilon_{j3}:=\frac1n{\sum_{k\ne
l\in\mathbb T_j}}X_{jk}X_{jl}R^{(j)}_{kl},\quad
\varepsilon_{j4}:=\frac1n(\Tr \mathbf R^{(j)}-\Tr\mathbf R).\label{epsjn}
\end{align}
This relation immediately implies the following two equations
\begin{align}
 R_{jj}=&-\frac1{z+m_n(z)}-\sum_{\nu=1}^3\frac{\varepsilon_{j\nu}}{(z+m_n(z))^2}+\notag\\& \q\q\q + \sum_{\nu=1}^3
\frac1{(z+m_n(z))^2}\varepsilon_{j\nu}\varepsilon_jR_{jj}+\frac1{z+m_n(z)}\varepsilon_{j4}R_{jj}, \notag
\end{align}

and summing in $j=1,\ldots,n$,
\begin{align} \label{repr04}
 m_n(z)=&-\frac1{z+m_n(z)}-\frac1{(z+m_n(z))}\frac1n\sum_{j=1}^n\varepsilon_jR_{jj}\\
=&-\frac1{z+m_n(z)}-\frac1{(z+m_n(z))^2}\frac1n\sum_{\nu=1}^3\sum_{j=1}^n\varepsilon_{j\nu}+\\&  +
\frac1{(z+m_n(z))^2}\frac1n\sum_{\nu=1}^3\sum_{j=1}^n\varepsilon_{j\nu}\varepsilon_{j}R_{jj}+
\frac1{z+m_n(z)}\frac1n\sum_{j=1}^n\varepsilon_{j4}R_{jj}. \label{repr03}
\end{align}
Let $V_0=4+v_0$ and $v_0=\frac{c_0\beta_n^4}{n} $, where $c_0\ge\frac{3\cdot6^8}{\log2}$. Without loss of generality we may assume that $nv_0\ge 2$, for any $n\ge 1$.
 We choose $c_0$ later.
 Let $v_k=v_0+\frac{k}{n^2}$, for $k=0,\ldots, N=4n^2.$
Introduce the events, for $k=0,\ldots, N$,
\begin{equation}\notag
\mathcal A_k=\{\omega: \, |m_n(z)-s(z)|\le \frac12,\text{ for any } z=u+iv\text{ such that }v\ge v_k\}.
\end{equation}
Furthermore, let $\mathcal D_k$, for $k=0,\ldots, N$, denote the events
\begin{equation}\notag
\mathcal  D_k=\{\omega:\, |\varepsilon_j|\le \gamma_0; j=1,\ldots, n;\,z=u+iv:\,v\ge v_k\},
\end{equation}where $\gamma_0=\frac3{100}$ is an absolute constant.

\begin{lem}\label{cru1}
 Assuming the conditions of Theorem $1.1$  the following relations hold
 \begin{equation}
 \mathcal  D_k\cap\mathcal  A_k\subset \mathcal  A_{k-1},\text{ for } k=1,\ldots, N.
 \end{equation}
\end{lem}
\begin{proof}  For a proof of this Lemma see Subsection \ref{key} of the Appendix.
   \end{proof}
\begin{cor}\label{cor001}
 We have
\begin{equation}\notag
  \Pr\{\mathcal A_k^c\}\le \sum_{l=k+1}^N\Pr\{\mathcal D_l^c\cap\mathcal  A_l\}.
 \end{equation}
\end{cor}
\begin{proof}Note that $v_N\ge 4$. We have, for $z=u+iv$ with $v\ge v_N$
$$
\max\{|m_n(z)|,|s(z)|\}\le \frac14 \text{ a. s.}
$$
That means that
\begin{equation}
 \Pr\{\mathcal A_N\}=1.
\end{equation}
By Lemma \ref{cru1}, we have
\begin{equation}\label{cru6}
 \Pr\{\mathcal A_k^{(c)}\}\le \Pr\{(\mathcal D_{k+1}\cap\mathcal  A_{k+1})^c\}\le \Pr\{\mathcal D_{k+1}^c\cap\mathcal A_{k+1}\}+\Pr\{A_{k+1}^c\}.
\end{equation}

 The claim of Corollary \ref{cor01} now follows from \ref{cru6} by induction.
 \end{proof}

In the following we shall systematically
  use (as above) for any  $n\times n $ matrix $\mathbf W$
 together with its resolvent $\mathbf R$, its Stieltjes transform $m_n$ etc. the corresponding quantities $\mathbb W^{(\mathbb A)}$, $\mathbf R^{(\mathbb A)}$ and $m_n^{(\mathbb A)}$
 for the corresponding  sub matrix  with entries $X_{jk}, j, k \not \in \mathbb A$, $\mathbb A \subset \mathbb T=\{1,\ldots,n\}$.

If $\mathbb A=\emptyset$ we shall omit the set $\mathbb A$ as exponent index.
 Fix some integer $L$ which will be chosen later.
 Introduce additionally the event, for any $\mathbb A\subset \mathbb T$ such that $|\mathbb A|\le L$,
 \begin{equation}
  \mathcal U^{(\mathbb A)}:=\{\im m_n^{(\mathbb A)}(z)\le \psi_0+\frac{L-|\mathbb A|+1}{nv}\},
 \end{equation}
where
\begin{equation}
 \psi_0\le B_0, \text { for all }z=u+iv, v\ge v_0,
\end{equation}
and
\begin{equation}
 \psi_0\ge \max\{\frac{1}{nv},c\sqrt v\},\text{ for }z=u+iv, u\in[-2,2], v\ge v_0.
\end{equation}
The function $\psi_0=\psi_0(z)$  and the absolute constant $B_0$ will be chosen later. 


\begin{rem}\label{psi0}
 Note that for any $z=u+iv$ the  inequality
\begin{equation}\label{recursres}
 \im m_n(z)> \im m_n^{(\mathbb A)}(z)-\frac{|\mathbb A|}{nv}
 \end{equation}
  holds a.s.
 This yields, for $\mathbb A_1\subset\mathbb A_2$,
\begin{equation}
 (\mathcal U^{(\mathbb A_1)})^c\subset(\mathcal U^{(\mathbb A_2)})^c.
\end{equation}
\end{rem}
We introduce  as well the notation \q $\kappa_n=\psi_0\lna^{\frac1{\varkappa}}$.
\begin{lem}\label{cru2}
\begin{equation}\notag
\Pr\{\mathcal D_k^c\cap\mathcal A_k\cap \mathcal U\}\le \exp\{-c\sqrt {nv_k/\kappa_n}\}.
\end{equation}
\end{lem}
\begin{proof}For a  proof of this Lemma see Subsection \ref{proofofcru2} of the Appendix.
 \end{proof}

Introduce now for $z=u+iv\in\mathbb C_+$
\begin{equation}\notag
\mathcal A_v=\{|m_n(z)-s(z)|\le \frac12, \text{ for any }
u\in\mathbb R\}.
\end{equation}
 Applying Lemmas \ref{cru1} and \ref{cru2}, we get
\begin{cor}\label{cru3} The following inequality holds for any $v\ge v_0$
\begin{equation}\label{gnu1}
 \Pr\{\mathcal A_v^c\cap\mathcal U\}\le \exp\{-C\sqrt{nv/\kappa_n}\},
\end{equation}
for some positive constant $C>0$.
There exists a constant $C>0$ such that
\begin{equation}\label{rjj}
\Pr\{\{|R_{jj}|\le 3,\text{ for any } j=1,\ldots,n\}\cap \mathcal U\}\le
\exp\{-C\sqrt{nv/\kappa_n}\}.
\end{equation}
Moreover, for $z=u+iv$ with $|u|\le 2$ and $0<v\le 5$, there exist a constant $\delta>0$ such that
\begin{equation}\label{rjj1}
\Pr\{\{|R_{jj}|\ge \delta,\text{ for any } j=1,\ldots,n\}\cap \mathcal U\}\le
\exp\{-C\sqrt{nv/\kappa_n}\}.
\end{equation}
\end{cor}
\begin{proof}Let $v_{k}\ge v\ge v_{k-1}$.
By Lemma \ref{cru1}, we have
\begin{align}\notag
\Pr\{\mathcal A_v^c\cap\mathcal  U\}&\le \sum_{t=k+1}^N\Pr\{\mathcal D_t^c\cap \mathcal A_t\cap \mathcal U\}\le
n^2\exp\{-C\sqrt{nv/\kappa_n}\}\notag\\&\le \exp\{-C'\sqrt{nv/\kappa_n}\}
\end{align}
with some positive constant $C'$.
This inequality and Lemma \ref{cru2} yield  inequality \eqref{gnu1}.

Furthermore, we note that
the events $\mathcal A_v$ and $\mathcal D_k$ together imply by \eqref{resjj}  that
\begin{equation}
|R_{jj}(u+iv)|\le 3.
\end{equation}
On the other hand  the events $\mathcal A_v$ and $\mathcal D_k$ with  \eqref{repr01}imply 
\begin{equation}\notag
 |R_{jj}(u+iv)|\ge \frac1{2|z+m_n(z)|}\ge \frac16,
\end{equation}
for $|u|\le 2$.
This proves inequalities \eqref{rjj} and \eqref{rjj1}.
\end{proof}
 \begin{cor}\label{la1}
  Assuming conditions of Theorem \ref{main}, we have
  \begin{equation}
   \Pr\{\mathcal D_k^c\}\le \exp\{-C\sqrt{nv}/\lna^{\frac2{\varkappa}}\}.
  \end{equation}

 \end{cor}
\begin{proof} 
The result follows from the inequality
 \begin{equation}
  \Pr\{\mathcal D_k^c\}\le \Pr\{\mathcal D_{k+1}^c\cap\mathcal A_{k+1}\}+\Pr\{\mathcal A_{k+1}^c\}
 \end{equation}
and Lemma \ref{cru2} for $B_0=\frac32$.
\end{proof}
\begin{cor}\label{cor11}
 Assuming the conditions of Theorem \ref{main}, we have
 \begin{equation}
  \Pr\{\mathcal A_v\}\le \exp\{-C\sqrt{nv}/\lna^{\frac2{\varkappa}}\}.
 \end{equation}

\end{cor}
\begin{proof}
 The claim follows from Corollaries \ref{la1} and \ref{cor001}.
\end{proof}

 \subsection{A Bound for $\sum_l|R_{jl}|^2$}
In this Section we investigate the quantity
\begin{equation}\notag
 H_m^{(j,\mathbb A)}=\E\{(\sum_{t\in \mathbb T_{\mathbb A,j}}|R_{jt}^{(\mathbb A)}|^2)^m\mathbb I\{\mathcal U\}|\mathfrak M^{(j,\mathbb A)}\},\text{ for }j\in\mathbb T_{\mathbb A}.
\end{equation}
In what follows we shall denote by $\Phi_n$ a generic \it non-negative random error
term such that
\begin{equation}
\E\Phi_n\le \exp\{-c\sqrt{nv/\kappa_n}/\ln n\},
\end{equation}
for some positive constant $c>0$.
Furthermore, introduce the event
\begin{equation}\label{ba}
 \mathcal B^{(\mathbb A)}=\cap_{l\in\mathbb T_{\mathbb A}}\{C_1\le |R^{(\mathbb A)}_{ll}|\le C_2\},
\end{equation}
for some constants $C_1,C_2$ such that
\begin{equation}\label{ba1}
 \Pr\{(\mathcal B^{(\mathbb A)})^c\cap\mathcal  U^{(\mathbb A)}\}\le\exp\{-c\sqrt{nv/\kappa_n}\}.
\end{equation}
We may apply  Corollary \ref{cru3} to the matrix $\mathbf
W^{(\mathbb A)}$ to prove  that such constants exist. Note that,
by definition of $\Phi_n$, for $1\le m\le C\log n$,
\begin{equation}\notag
\E^{\frac1m}\mathbb I\{(\mathcal B^{(\mathbb A)})^c\cap \mathcal U^{(\mathbb
A)}|\mathfrak M^{(j,\mathbb A)}\}\le \Phi_n.
\end{equation}

\begin{lem}\label{h1}
 Under the conditions of Theorem \ref{main} there exists a constant $C>0$ such that for
  any $1\le m\le C\log n$ 
 \begin{equation}\notag
 \,(H^{(j,\mathbb A)}_m)^{\frac1m}\le Cm\lna^{\frac2{\varkappa}}v^{-1}\psi_0+\Phi_n.
 \end{equation}

\end{lem}
\begin{proof}

Here and in what follows we shall  consider the case $\mathbb
A=\emptyset$ only. Otherwise consider the matrix $\mathbf
W^{(\mathbb A)}$ instead of $\mathbf W$. Applying the
identity
\begin{equation}\label{inter}
 R_{jq}=-\frac1{\sqrt n}\sum_{t\in\mathbb T_{j}}X_{jt}R^{(j)}_{qt}R_{jj},
\end{equation}
we obtain, for $\omega\in\mathcal B$,
\begin{align}\label{one03}
 \sum_{q\in \mathbb T_j}|R_{jq}|^2
 \le \frac Cn\sum_{q\in \mathbb T_j}|\sum_{t\in\mathbb T_{j,q}}X_{jt}R^{(j)}_{qt}|^2.
\end{align}

We may write
\begin{equation}\label{one3}
 (H^{(j)}_m)^{\frac1m}\le v^{-2}{\Pr}^{\frac1m}
 \{\mathcal B^c\cap\mathcal U|\mathfrak M^{(j)}\}+
 \E^{\frac1m}\{\mathbb I\{\mathcal B\cap \mathcal U\}(\sum_{q\in \mathbb T_{j,q}}|R_{jq}|^2)^m|
 \mathfrak M^{(j)}\}.
 \end{equation}

 Inequalities  \eqref{one03}  and \eqref{one3} together imply using  (3.26) as well as $m=\log n $  and $v\ge v_0$
\begin{align}\label{one1}
 (H^{(j)}_m&)^{\frac1m}\le \Phi_n\notag\\&+\mathbb I\{\im m_n^{(j)}(z)\le \psi_0+\frac1{nv}\}
 \Big(\E\{\Big(\frac1n\sum_{q\in\mathbb T_{j}}
 \big|\sum_{r\in\mathbb T_{j,q}}X_{jr}R^{(j)}_{qr}\big|^{2}\Big)^m\big|\mathfrak M^{(j)}\}\Big)^{\frac1m}.
\end{align}
Consider the random variables
$Y_r=\frac1{c\lna^{\frac1{\varkappa}}}X_{jr}$.
Note that $Y_1,\ldots, Y_n$ are independent and, by Remark \ref{trunc00},
$|Y_{r}|\le 1,\quad \E Y_r=0$.

Consider the quadratic form in $Y_1,\ldots,Y_n$
\begin{equation}\notag
 f(Y_1,\ldots,Y_n)=\sum_{q=1}^n(\sum_{r=1}^na_{qr}Y_r)^2
\end{equation}
with $a_{qr}=R_{rq}^{(j)}/\sqrt n$. Note that $f$ is a convex function.
Let $Z_1,\ldots, Z_n$ denote standard Gaussian r.v.'s.
Then it follows from results of Bobkov \cite{bobkov:96}, \cite{bobkov:00}  (Choquet comparison of measures), that
\begin{equation}\notag
 \E^{\frac1m}|f(Y_1,\ldots,Y_n)|^m\le \E^{\frac1m}|f(c_0Z_1,\ldots,c_0Z_n)|^m,
\end{equation}
were $c_0=\frac{\sqrt{2\pi}}2$.  Note that 
\begin{equation}\notag
 f(c_0Z_1,\ldots,c_0Z_n)=c_0^2f(Z_1,\ldots,Z_n).
\end{equation}
For the Gaussian r.v.'s we have (\cite{bobkovgoetze:99}, Theorem 3.1)
\begin{equation}\notag
 \E^{\frac1m}|f(Z_1,\ldots,Z_n)|^m\le Cm\E|f(Z_1,\ldots,Z_n)|=Cm\sum_{q=1}^n\sum_{r=1}^n|a_{qr}|^2.
\end{equation}
In our case 
\begin{equation}
\sum_{q=1}^n\sum_{r=1}^n|a_{qr}|^2\le \frac1{n}\Tr|\mathbf R^{(j)}|^2\le v^{-1}(\im m_n(z)+\frac1{nv}).
\end{equation}
Applying these inequalities, we get, using that $X_{jq}=c\lna^{\frac1{\varkappa}}Y_{q}$, 
\begin{align}\label{bobkov}
 \Big(\E\{\Big(\frac1n\sum_{q\in\mathbb T_{j}}(
 |\sum_{r\in\mathbb T_{j,q}}X_{jr}R^{(j)}_{qr}|^{2}\Big)^m\big|\mathfrak M^{(j)}\}\Big)^{\frac1m}\le
 Cmv^{-1}\lna^{\frac2{\varkappa}}(\im m_n(z)+\frac 1{nv}).
\end{align}
Note that by definition $\psi_0\ge \frac 1{nv}$.
Inequalities \eqref{bobkov} and \eqref{one1} together imply
\begin{align}
(H^{(j)}_m)^{\frac1m}&\le
\Phi_n+m\lna^{\frac2{\varkappa}}v^{-1}\mathbb I\{\im m_n^{(j)}(z)\le \psi_0+\frac1{nv}\}(\im m_n(z)+\frac
1{nv})\notag\\& \le Cm\lna^{\frac2{\varkappa}}v^{-1}\psi_0+\Phi_n.
\end{align}
Thus Lemma \ref{h1} is proved.

\end{proof}
Note that  for $1\le m\le C\log n$ 
\begin{align}\notag
 v^{-1}\E\,{\Pr}^{\frac1m}\{(\mathcal B^{(\mathbb A,l)})^c\cap\mathcal U^{(\mathbb A,l)}|
 \mathfrak M^{(j,\mathbb A)}\}&\le
 v^{-1}\E^{\frac1m}\mathbb I\{(\mathcal B^{(\mathbb A,l)})^c\cap\mathcal U^{(\mathbb A,l)}\}\notag\\
 &\le \exp\{-c\sqrt{nv/\kappa_n}/\log n\}.
\end{align}
This means that
\begin{equation}
v^{-1}\E{\Pr}^{\frac1m}\{(\mathcal B^{(\mathbb
A,l)})^c\cap\mathcal U^{(\mathbb A,l)}|
 \mathfrak M^{(j,\mathbb A)}\}\le\Phi_n.
\end{equation}
\begin{lem}\label{h2}
Assuming the conditions of Theorem $1.1$, there exist constants $c,C$ such that for any $1\le m\le C\log n$  and for any $\mathbb A$, $j,l$ such that
$l\in\mathbb T_{\mathbb A}$, $j\in\mathbb T_{\mathbb A,l}$ and $|\mathbb A|\le C\log n$,
\begin{align}
 \E^{\frac1m}\big\{(\sum_{q\in\mathbb T_{\mathbb A,j,l}}&|R^{(\mathbb A)}_{ql}|^2)^m\mathbb I\{\mathcal U^{(\mathbb A)}\}
 \big|\mathfrak M^{(\mathbb A,j)}\big\}\notag\\&\le
 C(\sum_{q\in\mathbb T_{\mathbb A,j,l}}|R^{(\mathbb A,j)}_{ql}|^2)\E\{\mathbb I\{\mathcal U^{(\mathbb A)}\}
 \big|\mathfrak M^{(\mathbb A,j)}\}\notag\\&+
 (\frac{Cm\lna^{\frac2{\varkappa}}}{\sqrt n})(\sum_{q\in\mathbb T_{\mathbb A,j,l}}|R^{(\mathbb A,j)}_{ql}|^2)^{\frac 12}
 \E\{\mathbb I\{\mathcal U^{(\mathbb A)}\}\big|\mathfrak M^{(j)}\}
 +\Phi_n.
\end{align}

\end{lem}
\begin{proof} For the proof of this Lemma see Subsection \ref{proofofh2} of the Appendix.

\end{proof}
\subsection{Some Auxiliary Bounds for $\sum_{j=1}^n\varepsilon_{j3}$ }
In this section we consider the sum
\begin{equation}\notag
\zeta=\frac1n\sum_{j=1}^n\zeta_j=\frac1n\sum_{j=1}^n\varepsilon_{j3},
\end{equation}
where
\begin{equation}\notag
\zeta_j:=\varepsilon_{j3}=\frac1n\sum_{r\ne q\in \mathbb T_j}X_{jr}X_{jq}R^{(j)}_{qr}.
\end{equation}
 We shall introduce similar quantities to $\zeta$  for the matrix $\mathbf W^{(\mathbb A)}$.
More precisely,
\begin{equation}\notag
\zeta^{(\mathbb A)}=\frac1n\sum_{l\in\mathbb T_{\mathbb A}}\zeta_{l}^{(\mathbb A)},
\end{equation}
where
\begin{equation}\notag
\zeta_{l}^{(\mathbb A)}=\frac1n\sum_{r\ne q\in\mathbb T_{\mathbb A,l}}X_{lr}X_{lq}R^{(\mathbb A,l)}_{rq}.
\end{equation}

Introduce the events, for $\mathbb A\subset \mathbb T$ such that $|\mathbb A|\le L$,
\begin{align}\notag
 \mathcal E^{(\mathbb A)}&=\{|\im m_n^{(\mathbb A)}-s(z)|\le \frac12-\frac{|\mathbb A|}{nv}\}, \quad
 \notag\\\mathcal F_l^{(\mathbb A)}&=\{|\zeta_l^{(\mathbb A)}|\le
 (L-|\mathbb A|+1)\rho_n(\im m_n^{(\mathbb A,l)})^{\frac12}\},\notag\\
 \mathcal F^{(\mathbb A)}&=\cap_{l\in\mathbb T_{\mathbb A}}\mathcal F_l^{(\mathbb A)},\quad\rho_n=\frac{c_1\sqrt{\kappa_n}}{\sqrt{nv}}.
\end{align}

 Introduce the events  $\mathcal G^{(j)}=\mathcal U^{(j)}\cap\mathcal E^{(j)}
 \cap\mathcal F^{(j)}$ and $\mathcal G=\mathcal U\cap\mathcal E
 \cap\mathcal F$.
Furthermore, we introduce
the random variables
\begin{equation}\notag
 \widehat\zeta_l=\zeta_l\mathbb I\{\mathcal G\},\quad \widehat \zeta_l^{(j)}=\zeta_l^{(j)}\mathbb I\{\mathcal G^{(j)}\}.
\end{equation}
Note that,
\begin{equation}\label{trunc0}
 |\widehat\zeta_l|\le \frac{CL\sqrt{\kappa_n}}{\sqrt{nv}},\quad
 |\widehat\zeta_l^{( j)}|\le \frac{2CL\sqrt{\kappa_n}}{\sqrt{nv}}.
\end{equation}

 Define   $\mathcal G^{(\mathbb A)}$  similarly to $\mathcal G^{(j)}$ and write
\begin{equation}\notag
  \widehat\zeta^{(\mathbb A)} =\frac1n\sum_{j\in\mathbb T_{\mathbb A}}\widehat\zeta_j^{(\mathbb A)}.
\end{equation}

\begin{lem}\label{trunczeta}
  Assuming the conditions of Theorem \ref{main} and condition
 \begin{equation}\label{u0}
  \Pr\{(\mathcal U^{(\mathbb A)})^c\}\le \exp\{-c\lna\},
 \end{equation}
 there exist constants $c,C>0$ such that for any $v\ge v_0$
 \begin{equation}
  \Pr\{\zeta\ne\widehat\zeta\}\le \exp\{-c\lna\}.
 \end{equation}
\end{lem}
\begin{proof}
For the proof of this lemma see Subsection \ref{proofoftrunczeta} of the Appendix.
\end{proof}

Recall that $\mathfrak M^{(\mathbb A)}$ denote the $\sigma$-algebra generated by all $X_{pq}$ except the entries  $X_{jq}$ and $X_{qj}$, for $q=1,\ldots,n$ and $j\in\mathbb A$, that is the rows and columns listed in
 $\mathbb A$.
Note that $\zeta^{(\mathbb A)}$ is measurable w.r.t. $\mathfrak M^{(\mathbb A)}$
and independent of $X_{jr}$ for   $j\in {\mathbb A}$ and any $r\in\mathbb T_{\mathbb A}$.

We consider the difference $\zeta_l-\zeta_l^{(j)}$.
We shall first expand the quantity
$\zeta_l-\zeta_l^{(j)}$ in terms of the r.v.'s $X_{jq}$, for $q\in\mathbb T_j$.

Note that
\begin{align}
 \widehat\zeta_l-\widehat\zeta_l^{(j)}=(\zeta_l-\zeta_l^{(j)})\mathbb I\{\mathcal  G\}\mathbb I\{\mathcal G^{(j)}\}
 -\zeta_l^{(j)}\mathbb I\{\mathcal  G^c\}\mathbb I\{\mathcal G^{(j)}\}+\zeta_l\mathbb I\{\mathcal  G\}\mathbb I\{(\mathcal G^{(j)})^c\}\notag
 \end{align}
and write using $R^{({l})}_{qr}-R^{(\tc{l,j})}_{qr}=  R^{(l)}_{jq}R^{(l)}_{jr}/R^{(l)}_{jj}$, (see as well \eqref{rec-central}),
\begin{align}\notag
 \zeta_l-\zeta_l^{(j)}&=\frac1n\sum_{q\ne r\in\mathbb T_l}X_{lq}X_{lr}R^{(l)}_{qr}-\frac1n\sum_{q\ne r\in\mathbb T_{l,j}}X_{lq}X_{lr}R^{(l,j)}_{qr}\notag\\& =\frac2nX_{lj}\sum_{q\in\mathbb T_{l,j}}X_{lq}R^{(l)}_{qj}+\frac1n\sum_{q\ne r\in\mathbb T_{l,j}}X_{lq}X_{lr}(R^{(l)}_{qr}-R^{(l,j)}_{qr})\notag\\&=\frac2nX_{lj}\sum_{q\in\mathbb T_{l,j}}X_{lq}R^{(l)}_{qj}+\frac1n\sum_{q\ne r\in\mathbb T_{l,j}}X_{lq}X_{lr}
R^{(l)}_{jq}R^{(l)}_{jr}(R^{(l)}_{jj})^{-1}\notag\\&=  \xi_{lj}+\eta_{lj},
\end{align}
 where 
\begin{equation}\label{conec01}
 \xi_{lj}=\frac1n\sum_{r\ne q\in \mathbb T_{l,j}}X_{lr}X_{lq}R^{(l)}_{jq}R^{(l)}_{jr}(R^{(l)}_{jj})^{-1},\quad
\eta_{lj}=\frac2nX_{lj}\sum_{q\in\mathbb T_{ l,j}}X_{lq}R^{(l)}_{jq}.
\end{equation}
Furthermore, we introduce notation
\begin{equation}
 \widehat\eta_{lj}=\eta_{lj}\mathbb I\{\mathcal G\}\mathbb I\{\mathcal G^{(j)}\},
 \quad \widehat \xi_{lj}=\xi_{lj}\mathbb I\{\mathcal G\}\mathbb I\{\mathcal G^{(j)}\},
\end{equation}
and
\begin{equation}\notag
 r_{j}=-\zeta^{(j)}\mathbb I\{\mathcal  G^c\}\mathbb I\{\mathcal G^{(j)}\}+\zeta\mathbb I\{\mathcal  G\}\mathbb I\{(\mathcal G^{(j)})^c\}.
\end{equation}
Then,
\begin{equation}\label{central}
\widehat\zeta-\widehat\zeta^{(j)}=\frac1n\sum_{l\in\mathbb T_j}\widehat\xi_{lj}+\frac1n\sum_{l\in\mathbb T_j}\widehat\eta_{lj}+r_{j}+\frac1n\widehat{\zeta_j}.
\end{equation}

Note that by expansion in the $l$th  and $j$th row 
\begin{equation}\label{inter1}
 R_{lj}=-\frac1{\sqrt n}\sum_{q\in\mathbb T_l}X_{lq}R^{(l)}_{jq}R_{ll}=
 -\frac1{\sqrt n}\sum_{q\in\mathbb T_j}X_{jq}R^{(j)}_{lq}R_{jj}.
\end{equation}
Using these equalities, we rewrite  $R^{(l)}$ in terms of $R$ and smaller terms via
\begin{align}\label{conec1}
 \xi_{lj}=& 
 (\frac1{\sqrt n}\sum_{q\in \mathbb T_l}X_{lq}R^{(l)}_{jq}-\frac1{\sqrt n}X_{lj}R^{(l)}_{jj})
 \Big(-R_{lj}R_{ll}^{-1}-\frac1{\sqrt n} X_{lj}R^{(l)}_{jj}\Big)(R^{(l)}_{jj})^{-1}
 \notag\\&\quad -\frac1{ n}\sum_{q\in \mathbb T_{l,j}}X_{lq}^2(R^{(l)}_{jq})^2 (R^{(l)}_{jj})^{-1}\notag\\=&
 \Big(-R_{lj}R_{ll}^{-1}-\frac1{\sqrt n} X_{lj}R^{(l)}_{jj}\Big)^2(R^{(l)}_{jj})^{-1}
 -\frac1{ n}\sum_{q\in \mathbb T_{l,j}}X_{lq}^2(R^{(l)}_{jq})^2 (R^{(l)}_{jj})^{-1}.
\end{align}

The last relation in \eqref{conec1} and the representation \eqref{inter1} together imply
\begin{align}\label{conec2}
 |\xi_{lj}|&\le 2\Big|\frac1{ n}\sum_{ q\in \mathbb T_{j}}X_{jq}R^{(j)}_{lq}\Big|^2
 |R_{jj}^{(l)}|^{-1}|R_{jj}|^2|R_{ll}|^{-2}\notag\\& \qquad\qquad +
 \frac{2\lna^{\frac2{\varkappa}}}{n}|R_{jj}^{(l)}|+\frac{C\lna^{\frac2{\varkappa}}}{n}
 \sum_{q\in\mathbb T_{l,j}}|R^{(l)}_{jq}|^2|R^{(l)}_{jj}|^{-1}.
\end{align}

\begin{lem}\label{cond1}
 Under the conditions of Theorem $1.1$ there exist constants $C,c>0$ such that for any subset
 $\mathbb A\subset \mathbb T$ with $|\mathbb A|\le C\log n$, 
 \begin{align}\label{cond2}
  \E^{\frac1m}\{|\widehat\xi_{lj}^{(\mathbb A)}|^m\big|\mathfrak M^{(j,\mathbb A)}\}&\le\frac{Cm\lna^{\frac2{\varkappa}}}{nv}
  (v\sum_{t\in\mathbb T_{j,l}}|R^{(j,\mathbb A)}_{lt}|^2\mathbb I\{\mathcal U^{(j)}\}
+\psi_0)+\Phi_n.
 \end{align}

\end{lem}
\begin{proof}We consider the case $\mathbb A=\emptyset$ only. In general we consider
$\mathbf W^{(\mathbb A)}$ instead of $\mathbf W$.
Note that for $\omega\notin \mathcal U$ we have $\widehat\xi^{(\mathbb A)}_{lj}=0$.
 By relation \eqref{conec2},
for $\omega\in\mathcal G\cap\mathcal G^{(j)}\cap\mathcal B^{(j)}\cap\mathcal B^{(l)}\cap\mathcal B$,  (where $\mathcal B^{(\mathbb A)}$
has been defined in\eqref{ba}), we have
\begin{align}
 |\xi_{lj}|&\le C|\frac1{\sqrt n}\sum_{ q\in \mathbb T_{j}}X_{jq}R^{(j)}_{lq}|^2
 +
 \frac{C\lna^{\frac2{\varkappa}}}{nv}\psi_0^2+\frac{C\lna^{\frac2{\varkappa}}}{n}
 \sum_{q\in\mathbb T_{l,j}}|R^{(l)}_{jq}|^2.\notag
\end{align}
From here, applying Rosenthal's inequality,  it follows 
\begin{align}
 &\E^{\frac1m}\{|\widehat\xi_{lj}|^m\mathbb I\{\mathcal B^{(j)}\cap \mathcal B^{(l)}\cap\mathcal B\}
 \big|\mathfrak M^{(j)}\}\notag\\
& \le \frac{Cm\lna^{\frac2{\varkappa}}}{n}\Big(\sum_{q\in\mathbb T_{j}}|R_{lq}^{(j)}|^2\mathbb I\{\mathcal U^{(j)}\}+ v^{-1}\psi_0\notag\\&\qquad\qquad\qquad+
 \E^{\frac1m}\{(\sum_{q\in\mathbb T_{j,l}}|R_{jq}^{(l)}|^2)^m\mathbb I\{ \mathcal U^{(l)}\cap\mathcal B^{(l)}\}\big|\mathfrak M^{(j)}\}\Big).\notag
\end{align} 
We use here that $\psi_0\le B_0$.
Applying now Lemma \ref{h1}
we obtain
\begin{align}
\E^{\frac1m}\{|\widehat\xi_{lj}|^m\mathbb I\{\mathcal B\}\big|\mathfrak M^{(j)}\}
&\le \frac{Cm\lna^{\frac2{\varkappa}}}{n}(
 \sum_{q\in\mathbb T_{j,l}}|R_{lq}^{(j)}|^2\mathbb I\{\mathcal U^{(j)}\}+v^{-1}\psi_0)+\Phi_n
.\notag
 \end{align}
 Thus Lemma \ref{cond1} is proved.
\end{proof}
\begin{lem}\label{cond02}
 Under the conditions of Theorem $1.1$ there exist constants $C,c>0$ such that
 \begin{align}\notag
  \E^{\frac1m}\{|\widehat\eta_{lj}|^m\big|\mathfrak M^{(j)}\}&\le \frac{m\lna^{\frac2{\varkappa}}}{n}\Big((\sum_{q\in\mathbb T_{jl}}|R^{(j)}_{lq}|^2)^{\frac12}
  \mathbb I\{\mathcal U^{(j)}\}+\psi_0\Big)+\Phi_n.
 \end{align}

\end{lem}
\begin{proof}We use equality \eqref{inter1} again.
We write
\begin{align}
  \eta_{lj}&=\frac{X_{lj}}{\sqrt n}(-R_{lj}R_{ll}^{-1}-\frac{X_{lj}}{\sqrt n}R_{jj}^{(l)})\notag\\&=
 \frac{X_{lj}}{\sqrt n}(R_{jj}\frac1{\sqrt n}\sum_{q\in\mathbb T_j}X_{jq}R^{(j)}_{ql}R_{ll}^{-1}
 -\frac{X_{lj}}{\sqrt n}R_{jj}^{(l)}).\notag
\end{align}
We may write
\begin{align}
\E\{|\widehat\eta_{lj}|^m\big|\mathfrak M^{(j)}\}&\le
\E\{|\widehat\eta_{lj}|^m\mathbb I\{\mathcal B\}\big|\mathfrak M^{(j)}\}+\E\{|\widehat\eta_{lj}|^m\mathbb I\{\mathcal B^{(c)}\}\big|\mathfrak M^{(j)}\}.
\end{align}
 Applying Rosenthal's inequality, we get
 \begin{equation}\notag
  \E^{\frac1m}\{|\widehat\eta_{lj}|^m\mathbb I\{\mathcal B\}\big|\mathfrak M^{(j)}\}\le
  \frac{m\lna^{\frac2{\varkappa}}}{n}(\sum_{q\in\mathbb T_{jl}}|R^{(j)}_{lq}|^2)^{\frac12}\mathbb I\{\im m_n^{(j)}\le \psi_0\}
  +\frac{C\lna^{\frac2{\varkappa}}}{n}.
 \end{equation}
 Moreover,
 \begin{equation}
   \E^{\frac1m}\{|\widehat\eta_{lj}|^m\mathbb I\{\mathcal B^c\}\big|\mathfrak M^{(j)}\}\le
   \frac{C\lna^{\frac2{\varkappa}}}{\sqrt{nv}}
   \E^{\frac1m}\{\mathbb I\{\mathcal B^c\cap\mathcal U\}\big|\mathfrak M^{(j)}\}\le\Phi_n.
 \end{equation}

Thus the Lemma is proved.
\end{proof}

\begin{lem}\label{cond3}
 Under the conditions of Theorem $1.1$ there exists a constant $C>0$ such that for the term $r_j$ in \eqref{central} we have
 \begin{equation}\notag
  \E^{\frac1m}\{|r_{j}|^m\big|\mathfrak M^{(j)}\}\le \frac{Cm\sqrt{\kappa_n}}{\sqrt{nv}}\mathbb I\{(\mathcal G^{(j)})^c\}+
 \Phi_n.
 \end{equation}

\end{lem}
\begin{proof}For the proof of this lemma see Subsection \ref{proofofcond3} of the Appendix.

\end{proof}

 We now conclude the bound of the r.h.s of \eqref{central}.

\begin{lem}\label{last1} There exist constants $c,C>0$ such that for any $1\le m\le C\log n$,
 \begin{align}
 \sum_{l\in\mathbb T_j}& \E^{\frac1m}\{|\widehat\zeta_l-\widehat\zeta^{(j)}_l|^m\big|\mathfrak M^{(j)}\}\notag\\&\le \frac{Cm\kappa_n}{{nv}}
 +\frac{Cm\sqrt{\kappa_n}}{\sqrt{nv}}\mathbb I\{(\mathcal G^{(j)})^c\}+
 \Phi_n.\notag
 \end{align}

\end{lem}
\begin{proof} In order to prove this claim it is enough\
 to sum the right hand sides of
 the inequalities in Lemmas \ref{cond1}, \ref{cond2} and \ref{cond3} for $l\in\mathbb T_j$ .

\end{proof}

We consider now the functions   $g^{(\mathbb A)}(s)=\E\exp\{s\widehat\zeta^{(\mathbb A)}\}$ and $g(s)=g^{(\emptyset)}(s)$.
\begin{lem}\label{key10}There exist  constants $c,C>0$ such that, 
for $s\le cnv/(\kappa_n\log n)$,
 \begin{equation}\label{rec1}
  g^{(\mathbb A)}(s)\le C^{|\mathbb A|}g(s).
 \end{equation}

\end{lem}
\begin{proof}
 We shall prove inequality \eqref{rec1} for $|A|=1$. The general case follows by induction.
 Let $\mathbb A=\{j\}$. Note that
 \begin{align}
  g(s)&=\E\exp\{s(\widehat\zeta-\widehat\zeta^{(j)})\}\exp\{\widehat\zeta^{(j)}\}=
  \E\big(\E\{\exp\{s(\widehat\zeta-\widehat\zeta^{(j)})\}\big|\mathfrak M^{(j)}\}\big)
  \exp\{\widehat\zeta^{(j)}\}.\notag
  \end{align}

Applying Jensen's inequality, we get
\begin{align}\label{main15}
  g(s)&
  \ge\E\Big(\exp\{s\E\{\widehat\zeta-\widehat\zeta^{(j)}
  \big|\mathfrak M^{(j)}\}\}
  \exp\{\widehat\zeta^{(j)}\}\Big).
 \end{align}

 By Lemma \ref{last1}, we have
 \begin{align}
V_j=\E\{|\widehat\zeta-\widehat\zeta^{(j)}|\big|\mathfrak M^{(j)}\}
&\le
\frac{C\kappa_n}{nv}
  +\frac{C\sqrt{\kappa_n}}{\sqrt{nv}}\mathbb I\{(\mathcal G^{(j)})^c\}+
  \Phi_n.\notag
\end{align}
This implies that
\begin{align}\label{last2}
 \exp\{-sV_j\}\ge
 \exp\{-\frac{s\kappa_n}{nv}\}\exp\{-\Phi_n\}
 \exp\{-\frac{C\sqrt{\kappa_n}}
 {\sqrt{nv}}\mathbb I\{(\mathcal G^{(j)})^c\}\},
\end{align}
Note that for any r.v.'s $a,b$ such that  $ab=0$, we have $\exp\{a+b\}=\exp\{a\}+\exp\{b\}-1.$
Since $\mathbb I\{(\mathcal G^{(j)})^c\}\widehat\zeta^{(j)}=0$, we get
\begin{align}
 \exp\{-sV_j\}\exp\{s\widehat\zeta_j\}&
 \ge\exp\{-\frac{s\kappa_n}{nv}\}
 \exp\{-\Phi_n\}\Big[
 \exp\{s\widehat\zeta^{(j)}\}
 \notag\\&
\qquad +\exp\{-\frac{s\sqrt{\kappa_n}}{\sqrt{nv}}\mathbb I\{(\mathcal G^{(j)})^c\}\}-1\Big].
 \end{align}
 From here it follows that
 \begin{align}
 \exp\{-sV_j\}\exp\{s\widehat\zeta_j\}&
 \ge\exp\{-\frac{s\kappa_n}{nv}\}\exp\{-\Phi_n\}\exp\{s\widehat\zeta^{(j)}\}\notag\\&
   \qquad -
 \exp\{-\frac{s\kappa_n}{nv}\}\exp\{-\Phi_n\}.
\end{align}

Note that
\begin{equation}\notag
 \exp\{-\Phi_n\}\ge \exp\{-c_0\}\mathbb I\{\Phi_n\le c_0\}
\end{equation}
and, for $0<s\le Cnv/\kappa_n$,
\begin{equation}\notag
\exp\{-\frac{s\kappa_n}{nv}\}\exp\{-\Phi_n\} \ge\exp\{-c_1\}\mathbb I\{\Phi_n\le c_0\}.
\end{equation}
Combining these bounds with inequality \eqref{main15}  and inequality \eqref{trunc0}, we get, for $0\le s\le c'nv/(\kappa_n\log n)$,
 \begin{align}
  g(s)&\ge \exp\{-c_1\}\E\mathbb I\{\Phi_n\le c_0\}\exp\{\widehat\zeta_j\}
\\&\ge
  \exp\{-c_1\}\E\exp\{\widehat\zeta_j\}
-\exp\{-c_1\}\exp\{\frac{cs\sqrt{\kappa_n}}{\sqrt{nv}}\}\E\mathbb I\{\Phi_n> c_0\}
  \notag\\&\ge\exp\{-c_1\}
\E\exp\{\widehat\zeta_j\}-\exp\{-c_1\}\exp\{\frac{cc'\sqrt{nv}}{\sqrt{\kappa_n}\log n}\}
  \E\mathbb I\{\Phi_n> c_0\}.\notag
 \end{align}
It is easy to see that
\begin{align}
 \E\mathbb I\{\Phi_n> c_0\}\le \exp\{-C\sqrt{nv/\kappa_n}/\log n\}.
\end{align}
We may chose the constant $c'$ small enough that 
\begin{equation}
\exp\{\frac{cc'\sqrt{nv}}{\sqrt{\kappa_n}\log n}\}
  \E\mathbb I\{\Phi_n> c_0\}\le \exp\{-c\lna\}.
\end{equation}
Without loss of generality we may assume that $g(s)\le 1$.
From here it follows that for $0\le s\le cnv/(\kappa_n\log n)$
\begin{equation}\notag
 \E\exp\{\widehat\zeta_j\}\le Cg(s).
\end{equation}
The last bound completes the proof.
Thus Lemma \ref{key10} is proved.

\end{proof}
\subsection{Some additional bounds}
We consider in this Section the quantities
\begin{align}
 \alpha_{lj}^{(\mathbb A,t)}=\xi_{lj}^{(\mathbb A)}-\xi_{lj}^{(\mathbb A,t)}\notag
\end{align}
and
\begin{equation}\notag
 \gamma_{lj}^{(\mathbb A,t)}=\eta_{lj}^{(\mathbb A)}-\eta_{lj}^{(\mathbb A,t)},
\end{equation}
where $t\in\mathbb T_{\mathbb A}$, $j\in\mathbb T_{\mathbb A,t}$, and $l\in\mathbb T_{\mathbb A,t,j}$.
 For $p\ne r\in\mathbb T_{\mathbb A,q}$ and $q\in\mathbb T_{\mathbb A}$  we have
 \begin{align}\label{rec-central}
  R_{qr}^{(\mathbb A)}- R_{qr}^{(\mathbb A,t)}=R^{(\mathbb A)}_{qt}R^{(\mathbb A)}_{rt}(R^{(\mathbb A)}_{tt})^{-1},
 \end{align}
(see, for instance \cite{ErdosYauYin:2010}, Lemma 3.2, formula (3.6)) and may represent
\begin{equation}\notag
 \alpha_{lj}^{(\mathbb A,t)}=\widetilde\alpha_{lj}^{(\mathbb A,t)}+\widehat\alpha_{lj}^{(\mathbb A,t)},
\end{equation}
where
\begin{align}
\widetilde\alpha_{lj}^{(\mathbb A,t)}&= \frac1nX_{lt}\sum_{q\in\mathbb T_{\mathbb A,l, t}}X_{lq}R^{(\mathbb A,l)}_{qj}
R^{(\mathbb A,l)}_{tj}( R^{(\mathbb A,l)}_{jj})^{-1},\notag\\
\widehat\alpha_{lj}^{(\mathbb A,t)}&=\widehat\alpha_{1lj}^{(\mathbb A,t)}+\widehat\alpha_{2lj}^{(\mathbb A,t)}
+\widehat\alpha_{3lj}^{(\mathbb A,t)}\notag
\end{align}
and
\begin{align}
\widehat\alpha_{1lj}^{(\mathbb A,t)}&=\frac1n\sum_{r\ne q\in\mathbb T_{\mathbb A,l,t}}X_{lq}X_{lr}
R_{qt}^{(\mathbb A,l)}R_{jt}^{(\mathbb A,l)}R^{(\mathbb A,l)}_{jr}
(R^{(\mathbb A,l)}_{jj}R^{(\mathbb A,l)}_{tt})^{-1},\notag\\
\widehat\alpha_{2lj}^{(\mathbb A,t)}&=\frac1n\sum_{r\ne q\in\mathbb T_{\mathbb A,l,t}}X_{lq}X_{lr}
R_{rt}^{(\mathbb A,l)}R_{jt}^{(\mathbb A,l)}
R^{(\mathbb A,l,t)}_{jq}(R^{(\mathbb A,l)}_{tt}R^{(\mathbb A,l)}_{jj})^{-1},\notag\\
\widehat\alpha_{3lj}^{(\mathbb A,t)}&=\frac1n\sum_{r\ne q\in\mathbb T_{\mathbb A,l,t}}X_{lq}X_{lr}R^{(\mathbb A,l,t)}_{jq}R^{(\mathbb A,l,t)}_{jr}
(R_{jt}^{(\mathbb A,l)})^2(R^{(\mathbb A,l)}_{tt}R^{(\mathbb A,l)}_{jj}R^{(\mathbb A,l,t)}_{jj})^{-1}.\notag
\end{align}
\begin{lem}\label{h3}
 Assuming the conditions of Theorem $1.1$, there exist constants $C,c$ such that 
 \begin{align}
  \E^{\frac1m}\{|\alpha_{lj}^{(\mathbb A,t)}|^m\big|\mathfrak M^{(\mathbb A,j,t)}\}&
  \le\frac{Cm^3}{(nv)^{\frac32}}\bigg(v\Big(\sum_{q\in\mathbb T_{t}}|R^{(j,t)}_{ql}|^2\Big)\lna^{\frac2{\varkappa}}\sqrt{\kappa_n}\,\mathbb I\{\mathcal U^{(j,t)}\}
 +\kappa_n^{\frac32}\bigg)
\notag\\& \qquad + \Phi_n.
\end{align}

\end{lem}
\begin{proof}
For the  proof of this lemma see  Subsection \ref{proofofh3} of the Appendix.
\end{proof}
\begin{cor}\label{h5}
Assuming conditions of Theorem $1.1$, there exist constants $C,c$ such that 
 \begin{align}
 \frac1n\sum_{l\in\mathbb T_{jt}} \E^{\frac1m}\{|\alpha_{lj}^{(\mathbb A,t)}|^m\big|\mathfrak M^{(\mathbb A,j,t)}\}\le
 \frac{Cm^3}{(nv)^{\frac32}}
\kappa_n^{\frac32}+\Phi_n.\notag
\end{align}

\end{cor}
\begin{proof}
 To prove the claim it is enough to use the inequality
 \begin{equation}\notag
  \frac1n\sum_{l\in\mathbb T_{jt}}\sum_{q\in\mathbb T_{t,j}}|R^{(t,j)}_{ql}|^2\le v^{-1}\im m_n(z)+\frac C{nv}.
 \end{equation}

\end{proof}
Similar to Lemma \ref{h3} we get
\begin{lem}\label{h4}
 Assuming the conditions of Theorem $1.1$, there exist constants $C,c$ such that for
 $0\le s\le cnv/\kappa_n$,
 \begin{align}\label{h40}
  \E^{\frac1m}\{|\gamma_{lj}^{(\mathbb A,t)}|^m\big|\mathfrak M^{(\mathbb A,j,t)}\}&\le\frac{C}{n\sqrt {nv}}
\Big(\sum_{q\in\mathbb T_{j,l,t}}|R_{ql}^{(j,t)}|^2\Big)^{\frac12}\mathbb I\{\mathcal U^{(j,t)}\}\sqrt{\kappa_n}+\Phi_n.
\end{align}

\end{lem}

\begin{proof}
For the proof of this Lemma see  Subsection \ref{proofofh4}
of the Appendix.
\end{proof}

\begin{cor}\label{h50}
Assuming conditions of Theorem $1.1$, there exist constants $C,c$ such that for $0\le s\le cnv/(\psi_0\lna^{\frac2{\varkappa}})$,
 \begin{align}
 \frac1n\sum_{l\in\mathbb T_{jt}} \E^{\frac1m}\{|\gamma_{lj}^{(\mathbb A,t)}|^m\big|\mathfrak M^{(\mathbb A,j,t)}\}\le
 \frac{Cm^3}{(nv)^{\frac32}}
\kappa_n^{\frac32}+\Phi_n.\notag
\end{align}

\end{cor}
\begin{proof}
 To prove the claim it is enough to use the inequality
 \begin{equation}\notag
  \frac1n\sum_{l\in\mathbb T_{jt}}(\sum_{q\in\mathbb T_{t,j}}|R^{(t,j)}_{ql}|^2)^{\frac12}\le v^{-\frac12}(\im m_n(z)+\frac C{{nv}})^{\frac12}.
 \end{equation}
 \end{proof}
\section{Bounds for the Laplace transform of $\widehat\zeta$}
 Let
$g(s)=\E\exp\{s\widehat\zeta\}$.
Let
\begin{equation}
\mathcal E_1=\mathcal E\cap\mathcal U.
\end{equation}
In this Section we prove the crucial result for the proof of Theorem \ref{main}.
\begin{prop}\label{mainprop}
 Assuming the conditions of Theorem \ref{main} and 
 \begin{equation}
  \E\mathbb I\{\mathcal E_1^c\}\le \exp\{-c\lna\},
 \end{equation}
there exist  constants $C_1,C_2,c>0$ depending on $\varkappa$ and $A$ only such that for $0\le s\le{cnv}/(\kappa_n\log n)$,
 \begin{equation}
  g(s)\le C_1\exp\Bigg\{\frac{C_2s^2\kappa_n^2}{n^2v^2}\Bigg\}.
 \end{equation}
\end{prop}
\begin{proof}
Consider the derivative of $g(s)$:
\begin{equation}\label{laplas1}
g'(s)=\E\widehat\zeta\exp\{s\widehat\zeta\}=\frac1n\sum_{j=1}^n\E\widehat\zeta_j
\exp\{s\widehat\zeta\}.
\end{equation}
We continue with equality \eqref{laplas1}. Let $k=c\log n$, where the constant $c$ will be chosen later.

{\it Step 1.}
We have
\begin{equation}\label{laplas2}
g'(s)=\sum_{\nu=0}^kA_{\nu}+A_{k+1},
\end{equation}
where
\begin{align}
A_{\nu}&=\frac {s^{\nu}}{\nu!n}\sum_{j=1}^n\E\widehat\zeta_j(\widehat\zeta
-\widehat\zeta^{(j)})^{\nu}
\exp\{s\widehat\zeta^{(j)}\},\text{ for }
\nu=0,1\ldots,k,\notag\\
A_{k+1}&=\frac1n\sum_{j=1}^n\E\widehat\zeta_j
\left(\exp\{s(\widehat\zeta-\widehat\zeta^{(j)})\}
-\sum_{\nu=0}^k
\frac{s^{\nu}}{\nu!}(\widehat\zeta-\widehat\zeta^{(j)})^{\nu}\right)
\exp\{s\widehat\zeta^{(j)}\}.\notag
\end{align}
Note that
\begin{equation}\notag
 \E\{\widehat\zeta_j\big|\mathfrak M^{(j)}\}=
 \E\{\zeta_j\mathbb I\{\mathcal G\}\big|\mathfrak M^{(j)}\}=
- \E\{\zeta_j\mathbb I\{(\mathcal G^c)\}\big|\mathfrak M^{(j)}\}.
\end{equation}
We use that
\begin{equation}\notag
 \mathbb I\{\mathcal G^c\}\le\mathbb I\{(\mathcal G^{(j)})^c\}+\mathbb I\{(\mathcal G^c)\}\mathbb I\{\mathcal G^{(j)}\}.
 \end{equation}
This implies

\begin{align}\label{ao}
 |A_0|&\le\frac1n\sum_{j=1}^n
 \E\Big(\E\{|\zeta_j|
 \big|\mathfrak M^{(j)}\}\Big)\mathbb I\{(\mathcal G^{(j)})^c\}
\exp\{s\widehat\zeta^{(j)}\}\notag\\&{\qquad +}
\frac1n\sum_{j=1}^n
 \E\Big(\E\{|\zeta_j|\mathbb I\{(\mathcal G^c)\}|\mathfrak M^{(j)}\}\Big)
 \mathbb I\{\mathcal G^{(j)}\}\exp\{s\widehat\zeta^{(j)}\}.
\end{align}

Note that 
\begin{align}\label{a1}
 \E\{(\E\{|\zeta_j|\big|\mathfrak M^{(j)}\})\mathbb I\{(\mathcal G^{(j)})^c\}
 \exp\{s\widehat\zeta^{(j)}\}&= \E|\zeta_j|\mathbb I\{(\mathcal G^{(j)})^c\}\notag\\&
 \le \E^{\frac12}|\zeta_j|^2\E^{\frac12}\mathbb I\{(\mathcal G^{(j)})^c\}.
\end{align}
We use here that
$\mathbb I\{(\mathcal G^{(j)})^c\}\exp\{\widehat\zeta^{(j)}\}=
\mathbb I\{(\mathcal G^{(j)})^c\}.$
Note that, for $v\ge v_0$,
\begin{equation}\label{a7}
\E|\zeta_j|^2\le n^{-1}v^{-1}\E\,\im m_n^{(j)}(z)\le C.
\end{equation}
The inequalities \eqref{a1} and  \eqref{a7} together imply
\begin{align}
 \E(\E\{|\zeta_j|\mathbb I\{(\mathcal G^{(j)})^c\}\big|\mathfrak M^{(j)}\})
 \exp\{s\widehat\zeta^{(j)}\}\le \exp\{-c\lna\}g(s).
\end{align}
Furthermore, applying Cauchy's inequality, we get
\begin{align}
 \E\{|\zeta_j|\mathbb I\{\mathcal G^c\}\mathbb I\{\mathcal G^{(j)}\}|\mathfrak M^{(j)}\}
 \le \E^{\frac12}\{|\zeta_j|^2|\mathfrak M^{(j)}\}
 \E^{\frac12}\{\mathbb I\{\mathcal G^c\}\mathbb I\{\mathcal G^{(j)}\}|\mathfrak M^{(j)}\}.
\end{align}
According to Remark \ref{rem1} we have  
\begin{equation}
 \E\{\mathbb I\{(\mathcal G^c)\}\mathbb I\{\mathcal G^{(j)}\}|\mathfrak M^{(j)}\}\le \exp\{-c\lna\}+
 \Phi_n.
\end{equation}
 Moreover by definition of $\zeta_j$ and  \eqref{l2res}
 \begin{equation}\label{a3}
  \E\{|\zeta_j|^2\big|\mathfrak M^{(j)}\}\le \frac C{nv}\im m_n^{(j)}(z).
 \end{equation}
Inequalities \eqref{ao} -- \eqref{a3} together imply  
\begin{equation}\label{error01}
 |A_0|\le  \frac {C}{(nv)^4}\E\exp\{s\widehat\zeta^{(j)}\}.
\end{equation}

For any $\mathbb A\subset \mathbb T$ and  for $j\notin \mathbb A$, we have
\begin{align}
\widehat\zeta^{(\mathbb A)}-\widehat\zeta^{(\mathbb A,j)}&=\frac1n\widehat\zeta_j^{(\mathbb A)}+
\frac1n\sum_{l\in\mathbb T_{\mathbb A,j}}(\widehat\zeta_l^{(\mathbb A)}-
\widehat\zeta_l^{(\mathbb A,j)}).\notag
\end{align}
Furthermore,
\begin{equation}
\widehat\zeta_l^{(\mathbb A)}- \widehat\zeta_l^{(\mathbb
A,j)}=\widehat\zeta_l^{(\mathbb A)}\mathbb I\{(\mathcal
G^{\mathbb A,j)})^c\} -\widehat\zeta_l^{(\mathbb A,j)}\mathbb
I\{(\mathcal G^{\mathbb A)})^c\}+ (\zeta_l^{(\mathbb
A)}-\zeta_l^{(\mathbb A,j)})\mathbb I\{\mathcal G^{\mathbb
A,j)}\} \mathbb I\{\mathcal G^{\mathbb A)}\}.\notag
\end{equation}
Recall that
\begin{align}
\eta_{lj}^{(\mathbb A)}&=\frac2n\sum_{q\in\mathbb T_{\mathbb A,l,j}}X_{lq}X_{lj}R^{(\mathbb A,l)}_{jq},
\notag\\
\xi^{(\mathbb A)}_{lj}&=\frac1n\sum_{q\ne r\in
\mathbb T_{\mathbb A,l,j}}X_{lq}X_{lr}R^{\mathbb A,l}_{jq}R^{(\mathbb A,l)}_{jr}(R^{(\mathbb A,l)}_{jj})^{-1}\notag
\end{align}
and
\begin{equation}\notag
 \widehat\eta_{lj}^{(\mathbb A)}=\eta_{lj}^{(\mathbb A)}\mathbb I\{\mathcal G^{(\mathbb A)}\}
 \mathbb I\{\mathcal G^{(\mathbb A,j)}\},
 \quad \widehat \xi^{(\mathbb A)}_{lj}=\xi^{(\mathbb A)}_{lj}\mathbb I\{\mathcal G^{(\mathbb A)}\}
 \mathbb I\{\mathcal G^{(\mathbb A,j)}\}.
\end{equation}
We may write now
\begin{align}\label{main10}
\widehat\zeta_l^{(\mathbb A)}-
\widehat\zeta_l^{(\mathbb A,j)}=r_{lj}^{(\mathbb A)}+\widehat \xi^{(\mathbb A)}_{lj}+
\widehat\eta_{lj}^{(\mathbb A)},
\end{align}
where
\begin{equation}\notag
 r_{lj}^{(\mathbb A)}=\widehat\zeta_l^{(\mathbb A)}\mathbb I\{(\mathcal G^{\mathbb A,j)})^c\}
 -\widehat\zeta_l^{(\mathbb A,j)}\mathbb I\{(\mathcal G^{\mathbb A)})^c\}.
\end{equation}
We represent $\widehat\zeta^{(\mathbb A)}-\widehat\zeta^{(\mathbb A,j)}$ in the form
\begin{equation}\label{rep1}
 \widehat\zeta^{(\mathbb A)}-\widehat\zeta^{(\mathbb A,j)}=\frac1n\sum_{l\in\mathbb T_{\mathbb A}}\theta_{l,j}^{(\mathbb A)},
\end{equation}
where 
\begin{align}
 \theta_{l,j}^{(\mathbb A)}=\begin{cases}\widehat\zeta_j^{(\mathbb A)},\text{ if }l=j,\\ \widehat\xi_{lj}^{(\mathbb A)}+\widehat\eta_{lj}^{(\mathbb A)},
 \text{ otherwise}.\end{cases}\notag
\end{align}

In what follows we shall consider $\mathbb A=\emptyset$.
Note that  by definition
\begin{equation}\label{dif1}
 |\theta_{l,j}|\le \frac{C\sqrt{\kappa_n}}{\sqrt{nv}}\text{ a.s.}
\end{equation}
Denote by
\begin{equation}
 {\sum}^*_{j_1,\ldots,j_t}=\sum_{j_1\in\mathbb T_{j_0}}
 \sum_{j_2\in\mathbb T_{j_0,j_1}}\cdots
 \sum_{j_{t}\in\mathbb T_{j_0,\ldots,j_{t-1}}}.
\end{equation}

Using these  notations we may write  
\begin{align}\label{smu}
 \E\widehat\zeta_{j_0}(\widehat\zeta-\widehat\zeta^{(j_0)})^{\nu}\exp\{\widehat\zeta^{(j_0)}\}&
 =\frac1{n^{\nu}}{\sum}^*_{j_1,\ldots,j_{\nu}}\E\widehat\zeta^{(j_0)}
 \prod_{t=1}^{\nu}\theta_{j_t,j_0}\exp\{\widehat\zeta^{(j_0)}\}\notag\\
 +\frac1{n^{\nu}}\sum_{t=1}^{\nu-1}\sum_{\mu_1,\ldots,\mu_t\ge1:\atop\mu_1+\cdots+\mu_t=\nu}&
 {\sum}^*_{j_1,\ldots,j_t}\E\widehat\zeta_{j_0}
 \prod_{l=1}^t\theta_{j_l,j_0}^{\mu_l}\exp\{\widehat\zeta^{(j_0)}\}
 =:V_{1\nu}+V_{2\nu}.
\end{align}
\begin{lem}\label{error1}
 Assuming the conditions of Theorem $1.1$  there exist constants $c,C>0$ such that
 for $0\le s\le cnv/(\kappa_n\log n)$
 \begin{equation}
  |V_{2\nu}|\le \left(\frac{C\kappa_n}{nv}\right)^{\nu+1}g(s).\notag
 \end{equation}

\end{lem}
\begin{proof}
 Applying inequality \eqref{dif1} and H\"older's inequality for 
 conditional expectations, we get 
\begin{align}\label{error3}
|\E\widehat\zeta_{j_0}\prod_{l=1}^t\theta_{j_l,j_0}^{\mu_l}
\exp\{\widehat\zeta^{(j_0)}\}|&\le \frac{C\sqrt{\kappa_n}}{\sqrt{nv}}
\left(\frac{C\sqrt{\kappa_n}}{\sqrt{nv}}\right)^{\nu-t}\notag\\&\qquad \times\E\prod_{l=1}^t\E^{\frac1{t+1}}\{|\theta_{j_l,j_0}|^{t+1}\big|\mathfrak M^{(j_0)}\}
\exp\{s\widehat\zeta^{(j_0}\}.
\end{align}
Summing the result in $j_1,\ldots,j_t$ and applying Lemmas \ref{cond1} and \ref{cond02}, we obtain
\begin{align}\label{error4}
 V_{2\nu}
\le\frac{C\sqrt{\kappa_n}}{\sqrt{nv}}
\sum_{t=1}^{\nu-1}\left(\frac{C\sqrt{\kappa_n}}{n\sqrt{nv}}\right)^{\nu-t}
\binom{\nu-1}{ t-1}
\E Q_{j_0}\exp\{s\widehat\zeta^{(j_0)}\}&,
\end{align}
where
\begin{align}
 Q_{j_0}=\frac1{n^t}{\sum}^*_{j_1,\ldots,j_t}\prod_{l=1}^t
 (\max\{\frac{t\lna^{\frac2{\varkappa}}}{n}&((\sum_{q\in\mathbb T_{j_0,j_l}}|R^{(j_0)}_{j_l,q}|^2)^{\frac12}
  \mathbb I\{\mathcal U^{(j_0)}\}+\psi_0)+\Phi_n,\notag\\&\frac{Ct\lna^{\frac2{\varkappa}}}{nv}
  (v\sum_{t\in\mathbb T_{j_0,l}}|R^{(j_0)}_{j_l,t}|^2
  \mathbb I\{\mathcal U^{(j_0)}\}
+\psi_0)+\Phi_n\})
 .\notag
\end{align}
After summing in $j_1,\ldots,j_t$ similar to the proof of Lemma \ref{last1} we get
\begin{align}
Q_{j_0}\le \big(\frac{t\kappa_n}{nv}+\Phi_n\big)^t.
\end{align}
This inequality implies
\begin{align}\label{error5}
\E Q_{j_0}\exp\{s\widehat\zeta^{(j_0)}\} &\le2^t\E
  \left((\kappa_n)^t+\Phi_n
  \right)\exp\{s\widehat\zeta^{(j_0)}\}.
\end{align}

For $s\le c'nv/(\varkappa_n\log n)$, we get, for some constants $c',C>0$
\begin{equation}\notag
 \E Q_{j_0}\exp\{s\widehat\zeta^{(j_0)}\}\le C^t\kappa_n^tg(s).
\end{equation}
Bounding  now the sum on the right hand side of \eqref{error4} we get
\begin{align}
 V_{2,\nu}&\le \left(\frac{C\nu\kappa_n}{nv}\right)^{\nu+1}\sqrt{\frac v{\psi_0}}\,\sum_{t=0}^{\nu-1}\binom{\nu-1}{t}
 \left(\frac{Cv}{n\psi_0}\right)^{\frac{\nu-t-1}2}g(s)\notag\\&\le
 \left(\frac{C\kappa_n\log n}{nv}\right)^{\nu+1}
\sum_{t=0}^{\nu-1}\binom{\nu-1}{t}
 \left(\frac{C\sqrt v}{n}\right)^{\frac{\nu-t-1}2}g(s)\le
 \left(\frac{C\kappa_n\log n}{nv}\right)^{\nu+1}g(s) .\label{error7}
\end{align}

Thus Lemma \ref{error1} is proved.
\end{proof}
\begin{cor}\label{cor01}
Assuming the conditions of Theorem $1.1$, there exist constants $c,C>0$ such that, for $0\le s\le cnv/(\kappa_n\log n)$,
 \begin{equation}\label{error02}
  \sum_{\nu=1}^{k+1}\frac {s^{\nu}}{\nu!}|V_{2\nu}|\le \left(\frac{Cs\kappa_n^2\log n}{(nv)^2}\right)g(s).
 \end{equation}
\end{cor}
\begin{proof}
 Note that we may chose a constant $c>0$ depending on $C>0$ defined in  Lemma \ref{error1} such that
 \begin{equation}\notag
  \frac{Cs\kappa_n\log n}{nv}\le C_0<1.
 \end{equation}
We have
\begin{equation}\notag
 \sum_{\nu=1}^k\frac {s^{\nu}}{\nu!}|V_{2\nu}|\le \frac{Cs\kappa_n^2\log n}{(nv)^2}\sum_{\nu=0}^k\frac1{\nu!}C_0^{\nu}g(s)\le \frac{Cs\kappa_n^2\log n}{(nv)^2}
\end{equation}
Thus Corollary \ref{cor01} is proved.
\end{proof}
\begin{lem}\label{lastterm}
Assuming the conditions of Theorem $1.1$ there exist constants $c,C>0$ such that for $0\le s\le cnv/(\psi_0\lna^{\frac2{\varkappa}})$
 \begin{equation}\label{error03}
  |A_{k+1}|\le  \left(\frac{sC\lna^{\frac2{\varkappa}}\psi_0^2}{(nv)^2(k+1)!}\right)g(s).
 \end{equation}

\end{lem}
\begin{proof}
 First we note 
\begin{align}
  |A_{k+1}|\le \frac{s^{k+1}}{(k+1)!}\E|\widehat\zeta_{j_0}||\widehat\zeta-\widehat\zeta^{(j_0}|^{k+1}\exp\{s\widehat\zeta^{(j_0)}\}.\notag
 \end{align}
Repeating the arguments of inequalities \eqref{error3} -- \eqref{error7}, we obtain
\begin{align}\notag
 |A_{k+1}|\le \frac{s}{(k+1)!}\frac{C\lna^{\frac2{\varkappa}}\sqrt{\psi_0}}{\sqrt{nv}}\left(\frac{Cs\lna^{\frac2{\varkappa}}\psi_0}{nv}\right)^kg(s).
\end{align}
Choosing $k=c\log n$ and $s$ such that
\begin{equation}\notag
\frac{Cs\lna^{\frac2{\varkappa}}}{nv}\le\alpha<1,
\end{equation}
we get
\begin{equation}\notag
 |A_{k+1}|\le \frac{s\lna^{\frac2{\varkappa}}\psi_0}{(nv)^2(k+1)!}g(s).
\end{equation}
Thus Lemma \ref{lastterm} is proved.
\end{proof}
Inequalities \eqref{error01}, \eqref{error02} and \eqref{error03} together imply,\newline for $0\le s\le cnv/\kappa_n$ 
\begin{equation}\notag
 g'(s)\le C(\frac1{n^2v^2}+\frac {Cs\lna^{\frac2{\varkappa}}\psi_0^2}{(nv)^2})g(s)+\frac1n\sum_{j_0=1}^n\sum_{\nu=1}^k\frac{s^{\nu}}{\nu!}V_{1,\nu}.
\end{equation}

{\it Step 2.} We consider now the quantities $V_{1,\nu_1}$ for $\nu_1=1,\ldots,k$,
\begin{equation}\notag
 V_{1,\nu_1}=\frac1{n^{\nu_1}}{\sum}^*_{j_1,\ldots,j_{\nu_1}}\E\widehat\zeta_{j_0}
 \prod_{t=1}^{\nu_1}\theta_{j_t,j_0}\exp\{\widehat\zeta^{(j_0)}\}.
\end{equation}
We represent $V_{1,\nu_1}$ in the form
\begin{align}\notag
 V_{1,\nu_1}=W_{1,\nu_1}+W_{2,\nu_1},
\end{align}
where 
\begin{align}
 W_{1,\nu_1}&=\frac1{n^{\nu_1}}{\sum}^*_{j_1,\ldots,j_{\nu_1}}\E\widehat\zeta_{j_0}
 \prod_{t=1}^{\nu_1}\theta_{j_t,j_0}\exp\{\widehat\zeta^{(j_0,j_{\nu_1})}\},\notag\\
  W_{2,\nu_1}&=\frac1{n^{\nu}}{\sum}^*_{j_1,\ldots,j_{\nu_1}}\E\widehat\zeta_{j_0}
 \prod_{t=1}^{\nu_1}\theta_{j_t,j_0}P_{j_0,j_{\nu_1}}\exp\{\widehat\zeta^{(j_0,j_{\nu_1})}\},\notag
\end{align}
with
\begin{equation}
 P_{j_0,j_{\nu_1}}=(\exp\{s(\widehat\zeta^{(j_0)}
 -\widehat\zeta^{(j_0,j_{\nu_1})})\}-1).
\end{equation}

We bound $W_{2,\nu_1}$ repeating the arguments of {\it Step 1}. First we represent $W_{2,\nu_1}$ in the form
\begin{align}
 W_{2\nu_1}=\sum_{\nu_2=1}^{k-\nu_1}\frac{s^{\nu_2}}{\nu_2!}W_{2,\nu_1,\nu_2}+W_{2,\nu_1,k-\nu_1+1},\notag
\end{align}
where 
\begin{align}
W_{2,\nu_1,\nu_2}&= \frac1{n^{\nu_1}}{{\sum}^*}_{j_1,\ldots,j_{\nu_1}}\E\widehat\zeta_{j_0}
 \prod_{t=1}^{\nu_1}\theta_{j_t,j_0}
 (\widehat\zeta^{(j_0)}-\widehat\zeta^{(j_0,j_{\nu_1})})^{\nu_2}
 \exp\{\widehat\zeta^{(j_0,j_{\nu_1})}\},\notag\\
 W_{2,\nu_1,k-\nu_1+1}&=\frac1{n^{\nu_1}}{\sum}^*_{j_1,\ldots,j_{\nu_1}}\E\widehat\zeta_{j_0}
 \prod_{t=1}^{\nu_1}\theta_{j_t,j_0}\Big(\exp\{s(\widehat\zeta^{(j_0)}-\widehat\zeta^{(j_0,j_{\nu_1})})\}\notag\\&\qquad\qquad\qquad\qquad
 -\sum_{\nu_2=0}^{k-\nu_1}\frac{s^{\nu_2}}{\nu_2!}
 (\widehat\zeta^{(j_0)}-\widehat\zeta^{(j_0,j_{\nu_1})})^{\nu_2}
 \Big)\exp\{\widehat\zeta^{(j_0,j_{\nu_1})}\}.\notag
\end{align}
Similar to formula \eqref{smu} we represent
\begin{align}
 W_{2,\nu_1,\nu_2}=V_{1,\nu_1,\nu_2}+V_{2,\nu_1,\nu_2},\notag
\end{align}
where 
\begin{align}
 V_{1,\nu_1,\nu_2}&=\frac1{n^{\nu_1}}{\sum}^*_{j_1^{(1)},\ldots,j_{\nu_1}^{(1)}}\frac1{n^{\nu_2}}{\sum}^{**}_{j_1^{(2)},\ldots,j_{\nu_2}^{(2)}}\E\widehat\zeta_{j_0}
  \prod_{t=1}^{\nu_1}\theta_{j_t^{(1)},j_0}\notag\\&\qquad\qquad\qquad\qquad
  \qquad\qquad\qquad\qquad\times
  \prod_{t=1}^{\nu_2}\theta_{j_t^{(2)},j_{\nu_1}^{(1)}}
  \exp\{\widehat\zeta^{(j_0,j_{\nu_1}^{(1)})}\},\notag\\
 V_{2,\nu_1,\nu_2}&=\frac1{n^{\nu_1}} \frac1{n^{\nu_2}}{\sum}^*_{j_1^{(1)},\ldots,j_{\nu_1}^{(1)}}
 \sum_{t=1}^{\nu_2^{(2)}-1}\sum_{\mu_1,\ldots,\mu_t\ge1:\atop\mu_1+\cdots+\mu_t=\nu_2^{(2)}}\notag\\&
 \qquad\qquad\qquad\qquad\times{\sum}^{**}_{j_1^{(2)},\ldots,
 j_{\nu_2}^{(2)}}\E\widehat\zeta_{j_0}\prod_{l=1}^{\nu_1}
 \theta_{j_l^{(1)},j_0}\prod_{l=1}^t\theta_{j_l^{(2)},j_{\nu_1}^{(1)}}^{\mu_l}
 \exp\{\widehat\zeta^{(j_0)},j_{\nu_1}^{(1)}\},\notag
\end{align}
where 
\begin{align}
 {\sum}^{**}_{j_1^{(2)},\ldots,j_{\nu_2}^{(2)}}:=\sum_{j_1^{(2)}\in\mathbb T_{j_0,j_{\nu_1}^{(1)}}}\cdots
 \sum_{j_t^{(2)}\in\mathbb T_{j_0,j_{\nu_1}^{(1)},j_1^{(2)},\ldots,j_{t-1}^{(2)}}}.\notag
\end{align}

Similarly to Corollary \ref{cor01} we may prove
\begin{equation}\notag
 \sum_{\nu_1=1}^{k}\sum_{\nu_2=1}^{k-\nu_1+1}\frac {s^{\nu_1}}{\nu!}\frac{s^{\nu_2}}{\nu_2!}|V_{2,\nu_1,\nu_2}|
 \le \left(\frac{Cs\lna^{\frac2{\varkappa}}\psi_0^2}{(nv)^2}\right)g(s).
\end{equation}
Analogously to inequality \eqref{error03}, we get
\begin{equation}\notag
 |W_{2,\nu_1,k-\nu_1+1}|\le\frac{s}{\nu_1!(k-\nu_1+1)!}\frac{C\lna^{\frac2{\varkappa}}}{\sqrt{nv}}\left(\frac{Cs\lna^{\frac2{\varkappa}}\psi_0^2}{nv}\right)^kg(s).
\end{equation}
Consider now the bound for $W_{1,\nu_1}$.
We introduce the notation
\begin{align}\notag
 \Delta\zeta_{j_0,j_{\nu_1}}=\widehat\zeta_{j_0}-
 \widehat\zeta_{j_0}^{(j_{\nu_1})},\qquad \Delta\theta_{j_t,j_0,j_{\nu_1}}=\theta_{j_t,j_0}-\theta_{j_t,j_0}^{(j_{\nu_1})}.
\end{align}
We may represent
\begin{align}
 W_{1,\nu_1}&=\frac1{n^{\nu_1}}{\sum}^*_{j_1,\ldots,j_{\nu}}\E\widehat\zeta_{j_0}
 \prod_{t=1}^{\nu_1-1}\theta_{j_t,j_0}^{(j_{\nu_1})}
 \theta_{j_0,j_{\nu_1}}\exp\{\widehat\zeta^{(j_0,j_{\nu_1})}\}
 +\sum_{l=0}^{\nu_1-1}\widehat W_{1,\nu_1,l},\notag
\end{align}
where
\begin{align}
 \widehat W_{1,\nu_1,0}&=\frac1{n^{\nu_1}}{\sum}^*_{j_1,\ldots,j_{\nu}}\E\Delta\zeta_{j_0,j_{\nu_1}}\times
 \prod_{t=1}^{\nu_1}\theta_{j_t,j_0}\exp\{\widehat\zeta^{(j_0,j_{\nu_1})}\}\notag
 \end{align}
 and, for $l=1,\ldots,\nu_1-1$
 \begin{align}
 \widehat W_{1,\nu_1,l}&=\frac1{n^{\nu_1}}{\sum}^*_{j_1,\ldots,j_{\nu}}\E\widehat\zeta_{j_0}^{(j_{\nu_1})}
 \prod_{t=1}^{l-1}\theta_{j_t,j_0}^{(j_{\nu_1})}\Delta\zeta_{j_l,j_{\nu_1}}
 \prod_{t=l+1}^{\nu_1}\theta_{j_t,j_0}
 \exp\{\widehat\zeta^{(j_0,j_{\nu_1})}\}.\notag
\end{align}
\begin{lem}\label{tram1}
 Assuming the conditions of Theorem \ref{main}, there exist constants $C,c>0$ such that for any $0\le s\le cnv/\kappa_n$, for $l=0,\ldots,\nu_1-1$,
 \begin{equation}\notag
  \sum_{\nu_1=1}^k\frac{s^{\nu_1}}{\nu_1!}\sum_{l=0}^{\nu_1-1}|\widehat W_{1,\nu_1,l}|\le \frac{C\lna^{\frac2{\varkappa}} s\psi_0^2}{(nv)^2}g(s).
 \end{equation}

\end{lem}
We consider first  $\widehat W_{1,\nu_1,0}$. Applying H\"older's inequality for conditional expectations, we get
\begin{align}
 |\widehat W_{1,\nu_1,0}|\le \frac1{n^{\nu_1}}{\sum}^*_{j_1,\ldots,j_{\nu}}\E&\Big\{\E^{\frac1{\nu_1+1}}\{|\Delta\zeta_{j_0,j_{\nu_1}}|^{\nu_1+1}\big|\mathfrak M^{(j_0,j_{\nu_1})}\}
 \notag\\&\times
 \prod_{t=1}^{\nu_1}\E^{\frac1{\nu_1+1}}\{|\theta_{j_t,j_0}|^{\nu_1+1}\big|\mathfrak M^{(j_0,j_{\nu_1})}\}\exp\{\widehat\zeta^{(j_0,j_{\nu_1})}\}\Big\}.\notag
\end{align}
Similar to inequality \eqref{error4} using the results of Lemmas \ref{last1} and \ref{h3}, we get
\begin{align}
| \widehat W_{1,\nu_1,0}|\le \frac{Cs\lna^{\frac2{\varkappa}}\psi_0^2}{n^2v^2}g(s).\notag
\end{align}
Analogously we may represent  $V_{1,\nu_1,\nu_2}$ in the form 
\begin{align}
 V_{1,\nu_1,\nu_2}=\frac1{n^{\nu_1}}\frac1{n^{\nu_2}}{\sum}^*_{j_1^{(1)},\ldots,j_{\nu_1}^{(1)}}{\sum}^{**}_{j_1^{(2)},\ldots,j_{\nu_2}^{(2)}}\E\widehat\zeta_{j_0}
  \prod_{t=1}^{\nu_1}&\theta_{j_t^{(1)},j_0}
  \prod_{t=1}^{\nu_2}\theta_{j_t^{(2)},j_{\nu_1}^{(1)}}\notag\\&
  \times\exp\{\widehat\zeta^{(j_0,j_{\nu_1},\;j_{\nu_2})}\}+Q,\notag
\end{align}
where 
\begin{align}
 Q=\frac1{n^{\nu_1}}\frac1{n^{\nu_2}}{\sum}^*_{j_1^{(1)},\ldots,j_{\nu_1}^{(1)}}{\sum}^{**}_{j_1^{(2)},\ldots,j_{\nu_2}^{(2)}}\E&\widehat\zeta_{j_0}
  \prod_{t=1}^{\nu_1}\theta_{j_t^{(1)},j_0}\prod_{t=1}^{\nu_2}
  \theta_{j_t^{(2)},j_{\nu_1}^{(1)}}\notag\\&\times(\exp\{\widehat\zeta^{(j_0,j_{\nu_1})}\}-
 \exp\{\widehat\zeta^{(j_0,j_{\nu_1},\;j_{\nu_2})}\}). \notag
\end{align}
 We shall again continue to  expand $Q$ until we get $k$
factors in the expectation. This leads to the inequality 
\begin{equation}\notag
 g'(s)\le (\frac{Cs\kappa_n^2}{n^2v^2}+\frac{\psi_0^2}{n^2v^2})g(s),
 \end{equation}
for $0\le s\le cnv/(\kappa_n\log n)$
with some positive constant $c>0$.
It implies that
\begin{equation}\label{main01}
 g(s)\le C\exp\{\frac{Cs^2\kappa_n^2}{n^2v^2}\},
 \end{equation}
for $0\le s\le cnv/(\kappa_n\log n)$.

Thus Proposition \ref{mainprop} is proved.
\end{proof}

\section{Large deviations } In this Section we  bound
\begin{equation}\label{dn1}
\delta_{n}=\frac1n\sum_{j=1}^n\varepsilon_{j}.
\end{equation}
We denote by
\begin{equation}\notag
 \delta_{ni}=\frac1n\sum_{j=1}^n\varepsilon_{ji}, \text{ for }i=1,\ldots,4,
\end{equation}
and start from $\delta_{n1}$.
\begin{lem}\label{deltan1}
 There exist constants $c$ and $C$ depending on $\varkappa$ and $\alpha$  such that
\begin{equation}\notag
\Pr\{|\delta_{n1}|>n^{-1}\beta_n\}\le C \exp\{-c\4 \lna\}.
\end{equation}
\end{lem}
\begin{proof}Recall that
\begin{equation}\notag
 \delta_{n1}=\frac1{n\sqrt n}\sum_{j=1}^nX_{jj}.
\end{equation}
By Remark \ref{trunc00}, $|X_{jj}|\le C\lna^{\frac1{\varkappa}}$.
Applying McDiarmid's inequality, we get
\begin{equation}\notag
 \Pr\{|\delta_{n1}|>n^{-1}\beta_n\}\le \exp\{-c\4 \lna\}.
\end{equation}

Thus,  Lemma \ref{deltan1} is proved.
\end{proof}
\begin{cor}
 The following inequality holds
 \begin{equation}\notag
 \Pr\{|\delta_{n1}|>\frac{\beta_n\psi_0}{nv}\}\le \exp\{-c\4 \lna\}.
\end{equation}
\end{cor}
\begin{proof}
 The proof follows from the inequality $\psi_0\ge c\sqrt v$ and Lemma \ref{deltan1}.
\end{proof}

Consider now the quantity
\begin{equation} \label{dn2}
 \delta_{n2}:=\frac1{n^2}\sum_{j=1}^n\sum_{l\in\mathbb T_j}(X_{jl}^2-1)R^{(j)}_{ll}.
\end{equation}
We prove the following Lemma
\begin{lem}\label{deltan2}Let $v_0=\frac{c_0\beta_n^4}n$ with some numerical constant $c_0\ge1$.
 Under the conditions of Theorem \ref{main} there exist  constants $c$ and $C$, depending on $\varkappa$ and
 $\alpha$ only,   such that, for $v\ge v_0$,
\begin{equation} \label{root0}
 \Pr\{|\delta_{n2}|> \frac{C\lna^{\frac2{\varkappa}}(\psi_0+\im m_n(z))}{nv}\}\le C\exp\{-c\4 \lna\}.
\end{equation}
\end{lem}
\begin{proof}
Introduce the  random variables
$\xi_{jl}={ X}_{jl}^2-1$,
where $|X_{jl}|\le \4 \lnak$.
We shall bound the following quantity
\begin{align}
{\delta}_{n2}=\frac1{n^2}\sum_{j=1}^n\sum_{l\in\mathbb T_j}\xi_{jl}R^{(j)}_{ll}. \notag
\end{align}
We may represent
\begin{align}
 {\delta}_{n2}=\widehat{\delta}_{n2}+\widetilde{\delta}_{n2}+\overline \delta_{n2},\notag
\end{align}
where
\begin{align}
 \widehat{\delta}_{n2}&=-\frac1{n^2}\sum_{j=1}^n\sum_{l\in\mathbb T_j}\xi_{jl}\frac1{z+m_n(z)},\notag\\
 \widetilde{\delta}_{n2}&=
 \frac1{n^2}\sum_{j=1}^n\sum_{l\in\mathbb T_j}\xi_{jl}(R^{(j)}_{ll}+\frac1{z+m_n^{(j)}(z)}),\notag\\
 \overline \delta_{n2}&=\frac1{n^2}\sum_{j=1}^n\sum_{l\in\mathbb T_j}\xi_{jl}
 \frac{m_n^{(j)}(z)-m_n(z)}{(z+m_n(z))(z+m_n^{(j)}(z))}.\notag
\end{align}
First we note
\begin{align}
 \Pr\{|\widehat\delta_{n2}|\ge \frac{C\beta_n}{n}\}&\le \Pr\{|\widehat\delta_{n2}|\ge \frac{C\beta_n}{nv},
 |m_n(z)-s(z)|\le \frac12\}\notag\\&+
 \Pr\{|\widehat\delta_{n2}|\ge \frac{C\lnak}{n},|m_n(z)-s(z)|> \frac12\}\notag\\&\le
 \Pr\{|\frac1{n^2}\sum_{j=1}^n\sum_{l\in\mathbb T_j}\xi_{jl}|
 \ge \frac{C'\lnak}{n}\}\notag\\&+\Pr\{|m_n(z)-s(z)|> \frac12\}.\notag
\end{align}
The first term we may bound using McDiarmid's inequality and the second one may be bounded using  Corollary \ref{cor11}.
Thus we get
\begin{align}\label{delta1}
\Pr\{|\widehat\delta_{n2}|\ge \frac{C\beta_n}{n}\}&\le\exp\{-c\sqrt{nv}/\lna\}.
\end{align}

 By \eqref{repr01}, we have
\begin{equation}\label{rt}
 |R_{ll}^{(j)}+\frac1{z+m^{(j)}(z)}|\le |\varepsilon_l^{(j)}||R_{ll}^{(j)}|.\notag
\end{equation}
Furthermore, we have
\begin{equation}\label{root2}
 \varepsilon_{l2}^{(j)}=\sum_{q\in\mathbb T{j,l}}(X_{ql}^2-1)R^{(j,l)}_{qq}.
\end{equation}
Conditioning on $\mathfrak M^{(j,l)}$ and applying McDiarmid's inequality, we get
\begin{equation}\label{root01}
 \Pr\{|\varepsilon_{l2}^{(j)}|>\frac{C\lna^{\frac2{\varkappa}}}{\sqrt{n}},\mathcal B^{(j)}\}\le\exp\{-c\lna\}.
\end{equation}
Using the result of Corollary \ref{cru3} and inequality $\psi_0\ge c\sqrt v$, we arrive at
\begin{equation}\label{root}
 \Pr\{|\varepsilon_{l2}^{(j)}|>\frac{C\lna^{\frac2{\varkappa}}\sqrt{\psi_0}}{\sqrt{nv}},\mathcal B^{(j)}\}\le\exp\{-c\lna\}.
\end{equation}
Consider now
\begin{equation}\label{root3}
 \varepsilon_{l3}^{(j)}=\frac1n\sum_{q\ne r\in\mathbb T_{j,l}}X_{ql}X_{rl}R^{(l,j)}_{qr}.
\end{equation}
Representing this quantity as a  martingale similar to \eqref{martingale1} and repeating the arguments \eqref{martingale2} -- \eqref{martingale5} with
$c\sqrt{\log n}$ instead of $\sqrt{nv_k}$ in the inequality \eqref{martingale3} and $x=\gamma_n\sqrt{\log n}$ instead $x=\gamma_n\sqrt{nv}$,
 we get
\begin{equation}\label{root4}
\Pr\{|\varepsilon_{l3}^{(j)}|>\frac {C\beta_n}{\sqrt {nv}}(\im m_n^{(j)}(z))^{\frac12}\}\le \exp\{-c\lna\}.
\end{equation}

 Furthermore,
 note that
 \begin{equation}\label{root5}
  |\varepsilon^{(j)}_{l1}|\le \frac{C\lna}{\sqrt n}\le \frac{C\lna\sqrt{\psi_0}}{\sqrt{nv}}\le C\sqrt{\frac{\kappa_n}{nv}}\text{ a. s.},
 \end{equation}
and
\begin{equation}\label{root6}
 |\varepsilon^{(j)}_{l4}|\le \frac1{nv}\le C\sqrt{\frac{\kappa_n}{nv}}\text{ a.s.}
\end{equation}
Inequalities \eqref{root2} -- \eqref{root6} together imply
\begin{equation}\label{ro1}
 \Pr\{|\varepsilon^{(j)}_l|\ge\frac{C\lna^{\frac2{\varkappa}}(\sqrt{\im m_n^{(j)}}+\sqrt{\psi_0})}{\sqrt{nv}}\}\le
 \exp\{-c\lna\}.
\end{equation}
Applying now McDiarmid's inequality, we get
\begin{align}
 \Pr\{|\widetilde\delta_{n2}|\ge \frac{C\lna^{\frac2{\varkappa}}(\sqrt{\im m_n(z)}+\sqrt{\psi_0})}
 {\sqrt n\sqrt{nv}}\}\le
 \exp\{-c\lna\}.
\end{align}
Taking into account that $\psi_0\ge\sqrt v$, we get
\begin{equation}
 \Pr\{|\widetilde\delta_{n2}|\ge \frac{C\lna^{\frac2{\varkappa}}(\sqrt{\im m_n(z)}+\sqrt{\psi_0})\sqrt{\psi_0}}
 {\sqrt n\sqrt{nv}}\}\le
 \exp\{-c\lna\}.
 \end{equation}
 From here, using that $\frac1{\sqrt n}=\frac{\sqrt v}{\sqrt{nv}}\le \psi_0$ and that
$\sqrt{\psi_0\im m_n(z)}\le \frac12(\psi_0+\im m_n(z))$, we get
\begin{equation}
 \Pr\{|\widetilde\delta_{n2}|\ge \frac{C\lna^{\frac2{\varkappa}}(\im m_n(z)+\psi_0)}{\sqrt n\sqrt{nv}}\}\le
 \exp\{-c\lna\}.
 \end{equation}
 Note that
\begin{equation}
 |m_n(z)-m_n^{(j)}(z)|\le \frac1{nv}.
\end{equation}
This yields  that
\begin{equation}\label{line}
 \Pr\{|\overline\delta_{n2}|\ge\frac{C\beta_n^2}{n^{\frac32}v}\}
 \le \sum_{j\in\mathbb T}\Pr\{\frac1n|\sum_{l\in\mathbb T_j}\xi_{lj}|
 \ge\frac{C\beta_n^2}{\sqrt{n}}\}\le
 \exp\{-lna\}.
\end{equation}
From here, using that $\frac1{\sqrt n}\le \frac{\sqrt v}{\sqrt{nv}}\le \psi_0$ and that
$\sqrt{\psi_0\im m_n(z)}\le \frac12(\psi_0+\im m_n(z))$, we get

\begin{equation}
 \Pr\{|\overline\delta_{n2}|\ge\frac{C\beta_n^2\psi_0}{nv}\} \le\exp\{-lna\}.
\end{equation}

Thus Lemma \ref{deltan2} is proved.

\end{proof}

Let now
\begin{equation} \label{dn3}
 \delta_{n3}:=\frac1{n^2}\sum_{j=1}^n\sum_{q\ne r\in \mathbb T_j}X_{jq}X_{jr}R^{(j)}_{qr}=\frac1n\sum_{j=1}^n\zeta_j.
\end{equation}
\begin{lem}\label{deltan3}Let $v_0=\frac{c_0\beta_n^4}n$ with some numerical constant $c_0>1$.
Assuming the conditions of Theorem \ref{main}
 there exist constants $c$ and $C$, depending on $\varkappa$,  $\alpha$ only such that, for $v\ge v_0$,
\begin{equation} \notag
 \Pr\{|\delta_{n3}|> \frac{4\beta_n^4\psi_0}{n v}\}\le C\exp\{-c\4 \lna\}.
\end{equation}
\end{lem}
\begin{proof} By Lemma \ref{trunczeta},
\begin{equation}\notag
\Pr\{\delta_{n3}\ne \frac1n\sum_{j=1}^n\widehat\zeta_j\}\le\exp\{-c\lna\}.
\end{equation}
Applying Chebyshev's inequality, we obtain
\begin{align}
 \Pr\{|\frac1n\sum_{j=1}^n\widehat\zeta_j|>x\}\le \exp\{-sx\}
 \E\exp\{s\frac1n\sum_{j=1}^n\widehat\zeta_j\}=\exp\{-sx\}g(s).\notag
\end{align}
We choose $s=cnv/(\kappa_n\log n)$ and $x=\frac{c\kappa_n\log^4 n}{nv}$.
Then, by Proposition  \ref{mainprop}, we get
\begin{equation}
 \Pr\{|\frac1n\sum_{j=1}^n\widehat\zeta_j|>x\}\le\exp\{-c\lna\}.
\end{equation}

Thus the  Lemma is proved.
\end{proof}
Finally, we shall bound
\begin{equation} \label{dn4}
\delta_{n4}:=\frac1{n^2}\sum_{j=1}^n(\Tr\mathbf R-\Tr\mathbf R^{(j)})R_{jj}.
\end{equation}
\begin{lem}\label{deltan4}
 For any $z=u+iv$ with $v>0$ the following inequality
\begin{equation}
|\delta_{n4}|\le \frac1{nv}\im m_n(z) \text{ a. s.}
\end{equation}
holds.
\end{lem}
\begin{proof}
 By formula (5.4) in \cite{GT:03}, we have
\begin{align} \notag
 (\Tr\mathbf R-\Tr\mathbf R^{(j)})R_{jj}=(1+\frac1n\sum_{l,k\in T_j}X_{jl}X_{jk}(R^{(j)})^2_{lk})R_{jj}^2
=\frac{d}{dz}R_{jj}.
\end{align}
From here it follows that
\begin{equation}\notag
 \frac1{n^2}\sum_{j=1}^n(\Tr\mathbf R-\Tr\mathbf R^{(j)})R_{jj}=\frac1{n^2}\frac{d}{dz}\Tr\mathbf R=\frac1{n^2}\Tr \mathbf R^2.
\end{equation}
Finally, we note using \eqref{l2res} that
\begin{equation} \notag
 |\frac1{n^2}\Tr \mathbf R^2|\le \frac1{nv}\im m_n(z).
\end{equation}
The last inequality concludes the proof.
Thus,  Lemma \ref{deltan4} is proved.
\end{proof}

\section{Stieltjes transforms} We shall derive auxiliary bounds for the difference  between the Stieltjes transforms
$m_n(z)$ of the empirical
spectral measure
  of the matrix $\mathbf X$ and the Stieltjes transform $s(z)$ of the semi-circular law. Recalling the definitions of
 $\varepsilon_j, \varepsilon_{j\mu}$ in \eqref{epsjn} and of $\delta_{n\nu}$   in \eqref{dn1},\eqref{dn2}, \eqref{dn3} as well as \eqref{dn4},
 we introduce the additional notations
\begin{align}
 \delta_n':=\delta_{n1}+\delta_{n2}+\delta_{n3},\quad\widehat{\delta}_n:=\delta_{n4},\quad
 \overline
{\delta}_n:=\frac1n\sum_{\nu=1}^3
\sum_{j=1}^n
\varepsilon_{j\nu}\varepsilon_jR_{jj}. \label{dnp}
\end{align}
Recall that $r_n(z):=m_n(z)-s(z)$.
The representation \eqref{repr03}  implies
\begin{equation}\label{repr10}
 r_n(z)=\frac{r_n(z)}{(z+s(z))(z+m_n(z))}-\frac{\delta_n'}{(z+m_n(z))^2}+\frac{\widehat{\delta}_n}{z+m_n(z)}+\frac{\overline{\delta}_n}
{(z+m_n(z))^2}.
\end{equation}
The equality  \eqref{repr10}  yields
\begin{align}\label{repr101}
 |r_n(z)|&\le \frac{|\delta_n'|+|\overline{\delta}_n|}{|z+m_n(z)||z+s(z)+m_n(z)|}+\frac{|\widehat{\delta}_n|}{|z+s(z)+m_n(z)|}.
\end{align}
By Lemmas \ref{deltan1}, \ref{deltan2}, \ref{deltan3} and \ref{deltan4}, we get
\begin{equation}\label{fin1}
 \Pr\{|r_n(z)|\le \frac{C\lna^{\frac4{\varkappa}}\psi_0}{nv|z+m_n(z)+s(z)|}+\frac{\lna^{\frac4{\varkappa}}}{nv}\}\ge
 1-\exp\{-c\lna\}-\Pr\{\mathcal U^c\}.
\end{equation}
We prove now following
\begin{prop}\label{prop10}
 Assuming the assumption of Theorem \ref{main}, the following inequality holds
 \begin{equation}
  \Pr\{|r_n(z)|\ge \frac{C\beta_n^4}{nv}+\frac{C\beta_n^4}{n^2v^2|z+m_n(z)+s(z)|}\}\le \exp\{-c\lna\}.
 \end{equation}

\end{prop}
\begin{proof}
First we put $\psi_0=\frac32$. Since $\mathcal A_{v_0}\subset \mathcal U$ by Corollary \ref{cor11} we get
\begin{equation}
 \Pr\{\mathcal U^c\}\le\exp\{-c\lna\}.
\end{equation}
By inequality \eqref{fin1}, we have
\begin{equation}
 \Pr\{|r_n(z)|\le \frac{C\lna^{\frac4{\varkappa}}}{|z+m_n(z)+s(z)|nv}\}\ge1-\exp\{-c\lna\}.
\end{equation}
Note that $|z+m_n(z)+s(z)|\ge \im m_n(z)$.
Assume that
\begin{equation}
 \im m_n(z)\ge \im s(z)+\frac{C^{\frac12}\lna^{\frac2{\varkappa}}}{\sqrt {nv}}.
\end{equation}
Then we get
\begin{equation}
 |r_n(z)|\le \frac{\sqrt C\lna^{\frac2{\varkappa}}}{\sqrt {nv}}.
\end{equation}

This inequality implies that
 \begin{equation}
 \Pr\{\im m_n(z)\ge\im s(z)+\frac{C\lna^{\frac2{\varkappa}}}{\sqrt{nv}}\}\le
  \exp\{-c\lna\}.
 \end{equation}
Put now $\psi_0=\im s(z)+\frac{C\lna^{\frac2{\varkappa}}}{\sqrt{nv}}$. Note that
 for $|u|\le 2$, $\im s(z)\ge c\sqrt v$, and for $v\ge v_0$, $\psi_0\ge \frac1{nv}$.
 That means, that we may apply inequality \eqref{fin1} with 
 $\psi_0=\im s(z)+\frac {C\lna^{\frac1{\varkappa}}}{\sqrt{nv}}$.
 Using that $|z+s(z)+m_n(z)|\ge \im s(z)$, we get
\begin{equation}
  \im m_n(z)\le \im s(z)+\frac{C\lna^{\frac2{\varkappa}}}{n^2v^2\sqrt{\im m_n(z)}}+
 \frac{C\lna^{\frac2{\varkappa}}}{nv}.
\end{equation}
 From this inequality it follows that
 \begin{equation}
 \im m_n(z)\le \im s(z)+
  \frac{C\lna^{\frac2{\varkappa}}}{nv},
 \end{equation}
 with some other constant $C$.
 We put now $\psi_0=\im s(z)+
  \frac{C\lna^{\frac2{\varkappa}}}{nv}$. Thus  have proved that
  \begin{equation}
   \Pr\{\mathcal U^c\}\le \exp\{-c\lna\}.
  \end{equation}
Applying inequality \eqref{fin1} again, Proposition
 \ref{prop10} is proved.
\end{proof}

Note that using a union bound we may prove that
\begin{align}
 \Pr\Big\{|r_n(z)|&\le \frac{C\lna^{2{\varkappa}}\psi_0}{|z+m_n(z)+s(z)|}+\frac{\lna^{\frac2{\varkappa}}}{nv},
 \text{ for any }
 z=u+iv, |u|\le 2, V\ge v\ge v_0\Big\}\notag\\&\le \exp\{-c\lna\}.
\end{align}

\section{Proof of Theorem \ref{main}}To conclude the proof of Theorem \ref{main}
 we shall now apply the result of Corollary \ref{smoothing1} with $ v_0=\frac{C_0\beta_n^4}{n}$ and $V=4$
 to the empirical spectral
 distribution function
$\mathcal F_n(x)$ of the random matrix $\mathbf X$.
 At first we bound the integral over the line  $V=4$. Note that in this case we have $|z+m_n(z)|\ge1$ and
$|r_n(z)|\le \frac 12$ a.s.
 Moreover, $\im m_n^{(j)}(z)\le \frac1V\le \frac12$. In this case the  event $\mathcal B$ implies
 \begin{equation}
 |R_{jj}^{(\mathbb A)}|^{-1}\le C|u|,
 \end{equation}
 and
 \begin{equation}
  |R_{jj}^{(\mathbb A}|\le \frac {C}{|u|},
 \end{equation}
 for sufficiently large $u$ uniformly in $n$.
We may prove that, for $u\in \mathbb R$
\begin{equation}
 \Pr\{|r_n(z)|\ge \frac{C\beta_n^2}{n(|u|^2+1)}\}\le\exp\{-c\lna\}.
\end{equation}
This inequality implies that
\begin{equation}\label{daleko}
 \Pr\{\int_{-\infty}^{\infty}|r_n(u+iV)|du\le \frac{C\beta_n^2}{n}\}\ge 1-\exp\{-c\lna\}.
\end{equation}
Consider now $u\in\mathbb J_{\varepsilon}$, where $\varepsilon=v_0^{\frac23}$. By Proposition \ref{prop10},
we have
\begin{equation}\label{fin10}
 \Pr\{|r_n(z)|\le \frac{C\lna^{2{\varkappa}}}{n^2v^2|z+m_n(z)+s(z)|}+\frac{\lna^{\frac2{\varkappa}}}{nv}\}\ge
 1-\exp\{-c\lna\}.
\end{equation}
Note that
\begin{equation}
 |z+m_n(z)+s(z)|\ge \im(z+s(z))\ge c\sqrt{\gamma+v},
\end{equation}
where $\gamma=2-|u|$. Integrating now in $v\in [v_0/\sqrt{\gamma}, V]$, we get
\begin{equation}
 \int_{v_0/\sqrt{\gamma}}^V|r_n(u+iv)|dv\le
 \frac{C\beta_n^2\log n}{n}+\frac{C\beta_n^2}{n^2v_0}.
\end{equation}

Thus, Theorem \ref{main} is proved.
\vskip -0.5cm
\section{Proof of Theorem \ref{eigenvector}} We may express the  diagonal entries of the resolvent matrix
$\mathbf R$ as follows
\begin{equation}\label{representeigenvector}
 R_{jj}=\sum_{k=1}^n\frac1{\lambda_k-z}|u_{jk}|^2.
\end{equation}
Consider the distribution function, say $F_{nj}(x)$, of the probability distribution of the eigenvalues  $\lambda_k$
\begin{equation} \notag
 F_{nj}(x)=\sum_{k=1}^n|u_{jk}|^2\mathbb I\{\lambda_k\le x\}.
\end{equation}
Then we have
\begin{equation} \notag
 R_{jj}= R_{jj}(z)=\int_{-\infty}^{\infty}\frac1{x-z}dF_{nj}(x),
\end{equation}
which means  that $R_{jj}$ is the Stieltjes transform of the distribution $F_{nj}(x)$.
Note that, for any $\lambda>0$,

\begin{equation} \notag
 \max_{1\le k\le n}|u_{jk}|^2\le \sup_x(F_{nj}(x+\lambda)-F_{nj}(x))=:Q_{nj}(\lambda).
\end{equation}
On the other hand, it is easy to check that
 \begin{equation}\label{concentration}
 Q_{nj}(\lambda)\le 2\sup_u\lambda \4\im R_{jj}(u+i\lambda).
\end{equation}
By Corollary \ref{cru1} we obtain, for any $v\ge v_0$
with $v_0=\frac{C_0\beta_n^4}{n}$ with a sufficiently large
constant $C_0$, and for any $u\in\mathbb R$,
\begin{equation}\label{eigenvector1}
 \Pr\Bigl\{|R_{jj}|\ge 3\}\le C\exp\{-c \4 \lna\},
\end{equation}
with constants $C$ and $c$ depending on $\varkappa$, $\alpha$. From here using a  union bound it is straightforward to check that
\begin{equation}
 \Pr\{\sup_u|R_{jj}(u+iV_0)|\ge 3\}\le C\exp\{-c \4 \lna\},
\end{equation}
with some other constants $C,c>0$. Putting $\lambda=v_0$ in the inequality \eqref{concentration}, we obtain
\begin{equation}
 \Pr\{\max_{1\le k\le n}|u_{jk}|^2\ge \frac{C\beta_n^4}{n}\}\le C\exp\{-c \4 \lna\}.
\end{equation}
By an additional union bound we arrive at the inequality \eqref{deloc}.
To prove inequality \eqref{deloc1}, we consider the quantity
$ \rho_j:=R_{jj}-s(z)$.
Using equalities \eqref{repr01} and  \eqref{stsemi}, we get
\begin{equation} \notag
 \rho_j=-\frac{s(z)r_n(z)}{z+m_n(z)}+\frac{\varepsilon_j}{z+m_n(z)}R_{jj}.
\end{equation}
By  Proposition \ref{prop10},  and \eqref{ro1}, we have, for $v\ge v':=Cn^{-\frac12}$
\begin{equation} \notag
 \Pr\{|\rho_j|\le \frac{c\beta_n^2}{\sqrt{nv}}\}\ge 1-C\exp\{-c\lna\}.
\end{equation}
From here it follows that
\begin{equation} \notag
 \sup_{x\in\mathbb J_{\varepsilon}}\int_{v'}^V|r_j(x+iv)|dv\le \frac C{\sqrt n}.
\end{equation}
Similar to \eqref{daleko} we get
\begin{equation} \notag
 \int_{-\infty}^{\infty}|\rho_j(x+iV)|dx\le \frac {C\beta_n^2}{\sqrt n}.
\end{equation}
Applying Corollary \ref{smoothing1}, we get
\begin{equation} \notag
 \Pr\{\sup_x|F_{nj}(x)-G(x)|\le \frac{\beta_n^2}{\sqrt n}\}\ge 1-C\exp\{-c\4 \lna\}.
\end{equation}
Using now that
\begin{equation} \notag
\Pr\Bigl\{\sup_x|\mathcal F_{n}(x)-G(x)|\le \frac{\beta_n^4\log n}{ n}\Bigr\}\ge 1-C\exp\{-c\4 \lna\},
\end{equation}
we get
\begin{equation} \notag
\Pr\Bigl\{\sup_x|F_{nj}(x)-\mathcal F_{n}(x)|\le \frac{\beta_n^2}{\sqrt n}\Bigr\}\ge 1-C\exp\{-c\4\lna\}.
\end{equation}
Thus, Theorem \ref{eigenvector} is proved.
\section{Appendix}
\subsection{Proof of Remark \ref{localization}}\label{remark}
\begin{Proof of} {\it Remark \ref{localization}.} Let $x\in[0,\frac12]$.
Denote  by $\tau$  a random variable which is uniformly distributed on $[0,1]$ and use
  Taylor's formula to show
 \begin{equation}\label{tailor}
  G^{-1}(x)=-2+ \E_{\tau}\frac{2\pi x}{\sqrt{4-(G^{-1}(x\tau))^2}}.
 \end{equation}
 By monotonicity of $G^{-1}(x)$, we get
 \begin{equation}
  2+G^{-1}(x)\ge \frac{C x}{\sqrt{4-(G^{-1}(x))^2}}.
 \end{equation}
 Using $ \sqrt{4-y^2}= \sqrt{2+y}\sqrt{2-y},\; y=G^{-1}(x)$, there is another
  absolute constant $C>0$ such that
 \begin{equation}
  2+G^{-1}(x)\ge C x^{\frac23}.
 \end{equation}
From the last inequality we get
\begin{equation} \label{onethird}
\sqrt{2+ G^{-1}(\tau x)} \ge c (\tau x)^{1/3}
\end{equation}
and hence by  \eqref{tailor} it follows that
\begin{equation}
 2+G^{-1}(x)\le  c'x^{\frac23}\E_{\tau}\frac1{\tau^{\frac13}}\le C_2x^{\frac23},
\end{equation}
with some absolute constants $c',C_2>0$.
Similarly  for $x\in[\tfrac 12,1]$ we get
\begin{equation}
 C_1(1-x)^{\frac23}\le 2-G^{-1}(x)\le C_2(1-x)^{\frac23}.
\end{equation}
Summarizing, we may write for another absolute constant $C_1$,
\begin{equation}\label{new1}
 C_1\min\{x^{\frac23},(1-x)^{\frac23}\}\le 4-(G^{-1}(x))^2\le C_1\min\{x^{\frac23},(1-x)^{\frac23}\}.
\end{equation}
Furthermore,
 \begin{equation}
  \Delta_n^*=\sup_x|\mathcal F_n(x)-G(x)|=\max_{1\le k\le n}|\mathcal F_n(\lambda_k)-G(\lambda_k)|=
  \max_{1\le k\le n}|\frac kn-G(\lambda_k)|.
 \end{equation}
This implies that, for $\lambda_k\in[-2,2]$ and $\abs{\theta} \le 1$,
\begin{equation}
 \lambda_k=G^{-1}(\frac kn+\theta\Delta_n^*).
\end{equation}
By Taylor's formula we have
\begin{equation}
 G^{-1}(\frac kn+\theta\Delta_n^*)=G^{-1}(\frac kn)+\E_{\tau}\frac{2\pi\theta\Delta_n^*}
 {\sqrt{4-(G^{-1}(\frac kn+\tau\theta\Delta_n^*))^2}}.
\end{equation}
Consider first the case $2\Delta_n^*\le \frac kn\le \frac12-\Delta_n^*$.
Then by  \eqref{new1},
\begin{equation}
 \sqrt{4-(G^{-1}(\frac kn+\tau\theta\Delta_n^*))^2}\ge C|\frac kn+\tau\theta\Delta_n^*|^{\frac13}\ge
 C'(\frac kn)^{\frac13}.
\end{equation}
From here it follows that by Theorem \ref{main},  with probability
$1-C \exp\{-c l_{n,\alpha}\}$,
\begin{equation}
 |\lambda_k-\gamma_{nk}|\le C\beta_n^4n^{-\frac23}k^{-\frac13}.
\end{equation}
Similar we get, for $2\Delta_n^*\le \frac{n-k}n\le \frac12-\Delta_n^*$,
\begin{equation}
 |\lambda_k-\gamma_{nk}|\le C\beta_n^4n^{-\frac23}(n-k)^{-\frac13}.
\end{equation}

Thus Remark \ref{localization} is proved.
  \end{Proof of}
\subsection{Proof of Bounds for Kolmogorov Distance}\label{kolmdistance}
\begin{Proof of} {\it Proposition \ref{smoothing}}.
The proof of Proposition \ref{smoothing} is a straightforward adaptation  of  the proof of Lemma 2.1 from \cite{GT:03}. We include it here for the
sake of completeness. First we note that
\begin{align}
 \sup_x|F(x)-G(x)|&=\sup_{x\in[-2,2]}|F(x)-G(x)|=\max\Bigg\{\sup_{x\in\mathbb J_{\varepsilon}}|F(x)-G(x)|,\notag\\&
\sup_{x\in[-2,-2+\varepsilon]}|F(x)-G(x)|,\sup_{x\in[2-\varepsilon,2]}|F(x)-G(x)|\Bigg\}.
\end{align}
Furthermore, for $x\in[-2,-2+\varepsilon]$ we have
\begin{align}
 -G(-2+\varepsilon)\le F(x)-G(x)\le F(-2+\varepsilon)-G(-2+\varepsilon)+G(-2+\varepsilon)\notag\\
\le \sup_{x\in\mathbb J_{\varepsilon}}|F(x)-G(x)|+
G(-2+\varepsilon).
\end{align}
This inequality yields
\begin{equation}
\sup_{x\in[-2,-2+\varepsilon]}|F(x)-G(x)|\le \sup_{x\in\mathbb J_{\varepsilon}}|F(x)-G(x)|+G(-2+\varepsilon).
\end{equation}
Similarly we get
\begin{equation}
\sup_{x\in[2-\varepsilon,2]}|F(x)-G(x)|\le \sup_{x\in\mathbb J_{\varepsilon}}|F(x)-G(x)|+1-G(2-\varepsilon).
\end{equation}
Note that $G(-2+\varepsilon)=1-G(2-\varepsilon)$ and
$G(-2+\varepsilon)\le C\varepsilon^{\frac32}$  with some absolute constant $C>0$.
Combining all these relations  we get
\begin{equation}\label{gav1}
 \sup_x|F(x)-G(x)|\le\Delta_{\varepsilon}(F,G)+C\varepsilon^{\frac32},
\end{equation}
where $\Delta_{\varepsilon}(F,G)=\sup_{x\in\mathbb J_{\varepsilon}}|F(x)-G(x)|$.
We denote $v'=\frac v{\sqrt{\gamma}}$. For any $x\in \mathbb J'_{\varepsilon}$
\begin{align}\label{smoothm}
 \big|\frac1{\pi}&\im\Big(\int_{-\infty}^x(S_F(u+iv')-S_G(u+iv'))du\Big)\big|\notag\\&\ge \frac1{\pi}
\im\Big(\int_{-\infty}^x(S_F(u+iv')-S_G(u+iv'))du\Big)
\notag\\&=\frac1{\pi}\left[\int_{-\infty}^{x}\int_{-\infty}^{\infty}\frac{v'd(F(y)-G(y))}{(y-u)^2+{v'}^2}du\right]\notag\\&=
\frac1{\pi}\int_{-\infty}^x\left[\int_{-\infty}^{\infty}\frac{2v'(y-u)(F(y)-G(y))dy}{((y-u)^2+{v'}^2)^2}\right]\notag\\&=
\frac1{\pi}\int_{-\infty}^{\infty}(F(y)-G(y))\left[\int_{-\infty}^x\frac{2v'(y-u)}{((y-u)^2+{v'}^2)^2}du\right]dy\notag\\&=
\frac1{\pi}\int_{-\infty}^{\infty}\frac{F(x-v'y)-G(x-v'y)}{y^2+1} dy, \q \text{by change of variables}.
\end{align}
Furthermore, using \eqref{constant} and the definition of $\Delta(F,G)$ we note that
\begin{equation}\label{int1}
\frac1{\pi} \int_{|y|>a}\frac{|F(x-v'y)-G(x-v'y)|}{y^2+1}dy\le (1-\beta)\Delta(F,G).
\end{equation}
Since $F$ is non decreasing, we have
\begin{align}
 \frac1{\pi}\int_{|y|\le a}&\frac{F(x-v'y)-G(x-v'y)}{y^2+1}dy\ge \frac1{\pi}\int_{|y|\le a}\frac{F(x-v'a)-G(x-v'y)}{y^2+1}dy
\notag\\&\ge
(F(x-v'a)-G(x-v'a))\beta\notag\\&\q\q-\frac1{\pi}\int_{|y|\le a}|G(x-v'y)-G(x-v'a)|dy.
\end{align}
These inequalities together imply (using a change of variables in the last step)
\begin{align} \label{smoth1}
 \frac1{\pi}&\int_{-\infty}^{\infty}\frac{F(x-v'y)-G(x-v'y)}{y^2+1}dy\ge
\beta(F(x-v'a)-G(x-v'a))\notag\\&-\frac1{\pi}\int_{|y|\le a}|G(x-v'y)-G(x-v'a)|dy-(1-\beta)\Delta(F,G)\notag\\&\ge
\beta(F(x-v'a)-G(x-v'a))\notag\\&-\frac1{v'\pi}\int_{|y|\le v'a}|G(x-y)-G(x-v'a)|dy-(1-\beta)\Delta(F,G) .
\end{align}
Note that according to Remark \ref{rem2.2}, $x\pm v'a\in\mathbb J'_{\varepsilon}$ for any $x\in\mathbb J_{\varepsilon}$.
 Assume first that $x_n\in\mathbb J_{\varepsilon}$ is a
sequence  such that
$F(x_n)-G(x_n)\to\Delta_{\varepsilon}(F,G)$. Then $x_n':=x_n+v'a\in\mathbb J'_{\varepsilon}$.
Using \eqref{smoothm} and \eqref{smoth1}, we get
\begin{align}\label{gav20}
 \sup_{x\in\mathbb J'_{\varepsilon}}&\left|\im\int_{-\infty}^x(S_F(u+iv')-S_G(u+iv'))du\right|\notag\\&\ge
\im\int_{-\infty}^{x_n'}(S_F(u+iv')-S_G(u+iv'))du\notag\\&\ge\beta(F(x_n'-v'a)-G(x_n'-v'a))\notag\\&\q\q-\frac1{\pi v}
\sup_{x\in\mathbb J'_{\varepsilon}}{\sqrt{\gamma}}\int_{|y|\le 2v'a}|G(x+y)-G(x)|dy-(1-\beta)\Delta(F,G) \notag
\\&=\beta(F(x_n)-G(x_n))\notag\\&-\frac1{\pi v}
\sup_{x\in\mathbb J'_{\varepsilon}}{\sqrt{\gamma}}\int_{|y|<2v'a}|G(x+y)-G(x)|dy-(1-\beta)\Delta(F,G) .
\end{align}
Assume for definiteness  that $y>0$. Recall that $ \varepsilon\le 2\gamma$, for any $x\in\mathbb J_{\varepsilon}'$.
By Remark  \ref{rem2.2} with $\varepsilon/2$ instead of $\varepsilon$, we have $0<y\le 2v'a\le\sqrt 2\varepsilon$, for any
$x\in\mathbb J'_{\varepsilon}$.
 For the semi-circular law we have,
\begin{align}
 |G(x+y)-G(x)|&\le y\sup_{u\in[x,x+y]}G'(u)\le y\4C\4\sqrt{\gamma+y}\notag\\&\le C\4 y\sqrt{\gamma+2v'a}
\le C\4 y\4\sqrt{\gamma+\varepsilon}\le Cy\sqrt{\gamma}.
\end{align}
This yields after integrating in $y$
\begin{align}\label{gav3}
\frac1{\pi v} \sup_{x\in\mathbb J'_{\varepsilon}}{\sqrt{\gamma}}\int_{0\le y\le 2v'a}|G(x+y)-G(x)|dy
\le\frac C{v} \sup_{x\in\mathbb J'_{\varepsilon}}\gamma{v'}^2\le Cv.
\end{align}
Similarly we get that
\begin{align}\label{gav3*}
\frac1{\pi v} \sup_{x\in\mathbb J'_{\varepsilon}}{\sqrt{\gamma}}\int_{0\ge y\ge -2v'a}|G(x+y)-G(x)|dy
\le\frac C{v} \sup_{x\in\mathbb J'_{\varepsilon}}\gamma{v'}^2\le Cv.
\end{align}
By inequality \eqref{gav1}
\begin{equation}\label{gt1}
 \Delta_{\varepsilon}(F,G)\ge \Delta(F,G)-C\varepsilon^{\frac32}.
\end{equation}
The  inequalities
 \eqref{gav20}, \eqref{gt1} and \eqref{gav3}, \eqref{gav3*} together yield as $n$ tends to infinity
\begin{align}
 \sup_{x\in\mathbb J'_{\varepsilon}}&\left|\im\int_{-\infty}^x(S_F(u+iv')-S_G(u+iv'))du\right|\notag\\&
\ge(2\beta-1)\Delta(F,G)-C\4v -C\varepsilon^{\frac32}
,\label{Deltaeps}
\end{align}
for some constant $C>0$. Similar arguments may be used to prove this inequality
in case  there is a sequence $x_n\in\mathbb J_{\varepsilon}$ such $F(x_n)-G(x_n)\to-\Delta_{\varepsilon}(F,G)$.
In view of \eqref{Deltaeps} and $2\beta-1=1/2$ this completes the proof.
\end{Proof of}
\subsection{Proof of Lemma \ref{trunc}}
\begin{proof} First we consider $m_n(z)-\widehat m_n(z)$.
 Denote by
 \begin{equation}
  \widehat{\mathbf R}=(\widehat W-z\mathbf I)^{-1}.
 \end{equation}
We have
\begin{equation}
 m_n(z)-\widehat m_n(z)= \frac1n\Tr \mathbf R(\mathbf W-\widehat{\mathbf W})\widehat{\mathbf R}.
\end{equation}
This representation and inequality $\max\{\|\mathbf R\|,\|\widehat{\mathbf R}\|\}\le v^{-1}$ imply
\begin{equation}
 |m_n(z)-\widehat m_n(z)|\le \frac 1{\sqrt nv^2}\|\mathbf W-\widehat{\mathbf W}\|_2=v^{-2}\left(\frac1{n^2}\sum_{j,l=1}^n|X_{jl}-\widehat X_{jl}|^2\right)^{\frac12}.
\end{equation}
From here it follows that
\begin{equation}\label{trunc1}
 \Pr\{|m_n(z)-\widehat m_n(z)|>\frac{C}{n^2v^2}\}\le \sum_{j,l=1}^n\Pr\{|X_{jl}-\widehat X_{jl}|>\frac{C}{n^2}\}.
\end{equation}
Note that
\begin{equation}
 X_{jl}-\widehat X_{jl}=X_{jl}\mathbb I\{|X_{jl}|\ge C\lna^{\frac1{\varkappa}}\}-\E X_{jl}\mathbb I\{|X_{jl}|\ge C\lna^{\frac1{\varkappa}}\}.
\end{equation}
Condition \eqref{exptails} implies that
\begin{equation}
 |\E X_{jl}\mathbb I\{|X_{jl}|\ge C\lna^{\frac1{\varkappa}}\}|\le \exp\{-c\lna\}\le \frac {C}{2n^2}.
\end{equation}
From here it follows that
\begin{equation}\label{trunc3}
 \Pr\{|X_{jl}-\widehat X_{jl}|>\frac{C}{n^2}\}\le \Pr\{|X_{jl}|\ge C\lna^{\frac1{\varkappa}}\}\le A\exp\{-c\lna\}.
\end{equation}
Inequalities \eqref{trunc1} and \eqref{trunc3} together imply that there exists a constant $c'$ such that
\begin{equation}\label{trunc01}
 \Pr\{|m_n(z)-\widehat m_n(z)|>\frac{C}{n^2v^2}\}\le \exp\{-c'\lna\}.
\end{equation}

We prove now that
\begin{equation}\label{trunc4}
 \Pr\{|\widetilde m_n(z)-\widehat m_n(z)|\ge \frac C{n^2v^2}\}\le \exp\{-c\lna\}.
\end{equation}
Repeating the arguments of \eqref{trunc} -- \eqref{trunc3}, we need to prove
\begin{equation}\label{trunc5}
 \Pr\{|\widehat X_{jk}-\widetilde X_{jk}|>\frac C{n^3}\}\le \Pr\{(1-\sigma_{jk})\sigma_{jk}^{-1}
 |\widehat X_{jk}|>\frac C{n^2}\}.
\end{equation}
Note that
\begin{equation}
 \sigma_{jk}^2=1-\E X_{jk}^2\mathbb I\{|X_{jk}|\ge c\lna^{\frac1{\varkappa}}\}-
 (\E X_{jk}\mathbb I\{|X_{jk}|\ge c\lna^{\frac1{\varkappa}}\})^2\ge 1-\exp\{-c'\lna\}.
\end{equation}
The last bound is uniform in $j,k=1,\ldots,n$.
This implies that
\begin{equation}\label{trunc6}
 (1-\sigma_{jk})\sigma_{jk}^{-1}\le \exp\{-c''\lna\}.
\end{equation}
Inequalities \eqref{trunc5} and \eqref{trunc6} together imply \eqref{trunc4}.
Thus Lemma \ref{trunc} is proved.
\end{proof}
\subsection{Proof of the Key Lemma}\label{key}
\begin{proof}

Since $|s(z)|\le 1$ for any $z\in\mathbb C_+$, we obtain  that, for $\omega\in A_k$, for any $z=u+iv$ with $v\ge v_{k}$,
 \begin{equation}\notag
  |m_n(z)|\le |s(z)|+|m_n(z)-s(z)|\le \frac32,
 \end{equation}
 and
 \begin{equation}
  |z+m_n(z)|\ge|z+s(z)|-|m_n(z)-s(z)|\ge \frac12.
 \end{equation}

Furthermore,
\begin{equation}\notag
 |m_n^{(j)}(z)|\le |m_n(z)|+\frac1{nv}\le 2,
\end{equation}
for any $z=u+iv$ with $v\ge v_0$.
Note that for semi-circular law $|s(z)+z|\ge 1$ for any $z\in\mathbb C^+$.
This implies that, for any $\omega\in\mathcal A_k$ and any $z=u+iv$ with $v\ge v_k$,
\begin{equation}\notag
 |z+m_n(z)|\ge|z+s(z)|-|m_n(z)-s(z)|\ge \frac12.
\end{equation}
From here it follows that, for any $\omega\in\mathcal A_k$ and for any $z=u+iv$ with $v\ge v_0$, by \eqref{repr01}, for any $j=1,\ldots,n$
\begin{equation}\label{epsilon}
 |R_{jj}|\le 2(1+\gamma_0|R_{jj}|),
 \end{equation}
 which implies
 \begin{equation}\label{08}
 |R_{jj}|\le 2(1-2\gamma_0)^{-1}.
\end{equation}
From
$1+zs(z)+s^2(z)=0$ (see \eqref{stsemi}) we get $(z+s(z))s(z)=-1$ and hence by
  representation \eqref{repr03}
\begin{align}\label{res}
 \theta_n(z)&:=\frac1n\sum_{j=1}^n\varepsilon_jR_{jj}=-m_n(z)(z+m_n(z))-1\notag\\&=-(z+2s(z)+m_n(z)-s(z))(m_n(z)-s(z)).
\end{align}
From \eqref{repr04} it follows now
\begin{equation}\label{00}
 |m_n(z)-s(z)|\le \frac{2\gamma_0}{(1-2\gamma_0)|m_n(z)+s(z)+z|}.
\end{equation}
 Assume that $|m_n(z)+s(z)+z|\ge\sqrt{\gamma_0}$. Then
\begin{equation}\label{01}
 |m_n(z)-s(z)|\le \frac{2\sqrt{\gamma_0}}{1-2\gamma_0}.
\end{equation}
 Let now $|z+m_n(z)+s(z)|\le \sqrt{\gamma_0}$. Assume that
 \begin{equation}\label{05}
|m_n(z)-s(z)|\ge 2\sqrt{\gamma_0}.
\end{equation}
Then
\begin{equation}\label{02}
 |z+2s(z)|\ge|s(z)-m_n(z)|-|z+m_n(z)+s(z)|\ge \sqrt{\gamma_0}.
\end{equation}
Since $\re(z^2-4)<0$, for $|u|\le 2$, we have
\begin{align}\label{03}
 |z+m_n(z)+s(z)|&\ge\im(z+m_n(z)+ s(z))\ge \frac12\im(z+2s(z))\notag\\&\ge \frac1{2\sqrt2}|z^2-4|^{\frac12}
 =\frac1{2\sqrt2}|z+2s(z)|\ge\frac1{2\sqrt2}\sqrt{\gamma_0}.
\end{align}
Inequalities \eqref{00} and \eqref{03} under  assumption \eqref{05} together imply
\begin{equation}\label{06}
 |m_n(z)-s(z)|\le \frac{4\sqrt{2\gamma_0}}{1-2\gamma_0}.
\end{equation}
Combining inequalities \eqref{01}, \eqref{05} and \eqref{06} we get that \eqref{06} holds for any $\omega\in\mathcal A_k$.
Simple calculations show that if $\gamma_0\le\frac3{100}$ then
\begin{equation}
 |m_n(z)-s(z)|\le \frac14.
\end{equation}

Furthermore, for any $v_k\ge v\ge v_{k-1}$,
\begin{equation}\notag
 |s_n(u+iv_{k-1})-s_n(u+iv)|\le \frac{|v-v_{k-1}|}{vv_{k-1}}\le \frac1{n^2v_0^2}\le \frac14.
\end{equation}
The last two inequalities together imply, for any $z=u+iv$ with $v\ge v_{k-1}$,
$|m_n(z)-s(z)|\le \frac12$ or $\omega\in\mathcal A_{k-1}$.
Thus the Lemma is proved.
\end{proof}
\begin{rem}
 The inequality \eqref{08} yields, for $\gamma_0\le \frac3{100}$, 
\begin{equation}\label{resjj}
 |R_{jj}|\le \frac32.
\end{equation}
\end{rem}

\subsection{The Proof of Lemma \ref{cru2}}\label{proofofcru2}
Note that
\begin{equation}
 \Pr\{\mathcal D_k^c\cap\mathcal A_k\cap\mathcal U\}\le \sum_{j=1}^n\sum_{\nu=1}^4\Pr\{\{|\varepsilon_{j\nu}|\ge\frac14\gamma_0\}\cap \mathcal A_k\cap\mathcal  U\}.
\end{equation}
First we note that according Remark \ref{trunc00} we may choose the  constant $A_0$ such that
\begin{equation}
 \Pr\{|\varepsilon_{j1}|\ge\frac14\gamma_0\}=0.
\end{equation}
Next, we have (see Lemma 4.1 in \cite{GT:03})
\begin{equation}\label{epsilon1}
 |\varepsilon_{j4}|\le\frac 1{nv}\text{ a. s.}
\end{equation}
We may choose the constant $C_0$ in the definition of $v_0$ such that
\begin{equation}
 \Pr\{|\varepsilon_{j4}|\ge\frac14\gamma_0\}=0.
\end{equation}
We shall bound $\varepsilon_{j\nu}$, for $\nu=2,3$ and $j=1,\ldots,n$.
 First we note that
\begin{align}\label{w1}
\Pr\Bigl\{\{|\varepsilon_{j\nu}|&> \frac14\gamma_0\}\cap\mathcal A_k\cap \mathcal U\Bigr\}\notag\\&\le
\Pr\{|\varepsilon_{j\nu}|\ge \frac1{4\sqrt{\psi_0}}\gamma_0(\im m_n^{(j)}(z))^{\frac12}; \im m_n^{(j)}(z)\le \psi_0\}\notag\\&\le
\Pr\{|\varepsilon_{j\nu}|\ge \frac1{4\sqrt{\psi_0}}\gamma_0(\im m_n^{(j)}(z))^{\frac12}\}.
\end{align}
We now estimate the last probability in the r.h.s of \eqref{w1}.

We we shall use McDiarmid's  inequality for sums of independent random variables and for martingales.
Let $\xi_1,\ldots,\xi_n$ be martingale-differences such that $\E\xi_j=0$ and $|\xi_j|\le \sigma_j$. Then,
for some numerical constant $c>0$,
\begin{equation}\label{largedev}
 \Pr\{|\sum_{j=1}^n\xi_j|>x\}\le c(1-\Phi(x/\sigma))\le\frac{c\sigma}x\exp\{-\frac{x^2}{2\sigma^2}\},
\end{equation}
where $\Phi(x)=\frac1{\sqrt{2\pi}}\int_{-\infty}^x\exp\{-\frac {y^2}2\}dy$ and $\sigma^2=\sigma_1^2+\cdots+\sigma_n^2$.
The last inequality holds for $x\ge\sigma$. (See, for instance \cite{Bentkus:2007},  p.1,  first inequality.)
Consider $\nu=2$. We put
$\eta_l=\frac1n(X_{jl}^2-1)R_{ll}^{(j)}$.

Note that $\E\eta_l=0$ and
$|\eta_l|\le 2\lna^{\frac2{\varkappa}} n^{-1}|{R^{(j)}_{ll}|}.$
Recall that
\begin{equation}
 \varepsilon_{j2}=\frac1n\sum_{l\in\mathbb T_j}(X_{jl}^2-1)R^{(j)}_{ll}=\frac1n\sum_{l\in\mathbb T_j}\eta_l.
\end{equation}

Recall that $\textfrak M^{(j)}$ denotes the $\sigma$-algebra generated by the random variables $X_{kl}$ with $k,l\in\mathbb T_j$. Let $\E_j$ and ${\Pr}_j$
denote the conditional  expectation and the conditional probability with respect to $\textfrak M^{(j)}$. Note that the random variables $X_{jl}$
and the $\sigma$-algebra
$\textfrak M^{(j)}$ are independent.
Applying inequality \eqref{largedev} with
 $x:=\frac1{4\sqrt{\psi_0}}\gamma_0(\im m_n^{(j)}(z))^{\frac12}$ and with
$\sigma^2=4\lna^{\frac4{\varkappa}} n^{-1}\Bigg(\frac1n\sum_{l\in\mathbb T_j}|{R^{(j)}_{ll}|^2}\Bigg)$,  we get
\begin{align} \label{previous}
\Pr\Bigl\{|\sum_{l\in \mathbb T_j}\eta_j|>x\Bigr\}&= \E{\Pr}_j\Bigr\{|\sum_{l\in \mathbb T_j}\eta_j|\ge x\Bigr\}\notag\\&\le
\E\exp\Bigl\{-\frac {x^2}{\sigma^2}\Bigr\}\le C\exp\{-cnv/({16{\psi_0}}\lna^{\frac4{\varkappa}})\}.
\end{align}
Since
\begin{equation}
 \frac1n\sum_{l\in\mathbb T_j}|R_{ll}^{(j)}|^2\le v^{-1}\im m_n^{(j)}(z),
\end{equation}
we have used here
\begin{equation}\notag
 \sigma^2\le 4\lna^{\frac4{\varkappa}} n^{-1}v^{-1}\im m_n^{(j)}(z).
\end{equation}

 Now we consider $\varepsilon_{j3}$.
We represent $\varepsilon_{j3}=\frac1n\sum_{l\in\mathbb T_j}\alpha_l$ as a martingale with martingale difference
\begin{equation}\label{martingale1}
\alpha_l=\sum_{1\le q\ne r\le l}X_{jq}X_{jr}R_{qr}^{(j)}-\sum_{1\le q\ne r\le l-1}X_{jq}X_{jr}R_{qr}^{(j)}=
2X_{jl}\sum_{q=1}^{l-1}X_{jq}R_{ql}^{(j)}.
\end{equation}
Because the $X_{jl}$ are bounded by $\lna^{\frac1{\varkappa}}$, we may apply McDiarmid's inequality.
This implies that
\begin{equation}\label{martingale2}
\Pr\{|\alpha_l|\ge (nv_k/{\psi_0})^{\frac14}\lna^{\frac1{\varkappa}}(\sum_{q\in\mathbb
T_j}|R^{(j)}_{ql}|^2)^{\frac12}\}\le \exp\{-c\sqrt{nv_k}/(\sqrt{\psi_0}\lna^{\frac2{\varkappa}})\}.
\end{equation}
Furthermore, we note that
\begin{align}\label{martingale3}
\frac1n\sum_{l\in\mathbb T_j}\E\Bigg\{\alpha_l\mathbb
I\Big\{|\alpha_l|\ge(nv_k/\psi_0)^{\frac14}\lna^{\frac1{\varkappa}}(\sum_{k\in\mathbb
T_j}|R^{(j)}_{lk}|^2)^{\frac12}\Big\}\big|\mathfrak M_l^{(j)}\Bigg\}\notag\\ \le
\exp\{-c'\sqrt{nv_k}/(\sqrt{\psi_0}\lna^{\frac2{\varkappa}})\},
\end{align}
where $\mathfrak M^{(j)}_{l-1}$ denotes the $\sigma$-algebra generated by
$X_{sk}$ with $s,k\in\mathbb T_j$ such that $1\le s,k\le l-1$. We
consider now the truncated martingale with
\begin{align}\label{martingale4}
\widehat\alpha_l =\alpha_l\mathbb I\{|\alpha_l|&\le
(nv_k/\psi_0)^{\frac14}\lna^{\frac1{\varkappa}}(\sum_{t\in\mathbb
T_j}|R_{lt}^{(j)}|^2)^{\frac12}\}\notag\\&-\E\{\alpha\mathbb I\{|\alpha|\le
(nv_k/\psi_0)^{\frac14}\lna^{\frac1{\varkappa}} (\sum_{t\in\mathbb
T_j}|R^{(j)}_{lt}|^2)^{\frac12}\}\big|\mathfrak M_{l-1}^{(j)}\}.
\end{align}
Applying McDiarmid's inequality for martingales (see \cite{Bentkus:2007}), we get
\begin{equation}\label{martingale5}
\Pr\{|\sum_{l\in\mathbb T_j}\widehat\alpha_l|\ge
x(\sum_{s,t\in\mathbb T_j}|R_{st}^{(j)}|^2)^{\frac12}\}\le
\exp\{-x^2\sqrt{\psi_0}/(\sqrt{nv_k}\lna^{\frac2{\varkappa}})\}.
\end{equation}
Putting $x=\frac1{4\sqrt{\psi_0}}\gamma_0\sqrt{nv_k}$, we get
\begin{equation}\label{martingale6}
\Pr\{|\sum_{l\in\mathbb T_j}\widehat\alpha_l|\ge\gamma_n
\sqrt {nv_k}(\sum_{s,t\in\mathbb T_j}|R_{st}^{(j)}|^2)^{\frac12}\}\le
\exp\{-c\sqrt{nv_k}/(\sqrt{\psi_0}\lna^{\frac2{\varkappa}})\}.
\end{equation}

We may rewrite the last inequality now as
\begin{equation}\notag
 \Pr\{|\frac1n\sum_{l\in\mathbb T_j}\widehat\alpha_l|\ge
\frac1{4\sqrt{\psi_0}}\gamma_0(v_k\frac1n\sum_{s,t\in\mathbb T_j}|R_{st}^{(j)}|^2)^{\frac12}\}\le
\exp\{-c\sqrt{nv_k}/(\sqrt{\psi_0}\lna^{\frac2{\varkappa}})\}.
\end{equation}

Thus the Lemma is
proved.
\subsection{The proof of lemma \ref{h2}}\label{proofofh2}
We consider the case $\mathbb A=\emptyset$ only. We start with the following inequality
\begin{align}\label{ineq1}
\E\{(\sum_{q\in\mathbb T_{j,l}}|R_{ql}|^2)^m\mathbb I\{\mathcal U\}\big|\mathfrak M^{(j)}\} &\le 4^m(\sum_{q\in\mathbb T_{j}}|R^{(j)}_{ql}|^2)^m
\E\{\mathbb I\{\mathcal U\}\big|\mathfrak M^{(j)}\}\notag\\&+
4^m\E\{(\sum_{q\in\mathbb T_{j,l}}|R_{ql}-R^{(j)}_{ql}|^2)^{m}\mathbb I\{\mathcal U\}\big|\mathfrak M^{(j)}\}.
\end{align}
Using the equality  (see, for instance \cite{ErdosYauYin:2010a}, 
Lemma 3.2)
\begin{equation}\label{recursion}
 R_{ql}-R_{ql}^{(j)}=R_{qj}R_{jl}(R_{jj})^{-1},
\end{equation}
we get
\begin{align}\label{ineq2}
\E\{(\sum_{q\in\mathbb T_{j,l}}|R_{ql}|^2)^m&\mathbb I\{\mathcal U\}\big|\mathfrak M^{(j)}\} 
\le 4^m(\sum_{q\in\mathbb T_{j,l}}|R^{(j)}_{ql}|^2)^m\E\{\mathbb I\{\mathcal U\}\big|\mathfrak M^{(j)}\}\notag\\&+4^m
\E\{(\sum_{q\in\mathbb T_{j,l}}|R_{qj}|^2)^m|R_{jl}|^m|R_{jj}|^{-m}\mathbb I\{\mathcal U\}\big|\mathfrak M^{(j)}\}.
\end{align}
Note that by  \eqref{l2res}
\begin{equation}
 \sum_{q\in\mathbb T_{j,l}}|R_{qj}|^2\le nv^{-2}\text{ a.s.},
\end{equation}
and by \eqref{repr01} 
\begin{equation}
|\varepsilon_j|\le\frac{C\lna^{\frac1{\varkappa}}}n+\frac1{nv}+v^{-1}
\lna^{\frac2{\varkappa}}+ \lna^{\frac2{\varkappa}}\sqrt nv^{-1}\text{ a.s.}
\end{equation}
These inequalities together imply that, for $v\ge v_0$
\begin{equation}
 |R_{jj}|^{-1}\le |z+m_n(z)|(1+|\varepsilon_j||R_{jj}|)\le (|z|+v^{-1})(1+n\lna^{\frac2{\varkappa}}v^{-1})\le Cn^2\text{ a.s.}
\end{equation}
Applying these inequalities, we get
\begin{align}\label{ineq3}
&\E\{(\sum_{q\in\mathbb T_{j,l}}|R_{qj}|^2)^m|R_{jl}|^m|R_{jj}|^{-m}\mathbb I\{\mathcal U\}\big|\mathfrak M^{(j)}\}\notag\\&\le
C^m\E\{(\sum_{q\in\mathbb T_{j,l}}|R_{qj}|^2)^{m}|R_{jl}|^{m}\mathbb I\{\mathcal U\}\big|\mathfrak M^{(j)}\}+
C^mn^{3m}\E\{\mathbb I\{(\mathcal B^{(l)})^c\}\mathbb I\{\mathcal U\}\big|\mathfrak M^{(j)}\},
\end{align}
where $\mathcal B^{(\mathbb A)}$ was defined in \eqref{ba}.
Applying H\"older's inequality, we get
\begin{align}\label{ineq5}
 \E\{(\sum_{q\in\mathbb T_{j,l}}|R_{qj}|^2)^m|R_{jl}|^m\mathbb I\{\mathcal U\}\big|\mathfrak M^{(j)}\}&\le
 \E^{\frac12}\{(\sum_{q\in\mathbb T_{j,l}}|R_{qj}|^2)^{2m}\mathbb I\{\mathcal U\}\big|\mathfrak M^{(j)}\}\notag\\&\qquad
 \times\E^{\frac12}\{|R_{jl}|^{2m}\mathbb I\{\mathcal U\}\big|\mathfrak M^{(j)}\}.
\end{align}
By Lemma \ref{h1}, we have
\begin{equation}\label{ineq6}
\E^{\frac12}\{(\sum_{q\in\mathbb T_{j,l}}|R_{qj}|^2)^{2m}\mathbb
I\{\mathcal U\}\big|\mathfrak M^{(j)}\} \le
Cm^m(\lna^{\frac2{\varkappa}})^m\psi_0^m+\Phi_n.
\end{equation}

 The second factor is bounded, using equality \eqref{inter}.
We get
\begin{align}
 \E^{\frac12}&\{|R_{jl}|^{2m}\mathbb I\{\mathcal U\}\big|\mathfrak M^{(j)}\}\le
 \E^{\frac12}\{|R_{jj}|^{2m}|\frac1{\sqrt n}\sum_{q\in\mathbb T_{j,t}}X_{jq}R^{(j)}_{ql}|^{2m}\mathbb I\{\mathcal U\}\big|
  \mathfrak M^{(j)}\}\notag\\&\le C^{2m}\E^{\frac12}\{|\frac1{\sqrt n}\sum_{q\in\mathbb T_{j,l}}X_{jq}R^{(j)}_{ql}|^{2m}\mathbb I\{\mathcal U\}\big|
 \mathfrak M^{(j)}\}+\E^{\frac12}\{\mathbb I\{(\mathcal B)^c\mathbb I\{\mathcal U\}\}\big|\mathfrak M^{(j)}\}.\notag
\end{align}

Applying now Rosenthal's inequality, we obtain
\begin{align}\label{ineq20}
\E^{\frac12}\{|R_{jl}|^{2m}\mathbb I\{\mathcal U\}\big|\mathfrak M^{(j)}\}&\le  C^mm^{m}(\frac1{\sqrt n})^{m}
(\sum_{q\in\mathbb T_{j,l}}|R^{(j)}_{ql}|^2)^{\frac m2}\E\{\mathbb I\{\mathcal U\}\big|\mathfrak M^{(j)}\}
\notag\\&\qquad\qquad +\E^{\frac12}\{\mathbb I\{(\mathcal B^{(l)})^c\}\mathbb I\{\mathcal U\}\big|\mathfrak M^{(j)}\}.
\end{align}
Inequalities \eqref{ineq5}, \eqref{ineq6}, \eqref{ineq20} together imply
\begin{align}\label{ineq21}
\E\{(\sum_{q\in\mathbb T_{j,l}}|R_{qj}|^2)^m&|R_{jl}|^m\mathbb I\{\mathcal U\}\big|\mathfrak M^{(j)}\}\notag\\&\le
\left(\frac{Cm^2\lna^{\frac2{\varkappa}}}{\sqrt n}\right)^m\psi_0^m
(\sum_{q\in\mathbb T_{j,l}}|R^{(j)}_{ql}|^2)^{\frac m2}\E\{\mathbb I\{\mathcal U\}\big|\mathfrak M^{(j)}\}\notag\\&\qquad +
(Cm\lna^{\frac2{\varkappa}}v^{-1})^m\E^{\frac12}\{\mathbb I\{(\mathcal B)^c\}\mathbb I\{\mathcal U\}\big|\mathfrak M^{(j)}\}.
\end{align}

Inequalities  \eqref{ineq2} and \eqref{ineq21} yield
\begin{align}
 \E\{(\sum_{q\in\mathbb T_{j,l}}|R_{ql}|^2)^m\mathbb I\{\mathcal U\}\big|\mathfrak M^{(j)}\}&\le C^m(\sum_{q\in\mathbb T_{jl}}|R^{(j)}_{ql}|^2)^m
 \E\{\mathbb I\{\mathcal U\}\big|\mathfrak M^{(j)}\}\notag\\&+
 (\frac{Cm^2\lna^{\frac2{\varkappa}}}{\sqrt n})^m(\sum_{q\in\mathbb T_{jl}}|R^{(j)}_{ql}|^2)^{\frac m2}\E\{\mathbb I\{\mathcal U\}\big|\mathfrak M^{(j)}\}\notag\\&
 +C^mn^{3m}\E\{\mathbb I\{(\mathcal B^{(l)})^c\}\mathbb I\{\mathcal U\}\big|\mathfrak M^{(j)}\}\notag\\&
 +C^mn^{2m}\E^{\frac12}\{\mathbb I\{(\mathcal B)^c\}\mathbb I\{\mathcal U\}\big|\mathfrak M^{(j)}\}.\notag
\end{align}

The last inequality completes the proof. Thus Lemma \ref{h2} is proved.
\subsection{The Proof of Lemma \ref{trunczeta}}\label{proofoftrunczeta}
Note that 
\begin{equation}\label{zeta}
 \Pr\{\zeta\ne\widehat\zeta\}\le \sum_{j\in \mathbb T}\Pr\{\mathbb F_j^c\}+n\Pr\{\mathcal  E^c\}+\Pr\{\mathcal U^c\}.
\end{equation}
By Corollary \ref{cru3},
\begin{equation}\notag
 \Pr\{\mathcal  E^c\}\le \exp\{-c\sqrt{nv_0}/\lna\}\le \exp\{-c\lna\}.
\end{equation}
From here it follows that
\begin{equation}\label{t1}
 n\Pr\{\mathcal  E^c\}\le \exp\{-c'\lna\}.
\end{equation}
We bound now $\Pr\{(\mathcal F_j)^c\}$ for $j=1,\ldots,n$.
We use the same arguments as in \eqref{martingale2} -- \eqref{martingale6}. Recall
 $\varepsilon_{j3}=\zeta_j=\frac1n\sum_{l\in\mathbb T_j}\alpha_l$  with
\begin{equation}\label{martingale01}
\alpha_l=\sum_{1\le q\ne r\le l}X_{jq}X_{jr}R_{qr}^{(j)}-\sum_{1\le q\ne r\le l-1}X_{jq}X_{jr}R_{qr}^{(j)}=
2X_{jl}\sum_{q=1}^{l-1}X_{jq}R_{ql}^{(j)}.
\end{equation}
Since $X_{jl}$ are bounded by $\lna^{\frac1{\varkappa}}$, we may apply McDiarmid's inequality.   This implies that
\begin{equation}\label{martingale02}
\Pr\{|\alpha_l|\ge C\beta_n(\sum_{q\in\mathbb
T_j}|R^{(j)}_{ql}|^2)^{\frac12}\}\le \exp\{-c\beta_n^2/\lna^{\frac2{\varkappa}}\}\le \exp\{-c\lna\}.
\end{equation}
Now we consider the truncated martingale-difference
\begin{align}
 \widehat\alpha_l=\alpha_l\mathbb I\{|\alpha_l|\le C\beta_n(\sum_{q\in\mathbb
T_j}|R^{(j)}_{ql}|^2)^{\frac12}\} -\E\{\alpha_l\mathbb I\{|\alpha_l|\le C\beta_n(\sum_{q\in\mathbb
T_j}|R^{(j)}_{ql}|^2)^{\frac12}\}\big|\mathfrak M_{l-1}^{(j)}\}.\notag
\end{align}
Applying McDiarmid's inequality for martingales (see \cite{Bentkus:2007}), we get
\begin{equation}\label{martingale05}
\Pr\{|\sum_{l\in\mathbb T_j}\widehat\alpha_l|\ge
x(\sum_{s,t\in\mathbb T_j}|R_{st}^{(j)}|^2)^{\frac12}\}\le
\exp\{-x^2/\beta_n^2\}.
\end{equation}
Putting $x=C\beta_n^2(\sum_{s,t\in\mathbb T_j}|R_{st}^{(j)}|^2)^{\frac12}$, we get
\begin{equation}\label{martingale06}
\Pr\{|\sum_{l\in\mathbb T_j}\widehat\alpha_l|\ge\gamma_n
\beta_n\sqrt{\lna}(\sum_{s,t\in\mathbb T_j}|R_{st}^{(j)}|^2)^{\frac12}\}\le
\exp\{-c\lna\}.
\end{equation}
Note that $|\alpha_l|\le \lna^{\frac2{\varkappa}}nv^{-1}$.
Similar to \eqref{martingale3} we get
\begin{align}\label{martingale03}
\frac1n\sum_{l\in\mathbb T_j}\E\{\alpha_l\mathbb
I\{|\alpha_l|&\ge\beta_n\sqrt{\lna}
(\sum_{k\in\mathbb
T_j}|R^{(j)}_{lk}|^2)^{\frac12}\}\big|\mathfrak M_{l-1}^{(j)}\}\notag\\&\le
\lna^{\frac2{\varkappa}}\exp\{2\log n\}\exp\{-c\lna\}\le \exp\{-c'\lna\}.
\end{align}
 Inequalities \eqref{t1}, \eqref{martingale05}, and  \eqref{martingale03} together complete the proof of lemma.
 Thus Lemma \ref{trunczeta} is proved.
\subsection{The Proof of Lemma \ref{cond3}}\label{proofofcond3}

We have
\begin{align}\label{term2}
 \E\{|\widehat\zeta_l|^m\mathbb I\{(\mathcal G^{(j)})^c\}\big|\mathfrak M^{(j)}\}\le
 \left(\frac{C\lna^{\frac1{\varkappa}}\sqrt{\psi_0}}{\sqrt{nv}}\right)^m\mathbb I\{(\mathcal G^{(j)})^c\}
\end{align}
and
\begin{align}\label{term3}
 \E\{|\widehat\zeta_l^{(j)}|^m\mathbb I\{(\mathcal G)^c\}\big|\mathfrak M^{(j)}\}\le
 \left(\frac{C\lna^{\frac1{\varkappa}}\sqrt{\psi_0}}{\sqrt{nv}}\right)^m\mathbb I\{\mathcal G^{(j)}\}\E\{\mathbb I\{\mathcal G^c\}\big|\mathfrak M^{(j)}\}.
\end{align}
Note that
\begin{equation}
 \mathcal E^c\cap\mathcal E=\emptyset,\quad \mathcal U^c\cap\mathcal U^{(j)}=\emptyset.
\end{equation}
From here it follows that
\begin{equation}\label{l0}
 \mathbb I\{\mathcal G^{(j)}\}\E\{\mathbb I\{\mathcal G^c\}\big|\mathfrak M^{(j)}\}\le
 \E\{\mathbb I\{\mathcal E^{(j)}\cap\mathcal U^{(j)}\cap\mathcal F^{(j)}\cap\mathcal F^{c}\}\big|\mathfrak M^{(j)}\}.
\end{equation}
Note that
\begin{equation}
 |\zeta_l|\le|\zeta_l^{(j)}|+|\zeta_l-\zeta_l^{(j)}|.
\end{equation}
Therefore,
\begin{equation}\label{l1}
 \mathcal F_l^{c}\subset (\mathcal F_l^{(j)})^c\cup\{|\zeta_l-\zeta_l^{(j)}|\ge \rho_n(\im m_n^{(j)}(z))^{\frac12}\}.
\end{equation}
The relation \eqref{l0} and \eqref{l1} together imply
\begin{equation}
 \mathbb I\{\mathcal F^{c}\cap\mathcal F^{(j)}\}\le \sum_{l\in\mathbb T_j}\mathbb I\{|\zeta_l-\zeta_l^{(j)}|\ge \rho_n(\im m_n^{(j)}(z))^{\frac12}\}.
\end{equation}
Moreover,
\begin{align}
 \mathbb I\{|\zeta_l-\zeta_l^{(j)}|\ge \rho_n(\im m_n^{(j)}(z))^{\frac12}\} \; \le \; & \mathbb I\{|\xi_{lj}|\ge\frac12\rho_n(\im m_n^{(j)}(z))^{\frac12}\}\notag\\&\qquad +
 \mathbb I\{|\eta_{lj}|\ge\frac12\rho_n(\im m_n^{(j)}(z))^{\frac12}\}.\notag
\end{align}

 Introduce the event
 \begin{equation}\notag
  \mathcal P:=\{\sum_{q\in\mathbb T_{j,l}}|R_{jq}^{(l)}|^2\le \Gamma(Cm^2(\im m_n(z)+\frac 1{nv})+\Phi_n)\},
 \end{equation}
for some constant $\Gamma>0$ which will be chosen later.
Applying  Chebyshev's inequality and Lemma \ref{h1} with $L=Am\log {n}$, where $A>0$ will be chosen later, we get
\begin{align}
 \E\{\mathbb I\{\mathcal G_l\}\mathbb I\{(\mathcal G_l^{(j)})^c\}&\mathbb I\{\mathcal P^c\}
\big|\mathfrak M^{(j)}\} \le\E\{\mathbb I\{ \mathcal
P^c\}\big|\mathfrak M^{(j)}\} \notag\\& \le
\frac{\E\{(\sum_{q\in\mathbb T_j}|R_{jq}^{(l)}|^2)^L\big|\mathfrak
M^{(j)}\}}{\Gamma^L(Cm^2(\im m_n^{(j)}(z)+\frac 1{nv})
+C\lna^{\frac2{\varkappa}}+\Phi_n)^L}\le
\frac{L^2}{\Gamma^L}.\notag
\end{align}
Choosing appropriate constants $\Gamma>1$ and $A\ge2\log\Gamma$, we get
\begin{align}\label{quadrat}
\E\{\mathbb I\{\mathcal G_l^c\}\mathbb I\{(\mathcal G_l^{(j)})\}\mathbb I\{\mathcal P^c\}
\big|\mathfrak M^{(j)}\} \le\frac {1}{n^{2m}}.
\end{align}

Let
\begin{align}
 V=\E\{\mathbb I\{\mathcal U^{(j)}\cap\mathcal E^{(j)}\cap\{|\xi_{lj}|\ge\frac12\rho_n(\im m_n^{(j)}(z))^{\frac12}\}\}\mathbb I\{\mathcal P\}\big|\mathfrak M^{(j)}\}.
\end{align}
It is straightforward to check that
\begin{align}
 V\le \Pr\{|\xi_{lj}|\ge\frac12\rho_n(\sqrt{\psi_0})^{-1}(\sum_{q\in\mathbb T_{l,j}}|R^{(l)}_{jq}|^2)\big|\mathfrak M^{(j)}\}.
\end{align}

Applying McDiarmid's inequality for martingale-differences similar to the proof of  Lemma \ref{trunczeta} , we get
\begin{equation}\label{quadrat2}
 \Pr\{|\xi_{lj}|\ge\frac12\rho_n(\sqrt{\psi_0})^{-1}(v\sum_{q\in\mathbb T_{l,j}}|R_{jq}^{(l)}|^2)\}\le
 \exp\{-\sqrt{nv}/\sqrt{\kappa_n}\}.
\end{equation}
Similarly we get
\begin{equation}\label{quadrat3}
 \Pr\{|\eta_{lj}|\ge\frac12\rho_n(v\sum_{q\in\mathbb T_{l,j}}|R_{jq}^{(l)}|^2)^{\frac12}\}
\le\exp\{-c\sqrt{nv}/\sqrt{\kappa_n}\}.
\end{equation}
Thus we have proved
\begin{align}\label{phi0}
 \E\{|\widehat\zeta_l^{(j)}|^m\mathbb I\{(\mathcal G)^c\}\}\big|\mathfrak M^{(j)}\}\le
 \left(\frac{C\lna^{\frac2{\varkappa}}}{\sqrt{nv}}\right)^m\Phi_n\le 
\Phi_n .
\end{align}

Inequalities \eqref{term2}, \eqref{term3} and (9.96) together
imply the claim. Thus Lemma \ref{cond3} is proved.
\begin{rem}\label{rem1}As follows from \eqref{quadrat} -- \eqref{quadrat3} we have proved
\begin{equation}
\E\{\mathbb I\{\mathcal G^c\}\mathbb I\{(\mathcal G^{(j)})^c\}\big|\mathfrak M^{(j)}\}\le \exp\{-c\lna\}+\Phi_n.
\end{equation}
\end{rem}
\subsection{The proof of lemma \ref{h3}}\label{proofofh3}
 For the simplicity we shall consider the case $\mathbb A=\emptyset$. In the other case we may
 consider matrix $\mathbf W^{(\mathbb A)}$ instead of matrix $\mathbf W$.
 First we bound $\E\{|\widetilde\alpha_{lj}^{(\mathbb A,t)}|^m\big|\mathfrak M^{(\mathbb A,j,t)}\}$.
 By the definition of this quantity we have
 \begin{align}
  \E\{|\widetilde\alpha_{lj}^{(t)}|^m&\mathbb I\{\mathcal U\}\big|\mathfrak M^{(j,t)}\}\notag\\&\le
  \left(\frac{C\lna^{\frac1{\varkappa}}}{n}\right)^m\E\{|\sum_{q\in\mathbb T_{l,j, t}}X_{lq}R^{(l)}_{qj}|^m
|R^{(l)}_{tj}|^m|(R^{(l)}_{jj})^{-1}|^m\mathbb I\{\mathcal U\}\big|\mathfrak M^{(j,t)}\}.\notag
 \end{align}
Continue this inequality, applying representation \eqref{inter}. We obtain
\begin{align}
 \E\{|\widetilde\alpha_{lj}^{(t)}|^m&\mathbb I\{\mathcal U\}\big|\mathfrak M^{(j,t)}\}\le\left(\frac{C\lna^{\frac1{\varkappa}}}{\sqrt n}\right)^m
 \E\Bigg\{\left|\frac1{\sqrt n}(\sum_{q\in\mathbb T_{l}}X_{lq}R^{(l)}_{qj}-X_{lj}R^{(l)}_{jj}-X_{lt}R^{(l)}_{tj})\right|^m\notag\\&
 \qquad\qquad\qquad\qquad\qquad\qquad\qquad\qquad\times|R^{(l)}_{tj}|^m|( R^{(l)}_{jj})^{-1}|^m\mathbb I\{\mathcal U\}
 \big|\mathfrak M^{(j,t)}\Bigg\}\notag\\&\le
 \left(\frac{C\lna^{\frac1{\varkappa}}}{\sqrt n}\right)^m\E\Bigg\{\left|R_{lj}(R_{jj})^{-1}-\frac1{\sqrt n}X_{lj}R^{(l)}_{jj}-
 \frac1{\sqrt n}X_{lt}R^{(l)}_{tj}\right|^m\notag\\&
 \qquad\qquad\qquad\qquad\qquad\qquad\qquad\qquad\times
 |R^{(l)}_{tj}|^m|(R^{(l)}_{jj})^{-1}|^m\mathbb I\{\mathcal U\}
 \big|\mathfrak M^{(j,t)}\Bigg\}.\notag
\end{align}
Applying representation \eqref{inter} again and Corollary \ref{cru3}, we get
\begin{align}
 \E\{|\widetilde\alpha_{lj}^{(t)}|^m\mathbb I\{\mathcal U\}\big|\mathfrak M^{(j,t)}\}&\le\left(\frac{C\lna^{\frac1{\varkappa}}}{\sqrt n}\right)^m
 \E\Bigg\{\Big(\big|\frac1{\sqrt n}\sum_{q\in\mathbb T_j}X_{jq}R^{(j)}_{lq}\big|^m
 \notag\\&+\left(\frac{\lna^{\frac1{\varkappa}}}{\sqrt n}\right)^m\Big(1+
 |\frac1{\sqrt n}\sum_{q\in\mathbb T_{l,j}}X_{jq}R^{(j)}_{tq}|^m\Big)\Big)\notag\\&
 \qquad\qquad\qquad\times
 \big|\frac1{\sqrt n}\sum_{r\in\mathbb T_{lj}}X_{qj}R_{tq}^{(l,j)}\big|^m\mathbb I\{\mathcal U\}
  \Big|\mathfrak M^{(j,t)}\Bigg\}\notag\\&+\E\big\{|\widetilde\alpha_{lj}^{(t)}|^m\mathbb I\{\mathcal U\}\mathbb I\{(\mathcal B\cap\mathcal B^{(l)})^c\}\big|\mathfrak M^{(j,t)}\big\}.
\end{align}

This inequality implies that
\begin{align}
 \E\{|&\widetilde\alpha_{lj}^{(t)}|^m\mathbb I\{\mathcal U\}\big|\mathfrak M^{(j,t)}\}\le\left(\frac{C\lna^{\frac2{\varkappa}}}{ n}\right)^m
 \E\{(\frac1{\sqrt n}|\sum_{q\in\mathbb T_{l,j}}X_{jq}R^{(j)}_{tq}|)^m\mathbb I\{\mathcal U\}
 \big|\mathfrak M^{(j,t)}\}\notag\\&+\left(\frac{C\lna^{\frac2{\varkappa}}}{n}\right)^m\E\{(\frac1{\sqrt n}|\sum_{q\in\mathbb T_{l,j}}X_{jq}R^{(j,l)}_{tq}|)^m
(|\frac1{\sqrt n}\sum_{q\in\mathbb T_j}X_{jq}R^{(j)}_{tq}|)^m\mathbb I\{\mathcal U\}\big|\mathfrak M^{(j,t)}\} \notag\\&+
\left(\frac{C\lna^{\frac1{\varkappa}}}{\sqrt n}\right)^m\E\{(\frac1{\sqrt n}|\sum_{q\in\mathbb T_{l,j}}X_{jq}R^{(j,l)}_{tq}|)^m
(|\frac1{\sqrt n}\sum_{q\in\mathbb T_j}X_{jq}R^{(j)}_{lq}|)^m\mathbb I\{\mathcal U\}\big|\mathfrak M^{(j,t)}\}\notag\\&
+\E\{|\widetilde\alpha_{lj}^{(t)}|^m\mathbb I\{\mathcal U\}\mathbb I\{(\mathcal B\mathcal B^{(l)})^c\}\big|\mathfrak M^{(j,t)}\}.
\end{align}
Denote by $\widetilde {\mathcal U}^{(j,t)}:=\{m_n^{(j,t)}(z)\le2\psi_0\}$.
Using that $\mathbb I\{\mathcal U\}\le \mathbb I\{\widetilde {\mathcal U}^{(j,t)}\}$ and applying H\"older's inequality and Rosenthal's inequality, we get
\begin{align}
 \E&\big\{|\widetilde\alpha_{lj}^{(t)}|^m\mathbb I\{\mathcal U\}\big|\mathfrak M^{(j,t)}\big\}\le\left(\frac{Cm\lna^{\frac3{\varkappa}}}{ n\sqrt n}\right)^m
 \E\big\{(\sum_{q\in\mathbb T_{l,j}}|R^{(j)}_{tq}|^2)^
 {\frac m2}\mathbb I\{\widetilde{\mathcal U}^{(j,t)}\}
 \big|\mathfrak M^{(j,t)}\big\}\notag\\
 & +\left(\frac{Cm\lna^{\frac2{\varkappa}}}{ n}\right)^{2m}
 \E^{\frac12}\{(\sum_{q\in\mathbb T_{l,j}}|R^{(j,l)}_{tq}|^2)^{ m}\mathbb I\{\widetilde{\mathcal U}^{(j,t)}\}\big|\mathfrak M^{(j,t)}\}\notag\\&\qquad\qquad\qquad\qquad\qquad\qquad\qquad\qquad\qquad\times
\E^{\frac12}\{(\sum_{q\in\mathbb T_j}|R^{(j)}_{tq}|^2)^{m}\mathbb I\{\widetilde{\mathcal U}^{(j,t)}\}\big|\mathfrak M^{(j,t)}\}\notag\\&
+\left(\frac{Cm\lna^{\frac2{\varkappa}}}{ n}\right)^{\frac32m}
 \E^{\frac12}\{(\sum_{q\in\mathbb T_{l,j}}|R^{(j,l)}_{tq}|^2)^{ m}\mathbb I\{\widetilde{\mathcal U}^{(j,t)}\}\big|\mathfrak M^{(j,t)}\}\notag\\&\qquad\times
\E^{\frac12}\big\{(\sum_{q\in\mathbb T_j}|R^{(j)}_{lq}|^2)^{m}\mathbb \mathbb I\{\widetilde{\mathcal U}^{(j,t)}\}\big|\mathfrak M^{(j,t)}\big\} 
+\E\{|\widetilde\alpha_{lj}^{(t)}|^m\mathbb I\{\mathcal U\}\mathbb I\{(\mathcal B\mathcal B^{(l)})^c\}\big|\mathfrak M^{(j,t)}\}.
\end{align}
This inequality and Lemmas \ref{h1} and \ref{h2} together imply
\begin{align}
 \E&\{|\widetilde\alpha_{lj}^{(t)}|^m\mathbb I\{\mathcal U\}\big|\mathfrak M^{(j,t)}\}\le
 \left(\frac{Cm^{\frac32}\lna^{\frac3{\varkappa}}}{ n\sqrt{nv}}\right)^m\psi_0^{\frac m2}
 +\left(\frac{Cm^2\lna^{\frac4{\varkappa}}}{ n^2v}\right)^m\psi_0^m+\Phi_n
 \notag\\ &+
 \left(\frac{Cm^{\frac32}\lna^{\frac3{\varkappa}}}{(nv)\sqrt n}\right)^m\psi_0^{\frac m2}(v\sum_{q\in\mathbb T_{j,t,l}}|R^{(j,t)}_{ql}|^2)^{\frac m2}\Big((v\sum_{q\in\mathbb T_{j,t,l}}|R^{(j,t)}_{ql}|^2)^{\frac m2}+\left(\frac{Cm\lna^{\frac2{\varkappa}}}{\sqrt n}\right)^{m}\Big)\notag\\&\qquad\qquad\qquad\qquad\qquad\qquad
 \qquad\qquad\qquad\qquad\qquad\qquad\times\mathbb I\{\widetilde{\mathcal U}^{(j,t)}\}
 .\notag
\end{align}
Using that $\psi_0\ge c\sqrt v$, we may write
\begin{align}
 \E&\{|\widetilde\alpha_{lj}^{(t)}|^m\mathbb I\{\mathcal U\}\big|\mathfrak M^{(j,t)}\}\le
 \left(\frac{Cm^{\frac32}\lna^{\frac3{\varkappa}}}{(nv)\sqrt {nv}}\right)^m\psi_0^{\frac {3m}2}+\Phi_n\notag\\&+\left(\frac{Cm^{\frac32}
 \lna^{\frac3{\varkappa}}}{(nv)\sqrt n}\right)^m\psi_0^{\frac m2}
 (v\sum_{q\in\mathbb T_{j,t,l}}|R^{(j,t)}_{ql}|^2)^{m} I\{\widetilde{\mathcal U}^{(j,t)}\}.
 \end{align}

We consider now $\widehat\alpha_{\nu lj}^{(\mathbb A,t)}$, for $\nu=1,2,3$.
We consider $\nu=1$ only. The other cases are similar.
Using equality \eqref{inter}, we write
\begin{align}
 \widehat\alpha_{1lj}^{(t)}&=\frac1{\sqrt n}\sum_{r\in\mathbb T_{l,t}}X_{lr}R^{(l)}_{jr}\frac1{\sqrt n}\sum_{ q\in\mathbb T_{l,t,r}}X_{lq}
R_{qt}^{(l)}R_{jt}^{(l)}(R^{(l)}_{jj})^{-1}(R^{(l)}_{tt})^{-1}\notag\\&=\frac1{\sqrt n}\sum_{r\in\mathbb T_{l,t}}X_{lr}R^{(l)}_{jr}
(R_{lt}R_{ll}^{-1}-\frac1{\sqrt n}X_{lr}R^{(l)}_{tr}-\frac1{\sqrt n}X_{lt}R^{(l)}_{tt})(R^{(l)}_{jj})^{-1}(R^{(l)}_{tt})^{-1}R_{jt}^{(l)}\notag\\&=
(R_{lj}R_{ll}^{-1}-\frac1{\sqrt n}X_{lt}R_{jt}^{(l)})(
R_{lt}R_{ll}^{-1}-
\frac1{\sqrt n}X_{lt}R^{(l)}_{tt})(R^{(l)}_{jj})^{-1}R_{jt}^{(l)}(R^{(l)}_{tt})^{-1}\notag\\&
-\frac1{ n}\sum_{r\in\mathbb T_{l,t}}X_{lr}^2R^{(l)}_{jr}R^{(l)}_{tr}R_{jt}^{(l)}(R^{(l)}_{jj})^{-1}(R^{(l)}_{tt})^{-1}.\notag
\end{align}
For $\omega\in\mathcal B^{(l)}$, we have 
\begin{align}
 |\widehat\alpha_{1lj}^{(t)}|\le C\Big(|R_{lt}|^2+|R_{lj}|^2+(|R^{(l)}_{jt}|^2+|R^{(l)}_{jt}|)\frac{C\lna^{\frac2{\varkappa}}}{n}\Big)|R_{jt}^{(l)}|\notag\\+
 \frac{C\lna^{\frac2{\varkappa}}}{n}(\sum_{r\in\mathbb T_{l,t}}|R^{(l)}_{jr}|^2+\sum_{r\in\mathbb T_{l,t}}|R^{(l)}_{tr}|^2)|R_{jt}^{(l)}|.\notag
\end{align}

Applying the representation \eqref{inter}, we get for $\omega\in\mathcal B\cap\mathcal B^{(l)}$, 
\begin{align}
 |\widehat\alpha_{1lj}^{(t)}|&\le \frac C{n^{\frac32}}\Big(\big|\sum_{q\in T_{t}}X_{tq}R^{(t)}_{ql}\big|^2+
 \big|\sum_{q\in T_j}X_{jq}R^{(j)}_{ql}\big|^2\Big)\big|\sum_{q\in\mathbb T_{j,l}}X_{jq}R^{(j,l)}_{qt}\big|\notag\\&+\frac{C\lna^{\frac2{\varkappa}}}{n^{\frac52}}\Big(
 \big|\sum_{q\in\mathbb T_{j,l}}X_{jq}R^{(j,l)}_{qt}\big|^3+
 \big|\sum_{q\in\mathbb T_{j,l}}X_{jq}R^{(j,l)}_{qt}\big|^2\Big)\notag\\&+
 \frac{C\lna^{\frac2{\varkappa}}}{n^{\frac32}}\Big(\sum_{r\in\mathbb T_{l,t}}|R^{(l)}_{jr}|^2+\sum_{r\in\mathbb T_{l,t}}|R^{(l)}_{tr}|^2\Big)
\Big |\sum_{q\in\mathbb T_{j,l}}X_{jq}R^{(j,l)}_{qt}\Big|.\notag
\end{align}
This inequality, Lemma \ref{h1} and McDiarmid's inequality together imply
\begin{align}
  \E^{\frac1m}&\{|\widehat\alpha_{1lj}^{(t)}|^m\mathbb I\{\mathcal U\}\big|\mathfrak M^{(t,j)}\}\le\left(\frac{C\lna^{\frac1{\varkappa}}}{\sqrt{nv}}\right)^3
  \Bigg(\Big(\E^{\frac2{3m}}\{(v\sum_{q\in\mathbb T_{t}}|R^{(t)}_{ql}|^2)^{\frac{3m}2}\mathbb I\{\mathcal U\}
  \big|\mathfrak M^{(j,t)}\}
  \notag\\&+
  \E^{\frac2{3m}}\{(v\sum_{q\in\mathbb T_{j}}|R^{(j)}_{ql}|^2)^{\frac{3m}2}\mathbb I\{\mathcal U\}\big|\mathfrak M^{(t,j)}\}\Big)
  \E^{\frac1{3m}}\{(v\sum_{q\in\mathbb T_{j,l}}|R^{(j)}_{qt}|^2)^{\frac{3m}2}\mathbb I\{\mathcal U\}\big|\mathfrak M^{(t,j)}\}\notag\\&+
 \E^{\frac1m}\{(v\sum_{q\in\mathbb T_{j,l}}|R^{(j)}_{qt}|^2)^{\frac{3m}2}\mathbb I\{\mathcal U\}\big|\mathfrak M^{(t,j)}\} +\E^{\frac1m}\{(v\sum_{q\in\mathbb T_{j,l}}|R^{(j)}_{qt}|^2)^{m}\mathbb I\{\mathcal U\}\big|\mathfrak M^{(t,j)}\}\notag\\&+
\Big(\E^{\frac2{3m}}\{(\sum_{r\in\mathbb T_{l,t}}|R^{(l)}_{jr}|^2)^{\frac{3m}2}\mathbb I\{\mathcal U\}\big|\mathfrak M^{(j)}\}
  +
 \E^{\frac2{3m}}\{(\sum_{r\in\mathbb T_{l,t}}|R^{(l)}_{tr}|^2)^{\frac{3m}2}\mathbb I\{\mathcal U\}\big|\mathfrak M^{(j)}\}\Big)\notag\\&
 \times
\E^{\frac1{3m}}\big\{(v\sum_{q\in\mathbb T_{jl}}|R^{(j,l)}_{qt}|^2)^{\frac{3m}2}\mathbb I\{\mathcal U\}\big|\mathfrak M^{(j,l)}\big\}+
n^3\E\{\mathbb I\{\mathcal U\cap(\mathcal B\cap\mathcal B^{(l)})^c\}\big|\mathfrak M^{(j,t)}\}\Bigg).
\end{align}
Using Lemmas \ref{h1} and \ref{h2}, we get
\begin{align}
 \E^{\frac1m}\{|\widehat\alpha_{1lj}^{(t)}|^m\big|\mathfrak M^{(t,j)}\}\le
 \frac{Cm^3\lna^{\frac3{\varkappa}}}{(nv)^{\frac32}}&\Bigg(v\Big(\sum_{q\in\mathbb T_{t}}|R^{(j,t)}_{ql}|^2\Big)\psi_0^{\frac12}\mathbb I\{\mathcal U^{(j,t)}\}
 +\psi_0^{\frac32}
\notag\\&+
n^3\E\{\mathbb I\{\mathcal U\cap(\mathcal B\cap\mathcal B^{(l)})^c\}\big|\mathfrak M^{(j,t)}\}\Bigg).\notag
\end{align}

Thus Lemma \ref{h3} is proved.
\subsection{Proof of Lemma \ref{h4}}\label{proofofh4}
The proof is similar to the proof of Lemma \ref{h3}. We shall consider the case $\mathbb A=\emptyset$ only.
First we represent
\begin{equation}
 \gamma_{lj}^{(t)}=\widetilde \gamma_{jl}^{(t)}+\widehat\gamma_{jl}^{(t)},
\end{equation}
where
\begin{align}
 \widetilde\gamma_{lj}^{(t)}&=\frac1n X_{lj}X_{lt}R^{(l)}_{jt},\notag\\
 \widehat\gamma_{lj}^{(t)}&=\frac1nX_{lj}\sum_{q\in\mathbb T_{l,j,t}}X_{lq}R^{(l)}_{qt}R^{(l)}_{jt}(R^{(l)}_{tt})^{-1}.
\end{align}
Note that
\begin{align}
 \E^{\frac 1m}\{(\widetilde\gamma_{lj}^{(t)})^m\mathbb I\{\mathcal U\}\big|\mathfrak M^{(j,t)}\}\le \frac{C\lna^{\frac2{\varkappa}}}{n}
 \E^{\frac 1m}\{|R^{(l)}_{jt}|^m\mathbb I\{\mathcal U\}\big|\mathfrak M^{(j,t)}\}
\end{align}
Applying the representation \eqref{inter} and McDiarmid's inequality, we obtain
\begin{align}
 \E^{\frac 1m}\{(\widetilde\gamma_{lj}^{(t)})^m\mathbb I\{\mathcal U\}\big|\mathfrak M^{(j,t)}\}\le \frac{C\lna^{\frac2{\varkappa}}}{n\sqrt n}
 \E^{\frac 1m}\{(\sum_{q\in\mathbb T_{j,l,t}}|R^{(l,j)}_{qt}|^2)^{\frac m2}\mathbb I\{\mathcal U\}\big|\mathfrak M^{(j,t)}\}+
 \Phi_n.
\end{align}
Applying Lemma \ref{h1} now, we get
\begin{equation}
 \E^{\frac 1m}\{(\widetilde\gamma_{lj}^{(t)})^m\mathbb I\{\mathcal U\}\big|\mathfrak M^{(j,t)}\}\le\frac{C\sqrt{\kappa_n}}{n\sqrt {nv}}
 +\Phi_n.
\end{equation}
Taking in account that $\psi_0\ge c\sqrt v$, we may write
\begin{equation}
 \E^{\frac 1m}\{(\widetilde\gamma_{lj}^{(t)})^m\mathbb I\{\mathcal U\}\big|\mathfrak M^{(j,t)}\}\le\frac{C\kappa_n^{\frac32}}{nv\sqrt {nv}}
 +\Phi_n.
\end{equation}
Re-expanding $\widehat \gamma_{lj}^{(t)}$ in $j$th row, we obtain for $\omega\in\mathcal B^{(l)}\cap\mathcal B$,
\begin{align}
 |\widehat \gamma_{lj}^{(t)}|\le\frac{C\lna^{\frac1{\varkappa}}}{\sqrt n}(|R_{lt}|+\frac1{\sqrt n}|X_{lt}|+\frac1{\sqrt n}|X_{lj}||R_{jt}^{(l)}|)|R^{(l)}_{jt}|.
\end{align}
The last inequality yields
\begin{align}
  \E^{\frac1m}\{|\widehat \gamma_{lj}^{(t)}|^m\mathbb I\{\mathcal U\}\big|\mathfrak M^{(j,t)}\}\le &\frac{C\lna^{\frac1{\varkappa}}}{n\sqrt n}
 \E^{\frac1{2m}}\{(\sum_{q\in\mathbb T_{l}}|R^{(t)}_{ql}|^2)^m\mathbb I\{\mathcal U\}\big|\mathfrak M^{(j,t)}\}\notag\\&\times
 \E^{\frac12}\{(\sum_{q\in\mathbb T_{l}}|R^{(j)}_{qt}|^2)^m\mathbb I\{\mathcal U\}\big|\mathfrak M^{(j,t)}\}
 +\Phi_n.
 \end{align}
Applying Lemmas \ref{h1} and \ref{h2}, we get
\begin{equation}
 \E^{\frac1m}\{|\widehat \gamma_{lj}^{(t)}|^m\mathbb I\{\mathcal U\}\big|\mathfrak M^{(j,t)}\}\le \frac{C\sqrt{\kappa_n}}{n\sqrt {nv}}
 (\sum_{q\in\mathbb T_{j,l,t}}|R_{ql}^{(j,t)}|^2)^{\frac12}\mathbb I\{\mathcal U^{(j,t)}\}+\Phi_n.
\end{equation}
Thus the Lemma is proved.
\vskip 0.05cm
{\bf Acknowledgement.} The authors would like to thank S. G.  Bobkov for
drawing their attention to some references about large deviations for martingales and quadratic forms.

\end{document}